\documentclass[11pt]{amsart}
\usepackage{amsmath,amsfonts,amssymb,amsthm}
\usepackage{amsmath,amsthm,indentfirst}
\usepackage{amssymb}
\usepackage{amsfonts}
\usepackage{bbm, dsfont}
\usepackage{mathtools}
\usepackage{color}
\usepackage{delarray}
\usepackage{amscd}
\usepackage[latin1]{inputenc}
\DeclareMathAlphabet{\mathpzc}{OT1}{pzc}{m}{it}

\DeclarePairedDelimiter{\norm}{\|}{\|}
\usepackage{hyperref}

\theoremstyle{plain}
\newtheorem*{maintheorem*}{Main Theorem}
\newtheorem*{thm*}{Theorem}
\newtheorem*{thma*}{Theorem A}
\newtheorem*{thmaa*}{Theorem A'}
\newtheorem*{thmb*}{Theorem B}
\newtheorem*{thmo*}{Theorem 1.1}
\newtheorem*{thmc*}{Theorem C}
\newtheorem*{thmd*}{Theorem D}
\newtheorem*{thmf*}{Theorem 4.1}
\newtheorem*{remark*}{Remark}
\newtheorem*{conjecture*}{Conjecture}
\newtheorem*{prop*}{Proposition}
\newtheorem*{lem*}{Lemma}


\makeindex

\def\bbz{\mathbb{Z}}
\def\bbq{\mathbb{Q}}

\def\bbr{\mathbb{R}}

\def\bbc{\mathbb{C}}

\def\hcal{\mathcal{H}}

\def\Rfrak{\mathfrak{R}}
\def\Ifrak{\mathfrak{I}}

\newcommand{\be}{\begin{equation}}
\newcommand{\ee}{\end{equation}}
\newcommand{\Ext}{\operatorname{Ext}}

\def\rhr{H^1(M,\Sigma,\bbr)}
\def\hr{H^1(M,\bbr)}

\def\strat1{\hcal_1(\alpha)}

\def\supp{{\rm{supp}}}

\def\setminus{-}

\newcommand{\bb}{\mathbb}
\newcommand{\gothic}{\mathfrak}

\newcommand{\cx}{{\bb C}}

\newcommand{\integers}{{\bb Z}}

\newcommand{\reals}{{\bb R}}

\newlength{\figboxwidth}             
\renewcommand{\bold}[1]{\medskip \noindent {\bf #1 }\nopagebreak}

\newcommand{\cross}{\times}

\newcommand{\gammabar}{{\overline{\gamma}}}

\newcommand{\zed}{\integers}

\newcommand{\SL}{\operatorname{SL}}

\newcommand{\dif}{\operatorname{d}}
\newcommand{\AGY}{\operatorname{AGY}}

\def\@ifundefined#1#2#3%
  {\expandafter\ifx\csname#1\endcsname\relax#2\else#3\fi}

\@ifundefined{theoremstyle}{
}{
\theoremstyle{plain} 
}
\newtheorem{theorem}{Theorem}[section]
\newtheorem{prop}[theorem]{Proposition}

\newtheorem{lemma}[theorem]{Lemma}
\newtheorem{cor}[theorem]{Corollary}
\newtheorem{corollary}[theorem]{Corollary}
\newtheorem{claim}[theorem]{Claim}

\newtheorem{theo}{Theorem}[section]

\newtheorem{coro}[theo]{Corollary}

\@ifundefined{theoremstyle}{
}{
\theoremstyle{definition} 
}
\newtheorem{definition}[theorem]{Definition}

\newtheorem{remark}[theorem]{Remark}

\newcommand{\cB}{{\mathcal B}}

\newcommand{\cF}{{\mathcal F}}

\newcommand{\cH}{{\mathcal H}}

\newcommand{\cS}{{\mathcal S}}

\newcommand{\cU}{{\mathcal U}}
\newcommand{\cV}{{\mathcal V}}
\newcommand{\cW}{{\mathcal W}}

\mathchardef\GG="321D
\newcommand{\gp}{{\gothic p}}

\newcommand{\noz}{k}

\newcounter{consta}
\renewcommand{\theconsta}{{\kappa_{\arabic{consta}}}}
\newcounter{constb}[section]

\newcounter{constc}[section]

\newcounter{constA}
\renewcommand{\theconstA}{{N_{\arabic{constA}}}}
\newcommand{\consta}{\refstepcounter{consta}\theconsta}

\newcommand{\constA}{\refstepcounter{constA}\theconstA}

\newcommand{\Teich}{\mathcal T}
\newcommand{\QMS}{\mathcal{Q}_1(S)}
\newcommand{\Qalpha}{\mathcal{Q}_1(1,\ldots,1)}
\newcommand{\Halpha}{\mathcal H_1(\alpha)}
\newcommand{\Amnfld}{\mathcal Q_1(\alpha)}
\newcommand{\tAmn}{\mathcal Q^1\mathcal T(\alpha)}

\newcommand{\ipt}{k}
\newcommand{\ML}{\mathcal{ML}(S)}
\newcommand{\PML}{{\rm P}\ML}
\newcommand{\MC}{\operatorname{Mod}_g}
\newcommand{\Sob}{\mathcal C^1}
\newcommand{\intsc}{{i}}
\newcommand{\ncc}{\mathcal N_{\rm nc}}
\newcommand{\nlifts}{\mathcal N}

\newcommand{\tW}{\tilde{W}}

\newcommand\tpbox{{{\mathsf B}}}

\newcommand\Hgood{\mathsf H^{\unst}_{t,{\rm int}}}
\newcommand\ER{\mathsf{E}}
\newcommand{\ET}{\mathsf{ET}}
\newcommand{\Net}{{\mathsf {Net}}}
\newcommand{\tNet}{\tilde{\mathsf {Net}}}
\newcommand{\mce}{{\bf g}}
\newcommand{\ttc}{{U(\tau)}}
\newcommand{\ttco}{{{P}(\tau)}}
\newcommand{\ttcod}{{P}_{\geq\delta}(\tau)}
\newcommand{\ttcn}{{\mathsf c}_{{\bf a}}}
\newcommand{\Nip}{{\mathcal N}}
\newcommand{\Ip}{{\mathcal O}_\tau}
\newcommand{\muth}{\mu_{{\rm Th}}}
\newcommand{\mlhoro}{{W^{\unst}}}
\newcommand{\tmlhoro}{{\tilde W^{\unst}}}
\newcommand{\mlsign}{{}}
\newcommand{\PU}{\tilde W^{\rm cs}}

\usepackage[frame,ps,dvips,matrix,arrow,curve,rotate]{xy}

\newcommand{\hubmas}{{\tilde{\mathcal P}_1}}
\newcommand{\bhubmas}{{\pi\circ {\tilde{\mathcal P}_1}^{-1}}}
\newcommand{\lipc}{L}
\newcommand{\cube}{\mathcal U}
\newcommand{\unst}{{\rm u}}
\newcommand{\stbl}{{\rm s}}

\newcommand{\vM}{X}

\title[Effective counting of simple closed geodesics]{Effective counting of simple closed geodesics on hyperbolic surfaces}

\author{Alex Eskin}
\thanks{A.E.\ acknowledges support by the NSF and the Simons Foundation}
\address{A.E.\ Department of Mathematics, University of Chicago, Chicago, IL 60637}
\email{eskin@math.uchicago.edu}

\author{Maryam Mirzakhani}

\author{Amir Mohammadi}
\thanks{A.M.\ acknowledges support by the NSF and Alfred P.\ Sloan Research Fellowship}
\address{A.M.\ Mathematics Department, UC San Diego, 9500 Gilman Dr, La Jolla, CA 92093}
\email{ammohammadi@ucsd.edu}

\date{}

\begin{document}
\maketitle
\begin{abstract}
We prove a quantitative estimate, with a power saving error term, 
for the number of simple closed geodesics of length at most $L$ on 
a compact surface equipped with a Riemannian metric of negative curvature. 
The proof relies on the exponential mixing rate for the Teichm\"{u}ller geodesic flow.   
\end{abstract}

\section{Introduction}
\label{sec:intro}
Let $g \ge 2$,\index{g@$g$ the genus of $S$} and let $S$\index{S@$S$ a compact surface of genus $g$} be a compact Riemann surface of genus $g$.
Let $\Teich(S)$\index{T@$\Teich(S)$ the Teichm\"{u}ller space of $S$} be the Teichm\"{u}ller space 
of complete hyperbolic metrics on $S$, and let 
\[
\mathcal M(S)\index{M@$\mathcal M(S)$ the moduli space of $S$}=\Teich(S)/\MC
\] 
be the corresponding moduli space,
where $\MC$\index{M@$\MC$ the mapping class group of $S$} is the mapping class group of $S$. 

Let $M\in \mathcal M(S)$\index{M@$M$ a hyperbolic structure on $S$}. 
Problems related to the asymptotic growth rate of the number of closed geodesics on $M$ 
have been long studied. 
In particular, thanks to works of Delsart, Huber, and Selberg we have the following:
There exists some $\delta=\delta(M)>0$ so that the number of closed geodesics of length at most $L$ on $M$ equals 
\be\label{eq:closed-geod}
\operatorname{Li}(e^L) + O_M(e^{L-\delta}),
\ee
where $\operatorname{Li}(x)=\int_2^x\frac{dt}{\log t}$; see~\cite{Bus} and references there.

More generally, the growth rate of the number of closed
geodesics on a negatively curved compact manifold was studied by Margulis,~\cite{Margulis-Thesis}.  
His proof, which is different from the above mentioned works, 
is based on the mixing property of the Margulis measure for the geodesic flow. 
In the constant negative curvature case, Margulis' method combined with an exponential mixing rate for the geodesic flow, 
also provides an estimate like~\eqref{eq:closed-geod} --- albeit with a weaker power saving $\delta$, see e.g.~\cite{MMO-Geod}.    

\subsection{Simple closed geodesics}
The aforementioned fundamental results do not provide any estimates for the number of simple closed geodesics on $M$. 
Indeed, very few closed geodesics on $M$ are simple,~\cite{BS2}, and it is hard to
discern them in $\pi_1(M)$,~\cite{BS1}. More explicitly, it was shown in~\cite{Ri} that 
the number of {\em simple} closed geodesics of length at most 
$L$ on $M$ is bounded above and below by $O_M(L^{6g-6})$.

In her PhD thesis,~\cite{M-Thesis} and~\cite{M-SimpleCl}, Mirzakhani proved an asymptotic growth rate for the number of simple closed geodesics 
of a given topological type on a hyperbolic surface $M$ --- recall that two simple closed geodesics $\gamma$\index{g@$\gamma$ a rational multicurve} and $\gamma'$ on $M$ are of the same topological type if there exists some $\mce\in\MC$\index{g@$\mce$ a mapping class} so that $\gamma'=\mce\gamma$. 

Let $\vM$\index{X@$\vM$ a Riemannian metric of negative curvature on $S$} be a compact surface equipped with a Riemannian metric of negative curvature. 
We emphasize that the curvature is not assumed to be constant; indeed, elements in $\mathcal M(S)$ 
will be denoted by $M$ to minimize the confusion. 
By a multi-geodesic $\gamma$ on $\vM$ we mean 
$\gamma=\sum_{i=1}^d a_i\gamma_i$
where $\gamma_i$'s are disjoint, essential, simple closed geodesics, and $a_i >0$ for all $1\leq i\leq d$. 
In this case, we define $\ell_\vM(\gamma)\index{l@$\ell_\vM(\gamma)$ the length of $\gamma$ in metric $X$}:=\sum a_i\ell_\vM(\gamma)$, where $\ell_\vM$ denotes the length function on $\vM$.
The multi-geodesic $\gamma$ will be called integral (resp.\ rational) if $a_i\in\bbz$ (resp.\ $a_i\in\bbq$).

Given a rational multi-geodesic $\gamma_0$ on $\vM$, define 
\[
s_\vM(\gamma_0, L)\index{s@$s_\vM(\gamma_0, L)$ number of simple closed curves}:=\#\{\gamma\in\MC.\gamma_0:\ell_\vM(\gamma)\leq L\}.
\]
Mirzakhani,~\cite[Thm.~1.1]{M-SimpleCl}, proved the following estimate when $M$ is a hyperbolic surface: 
\be\label{eq:simple-closed-geod}
s_M(\gamma_0, L)\sim n_{\gamma_0}(M)L^{6g-6},
\ee 
where $n_{\gamma_0}\index{n@$n_{\gamma_0}$ the Mirzakhani function}:\mathcal M(S)\to \bbr^+$ (the {\em Mirzakhani} function) is a continuous proper function; 
geometric informations carried by $n_{\gamma_0}$ are also studied in~\cite{M-SimpleCl}. 

In this paper we obtain a quatitative version of~\eqref{eq:simple-closed-geod}; 
moreover, our approach allows us to prove such a result in the more general setting of {\em variable} negative curvature.

\begin{theorem}\label{thm:main}
There exists some $\kappa=\kappa(g)>0$ so that the following holds.
Let $\vM$ be a compact surface of genus $g$ equipped with a Riemannian metric of negative curvature.   
Let $\gamma_0$ be a rational multi-geodesic on $\vM$. Then 
\[
s_\vM(\gamma_0,L)=n_{\gamma_0}(\vM)L^{6g-6}+O_{\gamma_0, \vM}(L^{6g-6-\kappa})
\]
where $n_{\gamma_0}(\vM)$ is a positive constant which depends on $\gamma_0$ and $\vM$. 
\end{theorem}

The proof of Theorem~\ref{thm:main} is based on the study of a related counting problem in the space of geodesic measured laminations on $S$, \`{a} la Mirzakhani.
The space of measured laminations on $S$, which we denote by $\mathcal{ML}(S)$\index{M@$\mathcal{ML}(S)$ the space of measured laminations on $S$}, is a piecewise linear integral manifold 
homeomorphic to ${\Bbb R}^{6g-6}$; but it does not have a natural differentiable structure,~\cite{Thurston:book:GTTM}. 
Train tracks were introduced by Thurston as a powerful technical device for understanding measured 
laminations. Roughly speaking, train tracks are induced by squeezing almost parallel
strands of a very long simple closed geodesic to simple arcs on a surface; they provide linear charts for $\ML$.

The mapping class group $\MC$ of $S$ acts naturally on $\mathcal{ML}(S)$. Moreover,
there is a natural $\MC$-invariant locally finite measure on $\mathcal{ML}(S)$, the
Thurston measure $\mu_{\rm Th}$\index{m@$\mu_{\rm Th}$ the Thurston measure}, given by the piecewise linear integral structure on $\mathcal{ML}(S)$,~\cite{Thurston:book:GTTM}.
For any open subset $U \subset\mathcal{ML}(S)$ and any $t>0$, we have
\[
\mu_{\rm Th}(tU)=t^{6g-6}\mu_{\rm Th}(U).
\]
On the other hand, any metric of negative curvature $\vM$ on $S$ induces the length function
$\lambda\mapsto\ell_\vM(\lambda)$ on $\mathcal{ML}(S)$, which satisfies $\ell_\vM(t\lambda)=t\ell_\vM(\lambda)$
for all $t>0$. It is proved in~\cite[App.~A]{M-Thesis} that $\ell_M$ is a convex function on $\mathcal{ML}(S)$ when $M$ is a hyperbolic surface. This fact remains valid in the more general setting of variable negative curvature, see~\S\ref{sec:add-ttc}. 

The source of the polynomially effective error term in Theorem~\ref{thm:main} 
is the exponential mixing property of the Teichm\"{u}ller geodesic flow proved by Avila, Gou\"{e}zel, and Yoccoz,~\cite{AGY-EM, AJ-EMQ, AG-Eigen}. 
We combine this estimate with ideas developed by Margulis in his PhD thesis,~\cite{Margulis-Thesis}, to prove the following theorem which is of independent interest --- see Theorem~\ref{prop:counting} for a more general statement. 

Let $\tau$ be a train track and let $U(\tau)$ be the corresponding train track chart.
For every $\lambda\in U(\tau)$ we let $\|\lambda\|_\tau$ denote the sum of the weights of $\lambda$ in $U(\tau)$, see~\S\ref{sec:tt}.

\begin{theorem}\label{thm:main-MLS}
There exists some $\consta\label{a:final-exp-mls}=\ref{a:final-exp-mls}(g)>0$ so the following holds. 
Let $\tau$ be a maximal train track. Let $L\geq1$ and let $\gamma_0$ be a simple close curve on $M$.
There exists a constant $c_{\gamma_0}>0$ so that  
\[
\#\{\gamma\in U(\tau)\cap\MC.\gamma_0: \|\gamma\|_\tau\leq L\}=c_{\gamma_0}{\rm vol}_\tau{L^{6g-6}}+ O_{\tau,\gamma_0}(L^{6g-6-\ref{a:final-exp-mls}})
\] 
where ${\rm vol}_\tau=\mu_{\rm Th}\{\lambda\in U(\tau):\|\lambda\|_\tau\leq 1\}$.
\end{theorem}

It is worth noting that in view of Theorem~\ref{thm:main-MLS}, the asymptotic behavior of the number of points in one $\MC$-orbit in 
the cone $\{\lambda:\|\lambda\|_\tau\leq L\}$ and that of the number of integral points in this cone agree up to multiplicative constant.

Theorem~\ref{thm:main-MLS}, in the more general form Theorem~\ref{prop:counting}, plays a crucial role in our analysis. Indeed, 
using the aforementioned convexity of the length function, we will prove Theorem~\ref{thm:main} using Theorem~\ref{prop:counting} in \S\ref{sec:proof-thm}.

It is an intriguing problem to investigate the asymptotic behavior of functions similar to and different from $s_\vM(\gamma_0,L)$ or the  complexity considered in Theorem~\ref{thm:main-MLS}. 
For instance, for a suitable formulation of a combinatorial length --- using intersection numbers --- the count is exactly a polynomial, see~\cite{FLP}. We also refer the reader to~\cite{CMP} where a related problem is studied for thrice punctured sphere.

\subsection{Outline of the paper}
In \S\ref{sec:notation} we collect some preliminary results. 
In \S\ref{sec:exp-mix} we prove an equidistribution result with an error term, Proposition~\ref{prop:thickening-smooth}, 
which may be of independent interest; see, e.g.~\cite{KM, LM-Horospherical}. 
The proof of this proposition is based on the exponential mixing rate for the Teichm\"{u}ller geodesic flow,~\cite{AGY-EM}, 
and the so called {\em thickening} technique, see~\cite{Margulis-Thesis, EMc}. 
In \S\ref{sec:mxing-counting} we prove Proposition~\ref{prop:ncc}; this proposition is one of the main ingredients in the proof, 
and could be compared to arguments in~\cite[Chap.~6]{Margulis-Thesis}. We will recall some basic facts about $\ML$, and study the relation between the linear structures on $\ML$ and the space of quadratic differentials in \S\ref{sec:tt} and \S\ref{sec:linear-ml-qt}.
The orbital counting in sectors of $\ML$ is studied in \S\ref{sec:int-pts}; the main result here is Theorem~\ref{prop:counting}.
We prove Theorem~\ref{thm:main} in \S\ref{sec:proof-thm}.

\subsection{Acknowledgement}
This project originated in fall of 2015 when the authors were members of the Institute for Advanced Study (IAS), we 
thank the IAS for its hospitality.  We thank C.~McMullen, K.~Rafi, and A.~Zorich for helpful discussions. 
We also thank F.~Arana-Herrera, H.~Oh, and A.~Wright for their comments on an earlier version of this paper.
We are in debt to G.~Margulis and F.~Arana-Herrera for drawing our attention to the case of variable negative curvature, and  
to K.~Rafi for providing the proof of Theorem~\ref{thm:int-convex}. 
Last, but not least, we thank the anonymous referee for their careful reading and several helpful comments.

\section{Preliminaries and notation}\label{sec:notation}
Let $\mathcal {Q}(S)$\index{Q@$\mathcal Q(S)$ the moduli space of quadratic differentials} denote the moduli space of quadratic differentials 
on $S$, and let $\QMS$\index{Q@$\QMS$ the moduli space of area one quadratic differentials} be the moduli space of quadratic differentials with area one on $S$. 
For any $\alpha = (\alpha_1,\dots, \alpha_\noz,\varsigma)$\index{a@$\alpha$ multiplicities of zeros} with $\sum\alpha_i=4g-4$
and $\varsigma\in\{\pm1\}$, define $\mathcal Q_1(\alpha)$\index{Q@$\mathcal Q_1(\alpha)$ a stratum of area one quadratic differentials} to be (a connected component) of the stratum of quadratic differentials 
consisting of pairs $(M,q)$ where $M\in\mathcal M(S)$ 
and $q$ is a unit area quadratic differential on $M$ 
whose zeros have multiplicities $\alpha_1,\dots, \alpha_\noz$
and $\varsigma=1$ if $q$ is the quare of an abelian differential and $-1$ otherwise. 
Then
\[
\QMS=\bigsqcup_\alpha\mathcal Q_1(\alpha).
\]
Put $\mathcal Q(\alpha)\index{Q@$\mathcal Q(\alpha)$ a stratum of quadratic differentials}:=\{tq: t\in \reals, q\in\mathcal Q_1(\alpha)\}$. 
Let $\Sigma\subset S$ be a set of $k$ distinct marked points. 
Let $\tAmn\index{Q@$\tAmn$ the space of marked surfaces}$ denote the space of 
quadratic differentials $(M, q)$ equipped with an equivalence class
of homeomorphisms $f: S\to M$ that send the marked points to the zeros of $q$. The equivalence relation is isotopy rel marked points.
Let $\pi\index{p@$\pi$ the covering map}:\tAmn\to\mathcal Q_1(\alpha)$ be the forgetful map which forgets the marking $f$; this is an infinite degree branched covering.  

Similarly, let $\Omega(S)$\index{H@$\Omega(S)$ the moduli space of Abelian differentials} denote the moduli space of Abelian differentials on $S$, and let $\Omega_1(S)$\index{H@$\Omega_1(S)$ the moduli space of area one Abelian differentials} be the moduli space of area one Abelian differentials.
For any $\alpha=(\alpha_1,\ldots,\alpha_\noz)$, we let $\mathcal H(\alpha)$\index{H@$\mathcal H(\alpha)$ a stratum of Abelian differentials} denote the corresponding stratum,
and let $\Halpha$\index{H@$\mathcal H(\alpha)$ a stratum of area one Abelian differentials} denote the area one abelian differentials. 

Note that passing to a branched double cover $\hat M$\index{M@$\hat M$ orienting double cover of $M$} of $M$,
we may realize $\mathcal Q_1(\alpha)$ as an {\em affine invariant submanifold} in $\mathcal H_1(\hat\alpha)$\index{a@$\hat \alpha$ orienting abelian differential}
corresponding to odd cohomology classes on $\hat M$, see \S\ref{sec:period-piece-ab}.
However, even if $q$ belongs to a compact subset of $\QMS$, 
the complex structure on $\hat M$ may have very short closed curves in the hyperbolic metric, e.g.\ a short saddle connection between two distinct zeros on $(M,q)$ could lift to a short loop in $\hat M$.
Note however that if $(\hat M,\omega)$ is the aforementioned double cover of $(M,q)$, then 
the length of the shortest saddle connection in $\omega$ is bounded by the length of the shortest saddle connection in $q$, i.e.\
compact subsets of $\mathcal Q_1(\alpha)$ lift to compact subsets of $\mathcal H_1(\hat\alpha)$.

\subsection{Period coordinates}\label{sec:period-piece-ab}
Let $x=(M,\omega)\in \mathcal H(\alpha)$, and let $\Sigma\subset M$ be the set of
zeros of $\omega$. Passing to a finite cover, which we continue to denote by $\mathcal H(\alpha)$, 
we assume there are no orbifold points in $\mathcal H(\alpha)$. Define the period map
\[
\Phi\index{F@$\Phi$ the period map}:\mathcal H(\alpha)\to H^1(M,\Sigma,\cx).
\] 
Let us recall that $\Phi$ can be defined as follows. 
Let $\#\Sigma=k$. 
Fix a triangulation $T$ of the surface by saddle connections of $x$, that is: $2g+\noz-1$ 
directed edges $\delta_1,\ldots,\delta_{2g+\noz-1}$ which form a basis for $H_1(M,\Sigma,\zed)$. 
Define
\[
\Phi(x)=\Bigl(\int_{\delta_i}\omega\Bigr)_{i=1}^{2g+\noz-1}.
\]
Note that this map depends on the triangulation $T$. If $T'$
is any other triangulation, and $\Phi'$ is the corresponding period map,
then $\Phi'\circ\Phi^{-1}$ is linear.
For any $x\in\mathcal H(\alpha)$, there is a neighborhood $\mathsf B(x)$ of $x$ 
so that the restriction of $\Phi$ to $\mathsf B(x)$ is a homeomorphism onto $\Phi(\mathsf B(x))$,
see \S\ref{sec:period-box}.
We always choose $\mathsf B(x)$ small enough so that, using the Gauss-Manin connection, 
the triangulation at $y\in\mathsf B(x)$ can be identified with the triangulation at $x$. 

We define the period coordinates at $x=(M,q)\in\mathcal Q(\alpha)$ as follows.
If $\varsigma=1$, then $q$ is a square of an abelian differential, and we may define period coordinates as above.  
If $\varsigma=-1$, we use the orienting double cover $\mathcal H(\hat\alpha)$ to define the period coordinates:
in this case there is a canonical injection from $\mathcal Q(\alpha)$ into $\mathcal H(\hat\alpha)$. 
Any Riemann surface in the image of this map is equipped with an involution. This way we get the period map from $\mathcal Q(\alpha)$
to $H_{\rm odd}^1(M,\Sigma,\cx)$ --- the anti-invariant subspace of the cohomology for the involution.

Put $h\index{h@$h$ the topological entropy of Teichm\"{u}ller geodesic flow}:=2g+\noz-2$ if $\varsigma=1$ 
and $h:=2g+\noz-3$ if $\varsigma=-1$; 
the number $h$ is the topological entropy of the Teichm\"{u}ller geodesic flow on $\mathcal Q_1(\alpha)$.



\subsection{$\SL(2,\reals)$-action on $\mathcal H_1(\alpha)$}\label{sec:sl2r-action}
Let $x\in\mathcal H_1(\alpha)$,
we write $\Phi(x)$ as a $2 \cross n$ matrix. The action of $g =
\begin{pmatrix} a & b \\ c & d \end{pmatrix} \in
\SL(2,\reals)$ in these coordinates is linear. 
We choose a fundamental domain for the action of 
the mapping class group and
think of the dynamics on the fundamental domain. Then, the
$\SL(2,\reals)$-action becomes
\begin{displaymath}
\begin{pmatrix} x_1 & \dots & x_n \\ y_1 & \dots & y_n \end{pmatrix}
\mapsto \begin{pmatrix} a & b \\ c & d \end{pmatrix} \begin{pmatrix} x_1 & \dots & x_n \\ y_1 & \dots & y_n
\end{pmatrix} A(g,x)\index{a@$A(g,x)$ the Kontsevich-Zorich cocycle},
\end{displaymath}
where $A(g,x) \in \operatorname{Sp}(2g,\zed) \ltimes \zed^{\noz-1}$
is the {\em Kontsevich-Zorich 
cocycle}. That is: $A(g,x)$ is the change of basis one needs to perform to return the
point $gx$ to the fundamental domain. It can be interpreted as the
monodromy of the Gauss-Manin connection restricted to the orbit of
$\SL(2,\reals)$.  

In the sequel, we let $a_t = \begin{pmatrix} e^t & 0 \\ 0 & e^{-t}  \end{pmatrix}$,
$u_t = \begin{pmatrix} 1 & t \\ 0 & 1  \end{pmatrix}$, and $\bar u_t = \begin{pmatrix} 1 & 0 \\ t & 1  \end{pmatrix}$.

We have the following.

\begin{theorem}[Veech-Masur]\label{thm:VM}
The space $\mathcal H_1(\alpha)$ carries a natural measure $\mu$\index{m@$\mu$ affine $\SL(2,\reals)$-invariant measure}
in the Lebesgue measure class such that
\begin{enumerate}
\item $\mathcal H_1(\alpha)$ has finite measure,
\item $\mu$ is $\SL(2,\reals)$-invariant and ergodic.
\end{enumerate}
\end{theorem}

More generally, for any affine invariant manifold, $\mathcal M\index{M@$\mathcal M$ affine invariant manifold}\subset\Halpha$,
we let $\mu$ denote the $\SL(2,\reals)$-invariant affine measure on $\mathcal M$.
In particular, all the strata in $\mathcal Q_1(S)$ are equipped with such invariant measures.

\subsection{Mapping class group action}\label{sec:mpc-action}
We denote elements in $\MC$ using bold letters, e.g., $\mce$ denotes an element in $\MC$. 
The action of $\MC$ on $\tAmn$ commutes with the action of $\SL(2,\reals)$, we will however denote both these actions as left action and write, e.g.\  $\mce\cdot\tilde x$ or simply denoted by $\mce\tilde x$.

\subsection{The constants}\label{sec:constants}
In the sequel we will use $\kappa_\bullet$ and $N_\bullet$, $\bullet=1,2,\ldots$ to denote various constants. 
Unless it is explicitly mentioned otherwise, these constants are allowed only to depend on the genus. 
The constants $\kappa_\bullet$ are meant to indicate small positive numbers
while $N_\bullet$ are used for constants which are expected to be $>1$. 

We will also use the notation $A \ll B$. This expression means: there exists a constant $c>0$ so that $A \leq cB$; the implicit constant $c$ is permitted to depend on the genus, but (unless otherwise noted) not on anything else. We write~$A\asymp B$ if~$A\ll B\ll A$.
If a constant (implicit or explicit) depends on another parameter others than the genus, 
we will make this clear by writing, e.g.~$\ll_\epsilon$,~$C(x)$, etc.

We also adopt the following $\star$-notation. 
We write $B=A^{\pm\star}$ if $B=A^{\pm c}$ where $c>0$ depends only on the genus.
Similarly, one defines $B\ll A^\star$, $B\gg A^\star$.
Finally, we also write~$A\asymp B^\star$ if~$A^\star\ll B\ll A^\star$ (possibly with different exponents).

\subsection{Modified Hodge norm}\label{sec:hodge-norm}

Let $M$ be a Riemann surface. By definition, $M$ has a complex structure. 
Let $\cH_M$ denote the set of holomorphic $1$-forms on
$M$. One can define the {\em Hodge inner product} on $\cH_M$ by
\begin{displaymath}
\langle \omega, \eta \rangle = \frac{{\bf i}}{2} \int_M \omega \wedge \bar{\eta}.
\end{displaymath}
We have a natural map $r: H^1(M,\reals) \to \cH_M$ which sends
a cohomology class $c \in H^1(M,\reals)$ to the 
holomorphic $1$-form $r( c) \in \cH_M$ such that 
the real part of $r( c)$ (which is a harmonic $1$-form)
represents $c$. We can thus define the Hodge inner product on
$H^1(M,\reals)$ by $\langle c_1, c_2 \rangle = \langle
r(c_1), r(c_2) \rangle$. Then
\begin{displaymath}
\langle c_1, c_2 \rangle = \int_M c_1 \wedge *c_2,
\end{displaymath}
where $*$ denotes the Hodge star operator and we choose harmonic representatives 
of $c_1$ and $*c_2$ to evaluate the integral.
We denote the associated norm by $\| \cdot \|_M$. 
This is the {\em Hodge norm}, see \cite{FarkasKra}.

If $x = (M,\omega) \in \cH_1(\alpha)$, we will often write $\| \cdot
\|_{{\rm H},x}$\index{n@$\norm{\;}_{{\rm H},x}$ the Hodge norm at $x$} to denote the Hodge norm $\| \cdot \|_M$ on
$\hr$. Since $\| \cdot \|_{{\rm H},x}$ depends only on $M$, we have 
$\|c\|_{{\rm H},kx} = \|c\|_{{\rm H},x}$ for all $c \in \hr$ and all $k\in{\rm SO}(2)$.

Let $E(x)=\mbox{span}\{[{\rm Re}(\omega)],[{\rm Im}(\omega)]\}$ --- the space $E(x)$ is often referred to as the {\em standard space}. 
We let 
\be\label{eq:the map p}
p: \rhr\rightarrow\hr
\ee 
denote the natural projection; $p$ defines an isomorphism between $E(x)$ and 
$p(E(x))\subset\hr$. 

For our applications in the sequel (and in order to account for the {\em loss} of hyperbolicity in the thin part of the moduli space)
we need to consider a modification of the Hodge norm.

\subsection*{The classes $c_\alpha$ and $\ast c_\alpha$}
Let $\alpha$ be a homology class in $H_1(M,\reals)$. We let $*c_{\alpha} \in H^1(M,\reals)$ be the
cohomology class so that  
\begin{displaymath}
\int_\alpha \omega = \int_M \omega \wedge *c_\alpha 
\end{displaymath}
for all $\omega \in H^1(M, \reals)$. Then, 
\begin{displaymath}
\int_M *c_\alpha \wedge *c_\beta = \intsc(\alpha,\beta)\index{i@$\intsc(\,,\,)$ algebraic intersection pairing},
\end{displaymath}
where $\intsc(\cdot, \cdot)$ denotes the algebraic intersection number.
Let $\ast\index{s@$\ast$ the Hodge star operator}$ denote the Hodge star operator, and let 
\begin{displaymath}
c_\alpha = \ast^{-1}(*c_\alpha). 
\end{displaymath}
Then, for any $\omega \in H^1(M,\reals)$ we have
\begin{displaymath}
\langle \omega, c_\alpha \rangle = \int_M \omega \wedge *c_\alpha = 
\int_\alpha \omega,
\end{displaymath}
where $\langle \cdot, \cdot \rangle$ is the Hodge inner product. 
We note that $*c_\alpha$ is a purely topological construction which 
depends only on $\alpha$, but $c_\alpha$
depends also on the complex structure of $M$.

Fix $\epsilon_* > 0$ (the \emph{Margulis constant}) so that any two
geodesics of hyperbolic length less than $\epsilon_*$ must be disjoint. 

{}{
Let $\sigma$ denote the hyperbolic metric in the conformal
class of $M$. For any closed curve $\alpha$ on $M$, let $\ell_M(\alpha)$
denote the length of the geodesic representative of 
$\alpha$ in the metric $\sigma$.} 

We recall the following.
\begin{theorem}\cite[Thm.~3.1]{ABEM}
\label{theorem:hodge:hyperbolic}
For any constant $L > 1$ there exists a constant $c > 1$, 
such that for any simple closed curve $\alpha$ with $\ell_M(\alpha) < L$, we have
\begin{equation}
\label{eq:hodge:hyperbolic}
\frac{1}{c} \ell_M(\alpha)^{1/2} \le \| c_\alpha \|_M <
c \,\ell_M(\alpha)^{1/2}.
\end{equation}
Furthermore, if $\ell_M(\alpha) < \epsilon_*$ 
and $\beta$ is the
shortest simple closed curve crossing $\alpha$, then
\begin{displaymath}
\frac{1}{c} \ell_M(\alpha)^{-1/2} \le \| c_\beta \|_M <
c \, \ell_M(\alpha)^{-1/2}.
\end{displaymath}
\end{theorem}

\subsection*{Short bases} Suppose $(M,\omega) \in \strat1$. Fix $\epsilon_1
< \epsilon_*$ 
and let $\alpha_1, \dots, \alpha_k$ be the curves with
hyperbolic length less than $\epsilon_1$
on $M$. For every $1 \le i \le k$,
let $\beta_i$ be the shortest curve in the flat metric defined by
$\omega$ with $i(\alpha_i, \beta_i) =1$. We can pick simple closed
curves $\gamma_r$, $1 \le r \le 2g-2k$ on $M$ so that the hyperbolic
length of each $\gamma_r$ is bounded by a constant $L$ depending only
on the genus, and so that the $\alpha_j$, $\beta_j$ and $\gamma_j$ form
a symplectic basis $\cS$ for $H_1(M,\reals)$. We will call such a
basis {\em short.} {}{A short basis is not unique, and
in the following we fix some measurable choice of a short basis at
each point of $\cH_1(\alpha)$. }

We recall the definition of a modified Hodge norm from~\cite{EMM-Isolation}; 
this is similar (but not the same) to the one defined in~\cite{ABEM}. 
The modified norm is defined on the tangent space to the
space of pairs $(M,\omega)$ where $M$ is a Riemann surface and
$\omega$ is a holomorphic $1$-form on $M$. Unlike the Hodge norm, 
the modified Hodge norm will
depend not only on the complex structure on $M$ but also on the choice
of a holomorphic $1$-form $\omega$ on $M$. Let $\{\alpha_i, \beta_i,
\gamma_r\}_{1 \le i \le k, 1 \le r \le 2g-2k}$ 
be a short basis for {}{$x=(M,\omega)$.}

We can write any $\theta \in H^1(M,\reals)$ as 
\begin{equation}
\label{eq:expand:in:basis}
\theta = \sum_{i=1}^k a_i (*c_{\alpha_i}) + \sum_{i=1}^k b_i
\ell_{\alpha_i}(\sigma)^{1/2} (*c_{\beta_i})  + \sum_{r=1}^{2g -2k }
u_i (*c_{\gamma_r}),
\end{equation}
We then define
\begin{equation}
\label{eq:def:modified:hodge:norm}
 \|\theta\|_x'' = \|\theta\|_{{\rm H},x} +  \left( \sum_{i=1}^k |a_i| +
   \sum_{i=1}^k |b_i| +
    \sum_{r=1}^{2g -2k} |u_r| \right).
\end{equation}
Note that $\| \cdot \|''$ depends on the choice of
a short basis; however, switching to a different short basis can change
  $\| \cdot \|''$ by at most a fixed multiplicative constant depending
  only on the genus. 

From (\ref{eq:def:modified:hodge:norm}) we have: for $1 \le i \le k$, 
\begin{equation}
\label{eq:star:c:alpha:prime:norm}
\|*\!c_{\alpha_i}\|_x'' \asymp 1, 
\end{equation}
see~\S\ref{sec:constants} for the notation $\asymp$. 
Similarly, we have 
\begin{equation}
\label{eq:star:c:beta:prime:norm}
\|*\!c_{\beta_i}\|_x'' \asymp  \| *\!c_{\beta_i} \|_{{\rm H},x} \asymp
\frac{1}{\ell_{M}(\alpha_i)^{1/2}}. 
\end{equation}
In addition, in view of Theorem~\ref{theorem:hodge:hyperbolic}, 
if $\gamma$ is any other moderate length curve on $M$, 
$\|*\!c_\gamma\|_x'' \asymp \|*\!c_\gamma \|_{{\rm H},x} = O(1)$. 
Thus, if $\cB$ is a short basis at {}{$x=(M,\omega)$}, then 
for any $\gamma \in \cB$, 
\begin{equation}
\label{eq:short:basis:extremal:length}
\Ext_\gamma({}{x})^{1/2} \asymp \|\!*\!c_\gamma\|_{{\rm H},x} \le
\|\!*\!c_\gamma\|'' \end{equation}
By $\Ext_\gamma({}{x})\index{e@$\Ext$ the extremal length}$ we mean the
extremal length of $\gamma$ in $M$, where $x =(M,\omega)$.

\bold{Remark.} From the construction, we see that the modified Hodge
norm is greater than the Hodge norm. Also, if the flat length of 
shortest curve in
the flat metric defined by $\omega$ is greater than $\epsilon_1$, then
for any cohomology class $c$, for some $N$ depending on
$\epsilon_1$ and the 
genus, 
\begin{equation}
\label{eq:modified:hodge:compare:to:hodge}
\|c\|'' \le N \|c\|_{{\rm H},x};
\end{equation}
i.e., the modified Hodge norm is within a multiplicative constant of
the Hodge norm. 
 
Note however that for a fixed absolute
cohomology class $c$, $\| c \|''_x$ is not a continuous
function of $x$, as $x$ varies in a Teichm\"uller disk; this is due to the
dependence on the choice of a short basis. To remedy this, we pick a
positive, continuous, ${\rm SO}(2)$-bi-invariant function $\phi$ on
$\SL(2,\reals)$ which is supported on a neighborhood of the identity with $\int_{\SL(2,\reals)} \phi(g) \, dg = 1$, and define
\begin{displaymath}
\|c\|'_x = \|c\|_{{\rm H},x} + \int_{\SL(2,\reals)} \|c\|''_{g
  x} \, \phi(g) \, dg.  
\end{displaymath}
It follows from~\cite[Lemma 7.4]{EMM-Isolation}
that for a fixed $c$, $\log \|c\|'_x$ is uniformly
continuous as $x$ varies in a Teichm\"uller disk. In fact, there is a
constant $m_0$ such that for all $x \in \strat1$, all $c \in
H^1(M,\reals)$ and all $t > 0$, 
\begin{equation}
\label{eq:log:abs:unif:cts}
e^{-m_0 t} \|c\|'_x  \le \|c\|'_{a_t x} \le e^{m_0 t}\|c\|'_x .
\end{equation}

\begin{remark} 
Even though $\|\cdot\|_x'$ is uniformly continuous as
long as $x$ varies in a Teichm\"uller disk, it may be only measurable
in general (because of the choice of short basis). 
\end{remark}

\subsection{Relative cohomology}
\label{sec:subsec:relative}
For $c \in \rhr$ and $x =(M,\omega)\in \strat1$, let $\gp_x(c)\index{p@ $\gp_x(c)$}$ 
denote the harmonic representative of
$p(c)$, where $p: \rhr \to \hr\index{p@ $p$ the natural map from $\rhr \to \hr$}$ is the natural map. We view $\gp_x(c)$
as an element of $\rhr$. Then, (similarly to~\cite[\S7]{EMM-Isolation}, see also \cite{ABEM} and \cite{EMR-Counting}) we define the modified Hodge norm $\|\;\|'$ on $\rhr$ as follows.
\begin{displaymath}
\|c\|'_x = \| p(c)\|'_x +
\sum_{(z,z') \in \Sigma\cross \Sigma} 
\left|\int_{\gamma_{z,z'}} (c - \gp_x(c))\right|,
\end{displaymath}
where $\gamma_{z,z'}$ is any path connecting the zeroes $z$ and $z'$ of
$\omega$. 
Since $c-\gp_x(c)$ represents the zero class  in absolute
cohomology, the integral does not depend on the choice of
$\gamma_{z,z'}$. Note that the $\|\cdot\|'$ norm on $\rhr$ 
is invariant under the action of ${\rm SO}(2)$. 

As above, we pick a
positive continuous ${\rm SO}(2)$-bi-invariant function $\phi$ on
$\SL(2,\reals)$ supported on a neighborhood of the identity such
that $\int_{\SL(2,\reals)} \phi(g) \, dg = 1$, and define
\begin{equation}
\label{eq:def:relative:hodge:norm}
\|c\|_x\index{n@$\norm{\;}_x$ the modified Hodge norm at $x$}= \int_{\SL(2,\reals)} \|c\|'_{g
  x} \, \phi(g) \, dg.  
\end{equation}
Then, the $\| \cdot \|_x$ norm on $\rhr$ is also invariant under the
action of ${\rm SO}(2)$.



By~\cite[Lemma 7.5]{EMM-Isolation} there exists some $\constA\label{A:Hodge-cont}$ so that
\be\label{eq:log-cont}
e^{-\ref{A:Hodge-cont}t}\|c\|_{{x}}\leq \|(a_t)_*c\|_{a_t{x}}\leq e^{\ref{A:Hodge-cont}t}\|c\|_{{x}}.
\ee

\subsection{The AGY-norm}\label{sec:AGY-norm}
Let $\|\cdot\|_{{\AGY},x}$ denote the norm defined in~\cite[\S2.2.2]{AGY-EM}.
We recall the definition: let $x=(M,\omega)\in\Halpha$. For any $c\in H^1(M,\Sigma,\cx)$, define
\be\label{def:AGY-norm}
\|c\|_{{\AGY},x}\index{n@$\norm{\;}_{{\AGY},x}$ the $\AGY$ norm at $x$}=\sup_{\gamma}\frac{|c(\gamma)|}{|\Phi(x)(\gamma)|}
\ee
where the supremum is taken over all saddle connections of $\omega$.
This defines a norm and the corresponding Finsler metric is complete, see~\cite{AGY-EM}. 

We note that $\|\;\|_{x}$ and $\|\;\|_{{\AGY},x}$ are commensurable 
to each other on compact subsets of $\Halpha$.  

For every $x=(M,q)\in\mathcal Q_1(\alpha)$,
we define the norms $\|\;\|_x$ and $\|\;\|_{{\AGY},x}$ using the branched double cover $\hat M$.

\begin{lemma}\label{lem:growth-AGY-norm}
Let $c\in H^1(M,\Sigma,\cx)$, $t\geq 0$ and $s\in [0,1]$. Then 
\be\label{eq:growth-AGY-norm}
e^{-2-2t}\|c\|_{{\AGY}, x}\leq \|(a_tu_s)_* c\|_{{\AGY},a_tu_sx}\leq e^{2+2t}\|c\|_{{\AGY}, x}.
\ee
\end{lemma}

\begin{proof}
This is proved in~\cite[Lemma 5.2]{AG-Eigen}, see also~\cite[eq.~(2.13)]{AGY-EM}, we recall the argument. 
Write $c=a'+{\bf i}b'$ and $\Phi(x)=a+{\bf i}b$. Then the definition~\eqref{def:AGY-norm}, implies that  
for all $t\geq 0$ and $|s|\leq 1$ we have 
\be\label{eq:growth-AGY-norm-proof}
\begin{aligned}
\|(a_tu_s)_*c\|_{{\AGY},a_tu_sx}&=\sup_{\gamma}\tfrac{|e^t(a'(\gamma)+sb'(\gamma))+{\bf i}e^{-t}b'(\gamma)|}{|e^t(a(\gamma)+sb(\gamma))+{\bf i}e^{-t}b(\gamma)|}\\
&\leq e^{2t}\sup_{\gamma}\tfrac{|a'(\gamma)+sb'(\gamma)+{\bf i}b'(\gamma)|}{|a(\gamma)+sb(\gamma)+{\bf i}b(\gamma)|}.
\end{aligned}
\ee

By the triangle inequality, for every $|s|\leq 1$ and every $z=z_1+{\bf i}z_2\in\mathbb C$ we have 
\[
|u_sz|=|z_1+sz_2+{\bf i}z_2|\leq |z_1+{\bf i}z_2|+|z_2|\leq 2|z|;
\]
since $z=u_{-s}u_sz$, we also get that $|u_sz|\geq |z|/2$.

This observation and~\eqref{eq:growth-AGY-norm-proof} imply that 
\[
\|(a_tu_s)_*c\|_{{\AGY},a_tu_sx}\leq 4e^{2t}\sup_{\gamma}\tfrac{|a'(\gamma)+{\bf i}b'(\gamma)|}{|a(\gamma)+{\bf i}b(\gamma)|}.
\]

The lower bound follows similarly. 
\end{proof}


\subsection{Non-divergence results}\label{sec:nondivergence} 
Recall that $\mathcal Q_1(\alpha)$ is realized as an affine invariant submanifold in $\mathcal H_1(\hat\alpha)$, moreover,
compact subsets of $\mathcal Q_1(\alpha)$ lift to compact subsets of $\mathcal H_1(\hat\alpha)$.
Let $u\index{u@$u(x)$ the Margulis function for the cusp}:\mathcal H_1(\hat\alpha)\to[2,\infty]$ be the function constructed in~\cite{EsMas-Upperbound} and \cite{A}.

\begin{theorem}
\label{thm:fast-return}
There exists a compact subset $K'_\alpha\subset\Amnfld$ and some $\constA\label{A:fast-return}>0$ 
with the following property.
For every $t_0$ and every $x \in \Amnfld$, there exists 
\[
\mbox{$s \in [0,1/2]$ and $t_0\leq t\leq\max\{2t_0,\ref{A:fast-return}\log u(x)\} $} 
\]
such that $x' = a_t u_s x \in K'_\alpha$. 
\end{theorem}
\begin{proof}
{}{
The stratum $\mathcal Q_1(\alpha)$ is an affine invariant submanifold in $\mathcal H_1(\hat\alpha)$. 
The claim thus follows from \cite[Thm.~2.2]{A} and~\cite[Lemma 6.3]{AG-Eigen} applied with $\delta=1/2$.
}
\end{proof}

\subsection{Period box}\label{sec:period-box}
Let $\tilde x=(M,q)\in\tAmn$. For every $r>0$ define
\[
\mathsf R_{r}(\tilde x):=\bigl\{\Phi(\tilde x)+a'+{\bf i}{b'}:{ a'},{ b'}\in H^1(M,\Sigma,\reals),\|a'+{\bf i} b'\|_{{\AGY},\tilde x}\leq r\bigr\}.
\] 
Let now $r>0$ be so that $\Phi^{-1}$ is a homeomorphism on $\mathsf R_{r}(\tilde x)\cap \Phi(\tAmn)$.
Put 
\[
\mathsf B_r(\tilde x)\index{b@$\mathsf B_r(\tilde x)$ ball of radius $r$ around $\tilde x$ in $\tAmn$}=\Phi^{-1}\Bigl(\mathsf R_{r}(\tilde x)\Bigr).
\]
The open subset $\mathsf B_r(\tilde x)$ will be called a {\em period box} of radius $r$ centered at $\tilde x$. 
Thanks to~\cite[Prop.\ 5.3]{AG-Eigen}, $\mathsf B_r(\tilde x)$ is well defined for 
all $0<r\leq 1/2$ and {\em all} $\tilde x\in\tAmn$. We also have the following. 

\begin{lemma}\label{lem:AGY-Moduli}
There exists some $\constA\label{A:AGY-Mod}$ so that for all $x\in\mathcal Q_1(\alpha)$ and every 
$0<r\leq u(x)^{-\ref{A:AGY-Mod}}$ the following hold. Let $\tilde x\in \tAmn$ be a lift of $x$. Then 
\begin{enumerate} 
\item The restriction of the covering map $\pi$ to $\mathsf B_{r}(\tilde x)$ 
is injective. 
\item For all $\tilde x_1, \tilde x_2\in \mathsf B_{r}(\tilde x)$, 
the Teichm\"{u}ller distance between $\tilde x_1$ and $\tilde x_2$ is at most $1$.
\end{enumerate}
\end{lemma}

\begin{proof}
The argument is similar to the one used in the proof of~\cite[Lemma~8.2]{EMM-Isolation}.

For part~(2) we will need the following two facts: $d_{\mathcal T}((a_tu_s)^{\pm1} z,(a_tu_s)^{\pm1}z')\leq 16e^{2t}$ for all $t\geq 0$ and $s\in[-1,1]$
where $d_{\mathcal T}$ denotes the Tichm\"uller distance. 
Moreover, there exist a constant $C\geq 1$ so that 
\[
C^{-1}d_{\AGY}(z,z')\leq d_{\mathcal T}(z,z')\leq Cd_{\AGY}(z,z')\quad\text{for all $z,z'\in K'_\alpha$,}
\]
where $K'_\alpha\subset\Amnfld$ is the compact set introduced in Theorem~\ref{thm:fast-return}. 

 
We now turn to the proof of the lemma. For every $x\in K'_\alpha$, there exists $0<r(x)\leq1/2$ so that $\mathsf B_{r(x)}(x)$ is embedded 
in the sense that the projection from
the Teichm\"uller space $\tAmn$ to the Moduli space $\mathcal Q_1(\alpha)$ restricted to $\mathsf B_{r(x)}(\tilde x)$ is injective.
Let $r_0 = \inf_{x \in K'_\alpha} r(x)$. By compactness of $K'_\alpha$, $r_0 >0$. Decreasing $r_0$ if necessary, we assume that  
for all $x\in K'_\alpha$ and all $\tilde x_1, \tilde x_2\in \mathsf B_{r_0}(\tilde x)$,
the Teichm\"{u}ller distance between $\tilde x_1$ and $\tilde x_2$ is at most $1$.

Let $N\geq 1$ be so that 
\begin{equation}
\label{eq:choice:N}
C2^{4\ref{A:fast-return}-N+16} < r_0\leq 1/2. 
\end{equation}
where $\ref{A:fast-return}$ is as in Theorem~\ref{thm:fast-return}. 

We will show that $\ref{A:AGY-Mod}=N$ satisfies the claims in the lemma.  
First note that in view of~\cite[Prop.\ 5.3]{AG-Eigen}, $\mathsf B(\tilde x):=\mathsf B_{u(x)^{-N}}(\tilde x)$ 
is well defined for all $x\in \mathcal Q_1(\alpha)$ and all the lifts $\tilde x\in \tAmn$. 
Suppose now that there exists $x \in \mathcal Q_1(\alpha)$ and 
$\tilde x_1, \tilde x_2 \in \mathsf B(\tilde x)$ such that $\tilde x_2 = \mce \tilde x_1$ for some $\mce$ in the mapping class group. 
Write
\begin{displaymath}
\tilde x_i =\tilde x + v_i, \quad\text{ where $\|v_i\|_{{\AGY},x}\leq u(x)^{-N}$} 
\end{displaymath}
By Theorem~\ref{thm:fast-return}, there exists $s \in [0,1/2]$ and $\tau \le \ref{A:fast-return} \log u(x)$ such that
$x' \equiv a_\tau u_s x \in K'_\alpha$.

Let $x_i' = a_\tau u_sx_i$, and put $\tilde{x}'_i=a_\tau u_s\tilde x_i$; also put $\tilde x'=a_\tau u_s \tilde x$. 
Then, in view of~\eqref{eq:growth-AGY-norm} we have 
\begin{equation}\label{eq:proof of AGY-Moduli}
\|v_i\|_{{\AGY},x_i'} \le e^{2+2\tau}u(x)^{-N}
\le 8u(x)^{2\ref{A:fast-return}-N+2} \le 2^{2\ref{A:fast-return}-N+5} \le r_0
\end{equation}
where for the last estimate we used (\ref{eq:choice:N}) and the fact that $u(x) \ge 2$. 
However, $\tilde{x}_2' = \mce\tilde{x}'_1$, thus, both $x_1'$ and $x_2'$ 
belong to the projection of $\mathsf B_{r_0}(\tilde{x}')$; this contradicts the fact that $\mathsf B_{r_0}(x')$ is embedded.

This contradiction shows that $\mathsf B_{u(x)^{-N}}(x)$ is embedded, establishing part~(1).

We now turn to part~(2). We use the above notation. Let $\tilde x_1, \tilde x_2\in \mathsf B_{u(x)^{-N}}(x)$, and define 
$x_i' = a_\tau u_sx_i\in K'_\alpha$ and $\tilde{x}'_i=a_\tau u_s\tilde x_i$ as above. 
Then~\eqref{eq:proof of AGY-Moduli} implies that 
\[
d_{\AGY}(\tilde{x}'_1,\tilde{x}'_2)\leq 16u(x)^{2\ref{A:fast-return}-N+2}
\] 
Hence, $d_{\mathcal T}(\tilde{x}'_1,\tilde{x}'_2)\leq 16Cu(x)^{2\ref{A:fast-return}-N+2}$. 
Since $\tilde x_i=(a_\tau u_s)^{-1}\tilde{x}'_i$, we conclude that
\[
d_{\mathcal T}(\tilde x_1,\tilde x_2)\leq Cu(x)^{4\ref{A:fast-return}-N+16}<1
\]
where we used~\eqref{eq:choice:N} and $u(x)\geq 2$ in the last inequality. The proof is complete.
\end{proof}

For every $x\in\mathcal Q_1(\alpha)$ we put
\be\label{eq:def-r(x)}
r(x)=u(x)^{-\ref{A:AGY-Mod}}; 
\ee
for every compact subset $K\subset\mathcal Q_1(\alpha)$, let $r(K)=\inf\{r(x): x\in K\}$\index{r@ $r(x)$ and $r(K)$}.

For every $0<r\leq r(x)$, we let $\mathsf B_r(x)$\index{b@ $\mathsf B_r(x)$} 
denotes $\pi(\mathsf B_r(\tilde x))$ where $\tilde x\in\tAmn$ is an arbitrary lift of $x$.
We refer to $\mathsf B_r(x)$ as the ball of radius $r$ centered at $x$.

\subsection{Horospherical foliation}\label{sec:Horo-Fol}
Given a point ${x}=(M,q)\in\Amnfld$, the tangent space ${\rm T}_{x}\Amnfld$ decomposes as 
\[
{\rm T}_{x}\Amnfld=\reals \mathbf v({x})\index{v@${\bf v}(x)$ the direction of the geodesic flow at $x$}\oplus E^{\unst}({x})\oplus E^{\stbl}({x})
\]
where ${\bf v}({x})$ with $\|{\bf v}(x)\|_{\AGY,x}=1$ determines the direction of the Teichm\"{u}ller geodesic flow, 
\begin{align*}
&E^{\unst}({x})\index{e@$E^{\unst, \stbl}(x)$}={\rm T}_{x}\Amnfld\cap \operatorname D\!\Phi_{x}^{-1}\bigl(H^1(\dagger,\ddagger,\reals)\bigr),\text{ and }\\ 
&E^{\stbl}({x})={\rm T}_{x}\Amnfld\cap\operatorname D\!\Phi_{x}^{-1}\bigl({\bf i}H^1(\dagger,\ddagger,\reals)\bigr).
\end{align*}
where $(\dagger,\ddagger)=(M,\Sigma)$ if $\varsigma=1$ and $(\dagger,\ddagger)=(\hat M,\hat\Sigma)$ if $\varsigma=-1$ ---
recall that $\hat M$ is the orienting double cover of $M$ and we use $\Phi$ to locally identify $\reals\Amnfld$ with $H^1(M,\Sigma,\cx)$ if 
$\varsigma=1$ and with the $H_{\rm odd}^1(\hat M,\hat\Sigma,\cx)$ if $\varsigma=-1$.

If $\Phi({x})=a+{\bf i}b$ for some ${x}\in\Amnfld$, then 
\be\label{eq:E+-real}
E^{\unst}({x})=\{a'\in H^1(M,\Sigma,\reals):\intsc(a',b)=0\},
\ee
and $E^{\stbl}({x})=\{{\bf i}b'\in {\bf i}H^1(M,\Sigma,\reals):\intsc(a,b')=0\}$ when $\varsigma=1$.
Similarly, one can define $E^{\unst, \stbl}$ in the case $\varsigma=-1$.

The subspaces $E^{\unst, \stbl}(x)$ depend smoothly on $x$, moreover, they 
are integrable. We denote the corresponding leaves by $\mlhoro(x)$\index{w@$\mlhoro(x)$ unstable foliation in $\Amnfld$} 
and $W^{\stbl}(x)$\index{w@$W^{\stbl}(x)$ stable foliation in $\Amnfld$}, respectively.
Also put 
\[
W^{\rm cu}(x):=\{a_tW^{\unst}(x): t\in\reals\}\index{w@$W^{\rm cu}(x)$ center-unstable foliation in $\Amnfld$}
\] 
and $W^{\rm cs}(x):=\{a_tW^{\stbl}(x):t\in\reals\}$\index{w@$W^{\rm cs}(x)$ center-stable foliation in $\Amnfld$}.


Let $\mu_{x}^{\unst}$ and $\mu_x^{\stbl}$\index{m@$\mu_x^{\unst},\mu_x^{\stbl}$ conditional measures of $\mu$ along $W^{\unst}(x), W^{\stbl}(x)$} denote the leafwise measures of the natural measure $\mu$ along $W^{\unst}(x)$ and $W^{\stbl}(x)$, respectively. Then $y\mapsto\mu_{y}^{\unst,\stbl}$ is constant along $W^{\unst,\stbl}(x)$, respectively, and we have
\be\label{eq:Margulis-measure}
(a_t)_*\mu_x^\unst=e^{-ht}\mu^\unst_{a_tx}\quad{and}\quad(a_t)_*\mu_x^\stbl=e^{ht}\mu^\stbl_{a_tx};
\ee
see also~\cite[\S4]{AG-Eigen} where these measures are defined using volume forms. 

If $\mathsf B_r(x)$ is a period box centered at $x$, then $\mu|_{\mathsf B_r(x)}$ has a product structure
as $\dif\!\operatorname{Leb}\times\dif\!\mu^{\stbl}\times\dif\!\mu^{\unst}$, see e.g.~\cite[Prop.\ 4.1]{AG-Eigen}.

Given $x\in\Amnfld$ and a period box $\mathsf B_r(x)$ with center $x$ and $0\leq r\leq r(x)$, we let 
\[
\mathsf B_r^{\unst, \stbl}(x)\index{b@ $\mathsf B_r^{\unst}(x), \mathsf B_r^{\stbl}(x), \mathsf B_r^{\rm cu}(x), \mathsf B_r^{\rm cs}(x)$}=\text{the connected component of $x$ in $\mathsf B_r(x)\cap W^{\unst, \stbl}(x)$}.
\] 
Define
$\mathsf B_r^\bullet(x)$ for $\bullet={\rm cu},{\rm cs}$ similarly. 

We also denote functions which are supported on the leaves $W^{\unst}$, $W^{\rm cu}$, etc. using the same superscript, e.g., 
$\phi^{\unst}$ denotes a function which is supported on a leaf $W^{\unst}(x)$.\index{f@$\phi^{\unst}, \phi^{\rm cu}, \phi^{\stbl}, \phi^{\rm cs}$} 

We use the norm $\|\cdot\|_{{\AGY},x}$ to induce a metric $d_{W^{\unst, \stbl}(x)}$ on $\mathsf B^{\unst, \stbl}_r(x)$ for $0<r<r(x)$.
Hence notions such as ${\rm diam}$ etc.\ refer to this metric. 

Let $\tW^\bullet(\tilde x)$\index{w@$\tW^\bullet(\tilde x)$ foliation $\bullet$ in $\tAmn$} denote the foliation $\bullet$ in $\tAmn$, 
and define $\mathsf B^\bullet(\tilde x)$ accordingly.

Let $w^{\unst,\stbl}\in E^{\unst,\stbl}({x})$. Then  
\be\label{eq:AGY-st-unst}
\|(a_t)_*w^{\unst}\|_{\AGY, a_tx}\geq \|w^{\unst}\|_{\AGY,x}\quad\text{and}\quad\|(a_t)_*w^{\stbl}\|_{\AGY, a_tx}\leq \|w^{\stbl}\|_{\AGY,x},
\ee
see~\cite[Lemma 5.2]{AG-Eigen}. Moreover, we have the following uniform hyperbolicity estimate.

\begin{prop}\label{prop:unif-hyp}
Let $K\subset \Amnfld$ be a compact subset.
There exist some $\consta( K)\label{a:unif-hyp}$ 
and some $t_0=t_0(K)$ with the following property.
Let $t\geq t_0$; suppose that $x,a_tx\in K$, moreover, assume that
\[
|\{\tau\in[0,t]:a_{\tau}x\in K\}|\geq t/3.
\]
Then 
\[
\mbox{$\|(a_t)_*w\|_{{\AGY},a_{t}x}\leq e^{-\ref{a:unif-hyp}(K)t}\|w\|_{{\AGY},x}\;\;$ and 
$\;\;\|(a_t)_*w\|_{a_{t}x}\leq e^{-\ref{a:unif-hyp}(K)t}\|w\|_{x}$}
\] 
for all $w\in E^{\stbl}(x)$ and all $t\geq t_0$. 
\end{prop}

\begin{proof}
Let $\|\;\|_{{\rm ABEM},x}$ denote the modified Hodge norm defined in~\cite[\S3]{ABEM}.
Let $C$ be a constant so that
\be\label{eq:ABEM-AGY-K}
C^{-1}\|v\|_{{\rm ABEM},y}\leq \|v\|_{{\AGY},y}\leq C\|v\|_{{\rm ABEM},y}
\ee 
for all $y\in K$.

In view of~\cite[Thm.~3.15]{ABEM}, there exists some $\consta( K)\label{a:unif-hyp-ABEM}$
so that under our assumptions in this proposition we have 
\be\label{eq:unif-hyp-ABEM}
\|(a_t)_*w\|_{{\rm ABEM},a_{t}x}\leq e^{-\ref{a:unif-hyp-ABEM}t}\|w\|_{{\rm ABEM},x}.
\ee 

We now compute
\begin{align*}
\|(a_t)_*w\|_{{\AGY},a_{t}x}&\leq C\|(a_t)_*w\|_{{\rm ABEM},a_tx}&&\text{since $a_tx\in K$}\\
&\leq Ce^{-\ref{a:unif-hyp-ABEM}t}\|w\|_{{\rm ABEM},x}&&\text{by~\eqref{eq:unif-hyp-ABEM}}\\
&\leq C^2e^{-\ref{a:unif-hyp-ABEM}t}\|w\|_{{\AGY},x}&&\text{since $x\in K$}.
\end{align*}
The claim thus holds with $\ref{a:unif-hyp}=\ref{a:unif-hyp-ABEM}/2$ and 
$t_0=\tfrac{4\log C}{\ref{a:unif-hyp-ABEM}}$.
\end{proof}

\begin{lemma}\label{cor:nondiv-horosphere}
Let $K'_\alpha$ be as in Theorem~\ref{thm:fast-return}.
There is a positive constant $\constA\label{A:horo-fast-return}$ and 
for every $0<\theta<1$ there exists 
$\consta(\theta)\label{a:horo-return}$,
and a compact subset $K_\alpha(\theta)\supset K'_\alpha$  
with the following properties. 
Let $x\in\Amnfld$, $0<r\leq r(x)$, and let $\mathsf B_r(x)$ be a period box centered at $x$. Put
\[
\mathsf H_t^{\unst}(x,\theta)\index{h@ $\mathsf H_t^{\unst}(x,\theta)$}:=\bigl\{y\in\mathsf B_r^{\unst}(x):
|\{\tau\in[0,t]:a_{\tau}y\in K_\alpha(\theta)\}|\geq \theta t\bigr\}.
\]

Then for every $t\geq\ref{A:horo-fast-return}\log u(x)$, we have  
\[
\mu_x^{\unst}\left(\mathsf B_r^{\unst}(x)\setminus \mathsf H_t^{\unst}(x,\theta)\right)
\leq e^{-\ref{a:horo-return}(\theta)t}\mu_x^{\unst}(\mathsf B_r^{\unst}(x)).
\]
\end{lemma}

\begin{proof}
See~\cite[Prop.\ 6.1]{AG-Eigen}.
\end{proof}

We apply the above with $\theta=0.5$, and put 
\be\label{eq:H-+-t}
\text{$K_\alpha=K_\alpha(0.5)$, $\quad\ref{a:horo-return}:=\ref{a:horo-return}(0.5),\quad$ and $\quad\mathsf H_t^{\unst}(x)\index{h@$\mathsf H_t^{\unst}(x)$}:=\mathsf H^{\unst}_t(x,0.5)\quad$}
\ee 
for the rest of the paper.

We have the following corollary

\begin{cor}\label{cor:unif-hyp'}
Let $x\in\Amnfld$, and let $t\geq\ref{A:horo-fast-return}\log u(x)$. 
For every $y\in \mathsf H_t^{\unst}(x)$ and every $w\in E^{\unst}(x)$ we have
\[
\|(a_{-t})_*w\|_{\AGY, a_{-t}y}\leq e^{-0.5\ref{a:unif-hyp}(K_\alpha)t}\|w\|_{\AGY, y}.
\] 
\end{cor}

\begin{proof}
Let $\tau_0<\tau_1$ be the first and the last time so that $a_\tau y\in K_\alpha$. 
Then in view of Lemma~\ref{cor:nondiv-horosphere}, $\tau_1-\tau_0\geq 0.5 t\geq (\tau_1-\tau_0)/3$. 
Therefore, by~\eqref{eq:AGY-st-unst} and Proposition~\ref{prop:unif-hyp}, we have 
\begin{align*}
\|(a_{-t})_*w\|_{\AGY, a_{-t}y}&\leq \|(a_{-\tau_1})_*w\|_{\AGY, a_{-\tau_1}y}\\
&\leq e^{-0.5\ref{a:unif-hyp}(K_\alpha)t}\|(a_{-\tau_0})_*w\|_{\AGY, a_{-\tau_0}y}\\
&\leq e^{-0.5\ref{a:unif-hyp}(K_\alpha)t}\|w\|_{\AGY, y}
\end{align*}
as we claimed. 
\end{proof}

\subsection{Smooth structure on affine manifolds}\label{sec:smooth-structure}
As it is done in~\cite[\S5.2]{AG-Eigen}, we use the affine structure to define a smooth structure
on $\tAmn$ and $\mathcal Q_1(\alpha)$. Let us recall the definition of a $C^k$-norm from~\cite{AG-Eigen}, see also~\cite{AGY-EM}.

Let $W\subset \mathcal Q_1(\alpha)$ be an affine submanifold.
For a function $\varphi$ on $W$ define
\[
c_k(\varphi)=\sup|\operatorname D^k\varphi(x,v_1,\ldots,v_k)|,
\]
where the supremum is taken over $x$ in the domain of $\varphi$ and $v_1,\ldots,v_k\in {\rm T}_xW$ with ${\AGY}$-norm at most $1$. 
Define the $C^k$-norm of $\varphi$ as 
$
\|\varphi\|_{C^k}=\sum_{j=0}^k c_j(\varphi).
$

By a $C^k$ function we mean a function whose $C^k$-norm is finite. 
The space of compactly supported $C^k$ functions on $W$ will be denoted by $C_c^k(W)$, similarly, 
we define $C_c^\infty(W)$\index{c@$C_c^\infty(W)$}.

In the sequel we will only need $C^1$-norm of functions. 
To avoid confusion between this norm and 
other relevant norms which will be used, and also since we often use the letter $C$ to denote various constants, 
define
\[
\Sob(\varphi)\index{c@$\Sob(\varphi)$ the $C^1$-norm of $\varphi$}:=\|\varphi\|_{C^1}.
\]
for any $C^1$ function $\varphi$.

In the sequel we will need to replace the characteristic functions of certain sets with their smooth approximations.
The following lemmas will provide such approximations.

\begin{lemma}[Cf.~\cite{AG-Eigen}, Prop.\ 5.8]
\label{lem:partition-unity-AG}
There exists $\constA\label{A:part-unity-exp}\label{A:part-unity-mult}$ so that the following holds.
Let $x\in\mathcal Q_1(\alpha)$. Let $D\subset W^{\unst}(x)$ be a compact set, and let $\epsilon\leq 0.1r(D)$, see~\eqref{eq:def-r(x)}.
There exists a finite collection $\{\varphi_i\}$ of $C^\infty$ functions on $W^{\unst}(x)$ with the following properties:
\begin{enumerate}
\item $0\leq\varphi_i\leq 1$ for all $i$.
\item $\Sob(\varphi_i)\leq \ref{A:part-unity-mult}\epsilon^{-\ref{A:part-unity-exp}}$.
\item For every $i$, $\varphi_i$ is supported on $\mathsf B^{\unst}_\epsilon(y_i)$ for some $y_i\in D$.
\item The covering $\{\mathsf B^{\unst}_\epsilon(y_i)\}$ of $D$ has multiplicity at most $\ref{A:part-unity-mult}$.
\item $\sum\varphi_i\leq 1$, and the equality holds on a neighborhood of $D$. 
\end{enumerate}   
\end{lemma}

\begin{proof}
This is proved in~\cite[Prop.\ 5.8]{AG-Eigen}. 
It is worth mentioning that~\cite[Prop.\ 5.8]{AG-Eigen} is stated for balls of size $\asymp 1$, to get our claim here, 
one needs to apply the argument there not to the $\AGY$ norm, but to the $\AGY$ norm scaled by $1/\epsilon$. 
\end{proof}

%
%
%
%
%
%

Let $W$ be one of the following: $\mathcal Q_1(\alpha)$, $W^{\unst, \stbl}(x)$, or $W^{{\rm cu}, {\rm cs}}(x)$, 
for some $x\in \mathcal Q_1(\alpha)$. Let $E\subset W$ be a compact subset. 
For any $0<\epsilon<0.1r(E)$ define
\[
E_{+,\epsilon}^W=\{y\in W:r(y)\geq \epsilon\text{ and }\mathsf B_\epsilon(y)\cap E\neq\emptyset\};
\]
note that $E_{+,\epsilon}^W$ is an open subset  of $W$ which contains $E$.

Let $r>0$ and $L>1$. Let $\mathcal S_{W}(E, r, L)$ \index{S@$\mathcal S_W(E, r, L)$ the set of Borel functions supported in $E$ which may be approximated by smooth functions} denote the 
class of Borel functions $0\leq f\leq 1$ supported and defined everywhere in $E$ with the following properties: 
for all $\epsilon\leq r/10L$ there exist $\varphi_{+,\epsilon},\varphi_{-,\epsilon}\in C_c^\infty(E_{+,\epsilon}^W)$ so that


\begin{enumerate}
\item[($\mathcal S$-1)] $\varphi_{-,\epsilon}\leq f\leq \varphi_{+,\epsilon}$,
\item[($\mathcal S$-2)] $\Sob(\varphi_{\pm,\epsilon})\leq \epsilon^{-L}$, and 
\item[($\mathcal S$-3)] $\|\varphi_{+,\epsilon}-\varphi_{-,\epsilon}\|_2\leq \epsilon^{1/2}\|f\|_2$.
\end{enumerate} 

If $W$ is clear from the context, we denote $\mathcal S_{W}(E, r, L)$ and $E_{+,\epsilon}^W$ simply by $\mathcal S(E,r,L)$ and $E_{+,\epsilon}$, respectively.

\begin{lemma}
\label{lem:Margulis-prop}
There exists some $L$ depending only on $\alpha$ so that  for all $0<r\leq r(x)$,
\[
1_{\mathsf B^{{\unst, \stbl}}_r(x)}\in\mathcal S_{W^{\unst, \stbl}(x)}(\mathsf B^{{\unst, \stbl}}_r(x),r, L).
\]
Similarly, $1_{\mathsf B_{r}(x)}\in\mathcal S(\mathsf B_{r}(x),r,L)$ for all $0<r\leq r(x)$.
\end{lemma}

\begin{proof}
We will show the claims hold if we choose $L>2\ref{A:part-unity-mult}$, see Lemma~\ref{lem:partition-unity-AG}, large enough.
Apply Lemma~\ref{lem:partition-unity-AG} with $\epsilon$ and  
$D=\mathsf B_{r-2\epsilon}^{\unst}$, and 
denote by $\{\varphi_{i,-}\}$ the functions obtained from that lemma. 
For a second time, apply Lemma~\ref{lem:partition-unity-AG} with
$\epsilon$ and $D=\mathsf B_r^{\unst}(x)$, and denote by $\{\varphi_{i,+}\}$ the functions thus obtained.   
Put
\[
\varphi_{\epsilon,-}=\sum\varphi_{i,-}\quad\text{and}\quad\varphi_{\epsilon,+}=\sum\varphi_{i,+}.
\]
These functions satisfy ($\mathcal S$-1) thanks to Lemma~\ref{lem:partition-unity-AG}(1) and~(5). 
Moreover, they satisfy ($\mathcal S$-2) thanks to Lemma~\ref{lem:partition-unity-AG}(1)---(4) and the fact that $L>2\ref{A:part-unity-mult}$. 

To see ($\mathcal S$-3), first note that $\mu_x^{\unst}\big(\mathsf B_r^{\unst}(x)\setminus\mathsf B_{r-2\epsilon}^{\unst}\big)\ll\epsilon$
where the implied constant depends only on $\alpha$. The claim in ($\mathcal S$-3) thus holds true in view of Lemma~\ref{lem:partition-unity-AG}(5) if we choose $L$ large enough, depending on $\alpha$. 

The second claim follows from the first claim, using the product structure of $\mathsf B_r(x)$ and of the measure $\mu$.
\end{proof}

We fix once and for all some $L$ so that Lemma~\ref{lem:Margulis-prop} holds true 
and drop $L$ from the notation. In particular, $\mathcal S(E,r, L)$ will be denoted by $\mathcal S(E,r)$.\index{s@$\mathcal S(E,r)$} 

Abusing the notation we will write $\mathcal S(x,r)$\index{s@$\mathcal S(x,r)$} for $\mathcal S(E,r)$ 
if the compact subset $E$ is not relevant except for the fact that
it is a compact subset containing the point $x$.

\section{Translates of horospheres}\label{sec:exp-mix}
In this section we will use a fundamental result of Avila,~Gou\"{e}zel, and~Yoccoz,~\cite{AGY-EM, AG-Eigen}
together with Margulis' thickening technique,~\cite{Margulis-Thesis, EMc, KM}, to study translations of pieces of the horospherical foliations along the geodesic flow.

\begin{theorem}[Exponential Mixing, \cite{AGY-EM, AJ-EMQ, AG-Eigen}]\label{thm:EM}
Let $(\mathcal M, \mu)$ be an affine invariant manifold. There exists a positive constant 
$\kappa=\kappa(\mathcal M, \mu)$ so that the following holds.
Let $\Psi_1,\Psi_2\in C_c^\infty(\mathcal M)$, then 
\[
\left|\int \Psi_1(a_tx)\Psi_2(x)\dif\!\mu(x)-\mu(\Psi_1)\mu(\Psi_2)\right|\ll \Sob(\Psi_1)\Sob(\Psi_2)e^{-\kappa t}
\]
where the implied constant depends on $(\mathcal M, \mu)$.
\end{theorem}

We remark that combining \cite{AGY-EM, AJ-EMQ, AG-Eigen} and~\cite{Rn}, the $\Sob$ norm in Theorem~\ref{thm:EM} may be replaced by the $p$-H\"older norm for any $p>0$. However, if we use the $p$-H\"older norm, the constant $\kappa$ will, in general, depend on $p$; in particular, $\kappa$ tends to $0$ as $p$ tends to $0$, see~\cite[Thm.~1]{Rn} and~\cite[Thm.~2.14]{AGY-EM}.    

It is also worth mentioning that the $\Sob$ norm in Theorem~\ref{thm:EM} 
may be taken to include derivatives only in the direction of ${\rm SO}(2)\subset\SL(2,\bbr)$, see~\cite{CHH} and~\cite[Thm.~1]{Rn} and references there. Our choice, $\Sob$, is more restrictive; this is tailored to our applications later, e.g., we will use the estimate that 
$\|\phi\|_\infty\leq \Sob(\phi)$ for any $\phi\in C_c^\infty(\mathcal M)$.

\begin{prop}\label{prop:thickening-smooth}
There exists some $\consta\label{a:th-smooth}$, depending on $\alpha$, 
with the following property.
Let $x\in\Amnfld$, $0<r\leq r(x)$, and let $\mathsf B_r(x)$ be a period box centered at $x$.
Let $\psi^{\unst}\in C_c^\infty(\mathsf B^{\unst}_r(x))$, then for any $\phi\in C_c^\infty(\mathcal Q_1(\alpha))$ we have
\[
\left|\int_{\mlhoro(x)}\phi(a_{\mlsign t}y)\psi^{\unst}(y)\dif\!\mu_x^{\unst}(y)-\int_{\Amnfld} \phi\dif\!\mu\int_{\mlhoro(x)}\psi^{\unst}\dif\!\mu_x^{\unst}\right|\leq \Sob(\phi)\Sob(\psi^{\unst})e^{-\ref{a:th-smooth} t}.
\]   
\end{prop}

We need some notation; we discuss the case $\varsigma=1$, the case $\varsigma=-1$ is similar. 
Let $\Phi(x)=a+{\bf i}b$; recall from~\eqref{eq:E+-real} that
\[
E^{\unst}({x})=\{a'\in H^1(M,\Sigma,\reals):\intsc(a',b)=0\}.
\]
Similarly $E^{\stbl}({x})=\{{\bf i}b'\in {\bf i}H^1(M,\Sigma,\reals):\intsc(a,b')=0\}$. 

These spaces can alternatively be described as follows. Recall that subspace 
$E(x)=\mbox{span}\{a,b\}$, then $E_{\mathbb C}(x)$ is $\SL_2(\bbr)$ equivariant. 
Let 
\[
H^1_{\mathbb C}(x)^\perp:=\{c\in H^1(M,\Sigma,\bbc): p(c)\wedge p(E(x)_{\mathbb C})=0\},
\]
similarly define $H^1_{\mathbb R}(x)^\perp$. 

The unstable leaf $W^{\unst}(x)$ is locally identified with $\Phi(x)+sb+w$ for 
$s\in\bbr$ and $w\in H^1_{\mathbb R}(x)^\perp$. 
Similarly the center stable leaf $W^{\rm cs}(x)$ is locally identified with $\Phi(x)+\tau{\bf v}(x)+s'{\bf i}a+{\bf i}w'$
where $\tau,s'\in\bbr$ and $w'\in H^1_{\mathbb R}(x)^\perp$.

Let $0<r\leq0.1r(x)$ and let $y\in \mathsf B_r(x)$. Write $\Phi(y)=a_y+{\bf i}b_y$. 
We define the stable projection $y^{\unst}\in B^{\unst}_{2r}(x)$
as the unique point so that $\Phi(y)=\Phi(y^\unst)+\tau{\bf v}(y)+sa_y+w$ where $\tau, s\in\bbr$ with $|\tau|, |s|\leq 2r$
and $w\in H^1_{\mathbb R}(y)^\perp$ with $\|w\|_{\AGY, x}\leq 2r$. Put  
\[
\mathsf{FB}_r(x)=\{y\in \mathsf B_r(x): y^{\unst}\in\mathsf B_r^{\unst}(x)\}.
\]
Then $\mathsf B_r^{\unst}(x)\subset \mathsf{FB}_r(x)$. 

For every $0<\delta<0.1r$ and every $y\in \mathsf{B}_r(x)$, let
\[
\mathsf D^{\rm cs}_{\delta}(y)=\{a_\tau z: |\tau|\leq \delta, z\in W^{\stbl}(y), \Phi(z)=\Phi(y)+w, \|w\|_{{\AGY},x}\leq \delta\}.
\]
For every $y\in\mathsf{B}_r^{\unst}(x)$, let $p^{\rm cs}_y:W^{\rm cs}(y)\cap \mathsf B_{r}(x)\to
W^{\rm cs}(x)$ be the projection along unstable leaves. Then $0.5\leq {\rm Jac}(p^{\rm cs}_y)\leq 2$, moreover we have 
\[
W^{\rm cs}(x)\cap \mathsf B_{0.1\delta}(x)\subset p^{\rm cs}_y\bigl(\mathsf D^{\rm cs}_{\delta}(y)\bigr)\subset W^{\rm cs}(x)\cap \mathsf B_{10\delta}(x). 
\] 

We now begin the proof of the Proposition~\ref{prop:thickening-smooth}.

\begin{proof} 
The idea is to relate the integral $\int_{\mlhoro(x)}\phi(a_{\mlsign t}y)\psi^{\unst}(y)\dif\!\mu_x^{\unst}(y)$ to correlations of the function $a_{-t}\phi$
with a {\em thickening} of $\psi^{\unst}$ in the direction of $W^{\rm cs}(x)$. Then we may use Theorem~\ref{thm:EM} to conclude the proof.

To that end, let $0<\epsilon<0.01r(x)$ be a parameter which will be fixed later. In particular,  
it will be taken to be of the form $e^{-\kappa t}$. Let $\psi^{\rm cs}$ 
be a smooth function supported in $\mathsf D^{\rm cs}_{\epsilon}(x)$ 
so that $\int_{W^{\rm sc}(x)}\psi^{\rm cs}=1$. We can choose such a function so that it moreover satisfies 
$\Sob(\psi^{\stbl})\ll \epsilon^{-\constA\label{A:char-sob}}$ where $\ref{A:char-sob}$ and the implied constant depend on $\alpha$. 


Define $\Psi$ on $\mathsf {FB}_r(x)$ by
\be\label{eq:psi-def}
\Psi(y)=\lambda_{y^{\unst}}\psi^{\rm cs}\bigl(p_{y^{\unst}}^{\rm cs}(y)\bigr)\cdot \psi^{\unst}(y^{\unst})
\ee
where $\lambda_{y^\unst}^{-1}=\int_{W^{\rm cs}(y^{\unst})}\psi^{\rm cs}\bigl(p_{y^{\unst}}^{\rm cs}(w)\bigr)\dif\!\mu_{y^\unst}^{\rm cs}(w)$. 
Extend $\Psi$ to a smooth function on $\Amnfld$ by defining $\Psi(y)=0$ for all
$y\not\in\mathsf {FB}_r(x)$; note that $\mu(\Psi)=\mu_x^{\unst}(\psi^{\unst})$, see the computation in~\eqref{eq:em-ineq-3}.

We need the following lemma.
\begin{lem*}
There exists $\consta\label{a:f-abos-cont}$ depending only on $\alpha$ so that 
\begin{equation}\label{eq:em-ineq-4}
\left|\int_{\mlhoro(x)}\phi(a_{\mlsign t}y)\psi^{\unst}(y)\dif\!\mu_x^{\unst}(y)-\int_{\Amnfld}\phi(a_{\mlsign t}z)\Psi(z )\dif\!\mu(z)\right|\ll 
\Sob(\phi)\Sob(\Psi)\epsilon^{\ref{a:f-abos-cont}}
\end{equation}
where the implied constant depends only on $\alpha$.
\end{lem*}

Let us assume the lemma and finish the proof of the proposition. 
Optimizing the choice of $\epsilon$ to be of size $e^{-\kappa t}$ for some small $0<\kappa<1$, 
the proposition follows from~\eqref{eq:em-ineq-4} and Theorem~\ref{thm:EM} applied with $\Psi_1=\phi$ and $\Psi_2=\Psi$ --- recall again that $\mu(\Psi)=\mu_x^{\unst}(\psi^{\unst})$. 
\end{proof}

\begin{proof}[Proof of the Lemma]
Since $\Psi$ is supported in $\mathsf {FB}_r(x)$, we need to estimate  
\begin{equation}\label{eq:em-ineq-1}
\int_{\mathsf B_r^{\unst}(x)}\phi(a_{\mlsign t}y)\psi^{\unst}(y)\dif\!\mu_x^{\unst}(y)-\int_{\mathsf {FB}_r(x)}\phi(a_{\mlsign t}z)\Psi(z )\dif\!\mu(z).
\end{equation}


Let $z\in\mathsf {FB}_r(x)$.
Recall that $\Phi(z)=\Phi(z^{\unst})+w$ where $w\in\bbr{\bf v}(z^{\unst})+E^{\stbl}(z^{\unst})$, indeed $z\in W^{\rm cs}(z^{\unst})$. 
In view of \eqref{eq:AGY-st-unst} we have
\[
\|(a_t)_*w\|_{\AGY, a_tx}\leq \|w\|_{\AGY,x}.
\]
Thus using the definition of $\Sob(\phi)$, we have
\[
|\phi(a_{\mlsign t}z)-\phi(a_{\mlsign t}z^{\unst})|\ll \epsilon^{\ref{a:f-abos-cont}}\Sob(\phi)
\]
where $\ref{a:f-abos-cont}$ and the implied constant depend only on $\alpha$.


In consequence, we may replace $\phi(a_{\mlsign t}z)$ by $\phi(a_{\mlsign t}z^{\unst})$ in~\eqref{eq:em-ineq-1}, and use  
the bound $\|\cdot\|_\infty\leq\Sob(\cdot)$, to conclude the following 
\begin{multline}\label{eq:em-ineq-2}
\left|\int_{\mathsf B_r^{\unst}(x)}\phi(a_{\mlsign t}y)\psi^{\unst}(y)\dif\!\mu_x^{\unst}(y)-\int_{\mathsf {FB}_r(x)}\phi(a_{\mlsign t}z)\Psi(z )\dif\!\mu(z)\right|\ll  \Sob(\phi)\Sob(\Psi)\epsilon^{\ref{a:f-abos-cont}}\quad+\\
\left|\int_{\mathsf B_r^{\unst}(x)}\phi(a_{\mlsign t}y)\psi^{\unst}(y)\dif\!\mu_x^{\unst}(y)-\int_{z\in\mathsf {FB}_r(x)}\phi(a_{\mlsign t}z^{\unst})\Psi(z )\dif\!\mu(z)\right|
\end{multline} 
where the implied constant depends on $\alpha$.

Recall the definition of $\Psi$ from~\eqref{eq:psi-def}, in particular recall the normalizing factor $\lambda_{y^{\unst}}$.
This and the product structure of $\mu$ yield the following  
\begin{align}
\notag\int_{z\in\mathsf {FB}_r(x)}\phi(a_{\mlsign t}z^{\unst})\Psi(z )\dif\!\mu(z)&=\int_{z\in\mathsf {FB}_r(x)}\lambda_{z^{\unst}} \phi(a_{\mlsign t}z^{\unst})\psi^{\rm cs}\bigl(p_{z^{\unst}}^{\rm cs}(z)\bigr)\psi^{\unst}(z^{\unst})\dif\!\mu(z)\\
\notag&=\int_{\mathsf B^{\unst}_r(x)}\phi(a_{\mlsign t}z^{\unst})\psi^{\unst}(z^{\unst})\int_{W^{\rm cs}(z^{\unst})}\lambda_{z^{\unst}}\psi^{\rm cs}\bigl(p_{z^{\unst}}^{\rm cs}(w)\bigr)\dif\!\mu_{z^{\unst}}^{\rm cs}(w)\dif\!\mu_x^{\unst}(z^{\unst})\\
\label{eq:em-ineq-3}&=\int_{\mathsf B_r^{\unst}(x)}\phi(a_{\mlsign t}y)\psi^{\unst}(y)\dif\!\mu_x^{\unst}(y). 
\end{align}

We now combine the estimates in~\eqref{eq:em-ineq-2} and~\eqref{eq:em-ineq-3}, and get the following.
\[
\left|\int_{\mlhoro(x)}\phi(a_{\mlsign t}y)\psi^{\unst}(y)\dif\!\mu_x^{\unst}(y)-\int_{\Amnfld}\phi(a_{\mlsign t}z)\Psi(z )\dif\!\mu(z)\right|\ll 
\Sob(\phi)\Sob(\Psi)\epsilon^{\ref{a:f-abos-cont}}
\]
where the implied constant is absolute.
\end{proof}

\begin{remark}
It is worth mentioning that Proposition~\ref{prop:thickening-smooth}
and its proof hold for any affine invariant manifold, $(\mathcal M,\mu)$. In the sequel, however,
we will only need this result for $\Amnfld$; and even more specifically, in our application to counting problems, we will need this result for 
the principal stratum $\Qalpha$. 
The main result in~\cite{AGY-EM} was generalized to $\Amnfld$ in~\cite{AJ-EMQ}.
\end{remark}

\begin{corollary}\label{cor:thickening-box}
There exist $\consta\label{a:th-box}$, $\consta\label{a:th-box-2}$, and $\constA\label{A:th-box-3}$ so that the following holds.
Let $x,z\in\Amnfld$ and suppose $0<r,r'\leq 0.01\min\{r(x), r(z)\}$. Let $\mathsf B\subset\mathsf B_{r'}(z)$ be so that 
$1_{\mathsf B}\in\mathcal S(z,r')$ and let $\psi^{\unst}\in C_c^\infty(\mathsf B^{\unst}_{r}(x))$.
Then for every $\epsilon<r'/L$, see Lemma~\ref{lem:Margulis-prop}, we have 
\[
\left|\frac{1}{\mu(\mathsf B)}\int_{\mlhoro(x)}1_{\mathsf B}(a_{\mlsign t}y)\psi^{\unst}(y)\dif\!\mu_{x}^{\unst}(y)-\int \psi^{\unst}\dif\!\mu_{x}^{\unst}\right|\leq \epsilon^{-\ref{A:th-box-3}}\Sob(\psi^{\unst})e^{-\ref{a:th-box} t}+\Sob(\psi^{\unst})\epsilon^{\ref{a:th-box-2}}.
\]
\end{corollary}

\begin{proof}
This follows from Proposition~\ref{prop:thickening-smooth} by approximating $1_\mathsf B$
with $\varphi_{\pm,\epsilon}$ and using properties ($\mathcal S$-1)---($\mathcal S$-3). 
\end{proof}

\section{A counting function}\label{sec:mxing-counting}



Let $x,z\in\Amnfld$.
Let $\psi^{\unst}$ be a function which is supported and defined everywhere in $\mathsf B_{0.1r(x)}^{\unst}(x)=\mathsf B_{0.1r(x)}(x)\cap W^{\unst}(x)$, and let $\phi^{\rm cs}$ be a function which is supported and defined everywhere in 
$\mathsf B^{\rm cs}_{0.1r(z)}(z)=\mathsf B_{0.1r(z)}(z)\cap W^{\rm cs}(z)$. For all $t>0$, define
\be\label{eq:num-conn-com}
\ncc(t,\psi^{\unst},\phi^{\rm cs})\index{n@ $\ncc(t,\psi^{\unst},\phi^{\rm cs})$}:=\sum \psi^{\unst}(y)\phi^{\rm cs}(a_ty)
\ee
where the sum is taken over all $y\in\mathsf B_{0.1r(x)}^{\unst}(x)$ so that 
$a_{\mlsign t}y\in \mathsf B_{0.1r(z)}^{\rm cs}(z)$ --- note that the sum is indeed over all $y\in\supp(\psi^\unst)$ so that $a_{\mlsign t}y\in\supp(\phi^{\rm cs})$.

Alternatively, the sum is taken over connected components of $a_t\supp(\psi^{\unst})\cap \supp(\phi^{\rm cs})$ (indeed the subscript ${\rm nc}$ stands for the number of connected components);
this point will be made more explicit later in this section, see e.g.\ Lemma~\ref{lem:wel-def-mod-stab} below and recall that $\mlhoro$ and $W^{\rm cs}$ are complementary foliations.


The function $\ncc$ may be thought of as a bisector counting function where one studies the asymptotic behavior of the number of translates of a piece of $\mlhoro$ by $\MC$ which intersect a cone in the Teichm\"{u}ller space.    

The following proposition is the main result of this section and provides an asymptotic behavior for $\ncc$. This proposition plays a prime role in the proof of Theorem~\ref{thm:main-MLS} in \S\ref{sec:int-pts}. 

\begin{prop}\label{prop:ncc}\label{cor:Nnc}
There exist $\consta\label{a:ncc-smooth}$ and $\constA\label{N:ncc-smooth}$ with the following property. 
Let $x,z\in\Amnfld$, and let $t\geq \ref{N:ncc-smooth}\max\{\log u(x),\log u(z)\}$. 
Let $\psi^{\unst}\in C_c^\infty(\mathsf B_{0.1r(x)}^{\unst}(x))$
with $0\leq \psi^{\unst}\leq 1$, and let $\phi^{\rm cs}\in C_c^\infty(\mathsf B_{0.1r(z)}^{\rm cs}(z))$.
Then
\[
|\ncc\bigl(t,\psi^{\unst},\phi^{\rm cs}\bigr)-e^{ht}\mu_{x}^{\unst}(\psi^{\unst})\mu_z^{\rm cs}(\phi^{\rm cs})|\leq
\Sob(\psi^{\unst})\Sob(\phi^{\rm cs})e^{(h-\ref{a:ncc-smooth}) t}
\]
where $h=\tfrac{1}{2}(\dim_{\bbr}\mathcal Q(\alpha)-2)$.
\end{prop}


The proof of this proposition is based on Lemma~\ref{lem:ncc} which in turn relies on Proposition~\ref{prop:thickening-smooth}.
In particular, the main term is given by Proposition~\ref{prop:thickening-smooth}. 
However, we need to control the contribution of two types of exceptional points as we now describe. 

Similar to Lemma~\ref{cor:nondiv-horosphere}, given a compact subset $K\supset K_\alpha$, define
\be\label{eq:HtxK}
{\mathsf H}_t^{\unst}(x, K)\index{h@ ${\mathsf H}_t^{\unst}(x, K)$}:=\bigl\{y\in\mathsf B_r^{\unst}(x): 
|\{\tau\in[0,t]:a_{\tau}y\in K\}|\geq t/2\bigr\}.
\ee
The first (and more difficult to control) type of exceptional points
are $y\in \mathsf B_r^{\unst}(x)$ so that $a_{\mlsign t}y\in\mathsf B_{r'}(z)$, however, $y\not\in {\mathsf H}_t^{\unst}(x, K)$.
The contribution coming from these points is controlled using~\cite[Thm.\ 1.7]{EMR-Counting}, 
see Theorem~\ref{thm:excep-traj} below. 

We also need to control the contribution of points $y\in \mathsf B_r^{\unst}(x)$
which are exponentially close to the boundary of $\mathsf B_r^{\unst}(x)$. 
This set has a controlled geometry, and we use a covering argument and 
Proposition~\ref{prop:thickening-smooth} to control this contribution. 
The argument here is standard and will be presented after we establish an essential estimate in~\eqref{eq:ncc-4}.

Let us begin with some preliminary statements which are essentially consequences of the fact that $\tmlhoro$ and $\tW^{\rm cs}$ are complimentary foliations in the spaces marked surfaces $\mathcal Q^1\mathcal T(\alpha)$.

Recall that for any $\tilde x\in\mathcal Q^1\mathcal T(\alpha)$, 
$\mathsf B_r^\bullet(\tilde x)$ denotes a ball in $\tW^{\bullet}(\tilde x)$ for $\bullet={\unst, \stbl},{\rm cs}, {\rm cu}$.

\begin{lemma}\label{lem:wel-def-mod-stab}
Let $\tilde x,\tilde x'\in \mathcal Q^1\mathcal T(\alpha)$ and let $0<r\leq 1/2$. 
Assume there are $\tilde y_1,\tilde y_2\in \tmlhoro(\tilde x)$ and some $t\in\reals$
so that $a_{\mlsign t}\tilde y_1$ and $a_t\tilde y_2$ belong to $\tpbox_r^{\rm cs}(\tilde x')$. Then $\tilde y_1=\tilde y_2$.
\end{lemma}

\begin{proof}
We present the argument when $\varsigma=-1$, the other case is similar. 
By the assumption, we have $a_{\mlsign t}\tilde y_i\in\tW^{\rm cs}(\tilde x')$
which implies that 
\[
\text{$\tilde y_i\in\tW^{\rm cs}(\tilde x')$ for $i=1,2$.}
\]
Recall now that $\tilde y_1,\tilde y_2\in\tmlhoro(\tilde x)$, hence, 
by~\eqref{eq:E+-real} the corresponding abelian differentials at $\tilde y_1$ and $\tilde y_2$ 
differ from each other by some $c\in H^1_{\rm odd}(\hat M,\hat \Sigma,\reals)$. 
However, since $\tilde y_1,\tilde y_2\in\tW^{\rm cs}(\tilde x')$, they differ from each other by some 
$c\in H^1_{\rm odd}(\hat M,\hat\Sigma,{\bf i}\reals)\oplus \reals{\bf v}(\tilde x')$.
Therefore, $\tilde y_1=\tilde y_2$.
\end{proof}


\begin{corollary}\label{cor:small-chart-trans}
Let ${\mce}_1,{\mce}_2\in\MC$ be so that $\mce_1\cdot\tmlhoro(\tilde y)=\tmlhoro(\tilde x)=\mce_2\cdot\tmlhoro(\tilde y)$.
Let $\tilde x_1$ and $\tilde x_2$ in $\tmlhoro(\tilde x)$. 
Assume for some $r,b>0$ that 
\[
\text{$\tpbox_r^{\rm cs}(\tilde x')\cap \mce_i\cdot a_{\mlsign t}\tpbox_{b}^{\unst}(\tilde x_i)\neq\emptyset$ for $i=1,2$ and some $t\in\reals$.} 
\]
Then
$\tpbox_r^{\rm cs}(\tilde x')\cap\mce_1\cdot a_{\mlsign t}\tpbox_{b}^{\unst}(\tilde x_1)=\tpbox_r^{\rm cs}(\tilde x')\cap\mce_2\cdot a_{\mlsign t}\tpbox_{b}^{\unst}(\tilde x_2)$.
In particular, we have
\[
\mce_1\cdot\tpbox_{b}^{\unst}(\tilde x_1)\cap\mce_2\cdot\tpbox_{b}^{\unst}(\tilde x_2)\neq\emptyset.
\]
\end{corollary}

\begin{proof}
Let $\tilde y_i\in\tpbox_r^{\rm cs}(\tilde x')\cap \mce_i\cdot a_{\mlsign t}\tpbox_{b}(\tilde x_i)$ for $i=1,2$.
Then $\tilde y_1,\tilde y_2\in \tpbox_r^{\rm cs}(\tilde x')\cap a_t\tmlhoro(\tilde x)$. 
Hence, by Lemma~\ref{lem:wel-def-mod-stab} we have $\tilde y_1=\tilde y_2$ which implies the claim.
\end{proof}

As was discussed above, there are two types of exceptional points. The first type will be controlled using the following theorem. 

\begin{theorem}[Cf.~\cite{EMR-Counting}, Thm.~1.7]\label{thm:excep-traj}
There exist $\constA\label{N:EMR}$ and a compact subset $\bar K_\alpha\supset K_\alpha$ so that
\[
\#\bigl\{y\in\mathsf B_{0.1r(x)}^{\unst}(x)\setminus{\mathsf H}_t^{\unst}(x,\bar K_\alpha): a_{\mlsign t}y\in \mathsf B_{0.1r(z)}^{\rm cs}(z)\bigr\}\ll u(x)^{\ref{N:EMR}}u(z)^{\ref{N:EMR}}e^{(h-0.5)t}
\]
where the implied constant is absolute. 
\end{theorem}

\begin{proof}
Let us write $r=0.1r(x)$ and $r'=0.1r(z)$.
For a compact subset $K\supset K_\alpha$, put 
\[
\ER_t(x,K)\index{e@ $\ER_t(x,K)$}:=\{y\in\mathsf B_{2r}^{\unst}(x)\setminus{\mathsf H}_t^{\unst}(x,K): a_{\mlsign t}y\in \mathsf B_{r'}^{\rm cs}(z) \}.
\]
In $\tAmn$ fix lifts $\mathsf B_{2r}^{\unst}(\tilde x)$ and $\mathsf B_{r'}(\tilde z)$  
for the sets $\mathsf B_{2r}^{\unst}(x)$ and $\mathsf B_{r'}^{\rm cs}(z)$, respectively. 
For every element $y\in \mathsf B_{2r}^{\unst}(x)$ we fix a lift $\tilde y\in\mathsf B_{2r}^{\unst}(\tilde x)$. 
Then for every $y\in\ER_t(x,K)$ there exists some $\mce_y\in\MC$ and some $\tilde z_y\in\mathsf B_{r'}^{\rm cs}(\tilde z)$
so that $a_{\mlsign t}\tilde y={\mce}_y\tilde z_y $.

Recall from Lemma~\ref{lem:AGY-Moduli} that the diameter of 
$\mathsf B_{r(q)}(\tilde q)$ in the Teichm\"{u}ller metric is at most $1$ for all $\tilde q$. 
Hence, for every $y\in\ER_t(x,K)$ we have
\begin{enumerate}
\item $\tilde y$ is within Teichm\"{u}ller distance $1$ from $\tilde x$ and 
$a_{\mlsign t}\tilde y={\mce}_y\tilde z_y $ is within Teichm\"{u}ller distance $1$ of ${\mce}_y \tilde z$, and
\item $|\{\tau\in[0,t]:\pi(a_{\tau}\tilde y)\in K)\}|< t/2$.
\end{enumerate}

It is shown in~\cite[Thm.~1.7]{EMR-Counting}, see also~\cite{EsMir-Counting}, 
that there exists some $K_0$
so that if $K\supset K_0$, then the number of $\{ \mce\tilde z\}$ for which such a $\tilde y$ exists is
\[
\ll u(x)^\star u(z)^\star e^{(h-0.5)t}
\] 
where the implied constant is absolute --- indeed apply with $\delta=0.1$ and $\theta=0.9$ and observe that the function
$G$ in~\cite[Thm.~1.7]{EMR-Counting} is dominated by our function $u$ here. 

We now claim that there exists some $C$ which depends on $\alpha$ and
$K$ so that the following holds: the map $y\mapsto {\mce}_y\tilde z$ from $\ER_t(x,K)$ to
$\{\mce\tilde z:\mce\in\MC\}$ is at most $C$-to-one.

First note that the above discussion together with the claim implies that
\be\label{eq:EMR}
\#\ER_t(x,K)\ll_Cu(x)^\star u(z)^\star e^{(h-0.5)t}, 
\ee  
as we wanted to show.

To see the claim, let $y_1,y_2\in\ER_t(x,K)$. Then there exists ${\mce}_1,{\mce}_2\in\MC$ so that
\[
\mce_i\cdot a_{\mlsign t}\tilde y_i\in\mathsf B_{r'}^{\rm cs}(\tilde z).
\]
Therefore, by Corollary~\ref{cor:small-chart-trans}, applied with $\tilde x_i=\tilde x$ and $b=2r$,
we have 
\begin{itemize}
\item either $\mce_1\cdot\mlhoro(\tilde x)\neq \mce_2\cdot\mlhoro(\tilde x)$ 
which in particular implies that ${\mce}_1\neq {\mce}_2$,
\item or $\mce_1\cdot\tpbox_{2r}^{\unst}(\tilde x)\cap\mce_2\cdot\tpbox_{2r}^{\unst}(\tilde x)\neq\emptyset$
which implies ${\mce}_1^{-1}{\mce}_2$ belongs to a fixed finite subset of $\MC$.
\end{itemize}
The claim thus follows and the proof is complete.
\end{proof}

The following lemma will play a crucial role in the proof of Proposition~\ref{prop:ncc}.

\begin{lemma}\label{lem:ncc}
There exists $\consta\label{a:lem-nnc}$ and $\constA\label{N:lem-nnc}$ with the following property. 
Let $x,z\in\Amnfld$, and let $t\geq \ref{N:lem-nnc}\max\{\log u(x),\log u(z)\}$. Let 
\begin{itemize}
\item $\psi^{\unst}\in C_c^\infty(\mathsf B_{0.1r(x)}^{\unst}(x))$ with $0\leq \psi^{\unst}\leq 1$, and 
\item $\phi^{\unst}\in C_c^\infty(\mathsf B_{0.1r(z)}^{\unst}(z))$ and $\phi^{\rm cs}\in C_c^\infty(\mathsf B_{0.1r(z)}^{\rm cs}(z))$. 
\end{itemize}
Put $\phi(y):=\phi^{\rm cs}(p^{\rm cs}_{y^{\unst}}(y))\phi^{\unst}(y^{\unst})$, see~\S\ref{sec:exp-mix}. Define 
\be\label{eq:ncc-r'}
\ncc'(t,\psi^{\unst},\phi)\index{n@ $\ncc'(t,\psi^{\unst},\phi)$}:=\sum \psi^{\unst}(y)\mu_{a_ty}^{\unst}(\phi)
\ee
where the sum is taken over all $y\in\mathsf B_{r}^{\unst}(x)$ so that $a_{\mlsign t}y\in\mathsf B_{r'}^{\rm cs}(z)$.
Then 
\[
|\ncc'(t,\psi^{\unst},\phi)-e^{ht}\mu_{x}^{\unst}(\psi^{\unst})\mu(\phi)|\leq
\Sob(\psi^{\unst})\Sob(\phi)e^{(h-\ref{a:lem-nnc}) t}
\]
where $h=\tfrac{1}{2}(\dim_{\bbr}\mathcal Q(\alpha)-2)$.
\end{lemma}

\begin{proof}
We will compute 
\[
\int_{\mlhoro(x)}\phi(a_{\mlsign t}y)\psi^{\unst}(y)\dif\!\mu_{x}^{\unst}(y)
\]
in terms of $\ncc'$. The claim will then follow from Proposition~\ref{prop:thickening-smooth}.

Let us write $r=0.1r(x)$ and $r'=0.1r(z)$. 
First note that 
\be\label{eq:Jac-diam}
r'\ll{\rm diam}\bigl(\mlhoro(z')\cap\mathsf B_{r'}(z)\bigr)\ll r'
\ee
where the diameter, ${\rm diam}$, is measured with respect to $\|\;\|_{z',{\AGY}}$ for all $z'\in\mathsf B_{r'}(z)$, see~\cite[Prop.~5.3]{AG-Eigen}.

Let $\bar K_\alpha$ be given by Theorem~\ref{thm:excep-traj} and put ${\mathsf H}_t^{\unst}(x)\index{h@ ${\mathsf H}_t^{\unst}(x)$}:={\mathsf H}_t^{\unst}(x,\bar K_\alpha)$, see~\eqref{eq:HtxK} for the notation.
Since $K_\alpha\subset\bar K_\alpha$, it follows from Lemma~\ref{cor:nondiv-horosphere} that 
\be\label{eq:Ht-large}
\mu_{x}^{\unst}\left(\mathsf B_r^{\unst}(x)\setminus {\mathsf H}_t^{\unst}(x)\right)
\leq e^{-\ref{a:horo-return}t}\mu_{x}^{\unst}(\mathsf B_r^{\unst}(x))
\ee
for every $t\geq t_0$ where $t_0$ depends only on $K_\alpha$.

It is more convenient for the proof to treat points in ${\mathsf H}_t^{\unst}(x)$ 
which are {\em too} close to the boundary of $\mathsf B_r^{\unst}(x)$ separately. Define
\[
\Hgood\index{h@ $\Hgood$}:=\{y\in{\mathsf H}^{\unst}_t(x): \mathsf B^{\unst}_{10e^{-\ref{a:half-unif-hyp}t}}(y)\subset\mathsf B^{\unst}_r(x)\}
\]
where $\consta\label{a:half-unif-hyp}:=\ref{a:unif-hyp}(\bar K_\alpha)/2$, 
see Proposition~\ref{prop:unif-hyp} for the definition of $\ref{a:unif-hyp}$. 
The precise radius which is used in the definition of $\Hgood$ is motivated by estimates for uniform hyperbolicity of the Teichm\"{u}ller geodesic flow, see Claim~\ref{claim:Cy-ball} below.

Using~\eqref{eq:Ht-large} and the definition of $\Hgood$ we have 
\be\label{eq:Hgood-large}
\mu_{x}^{\unst}\left(\mathsf B_r^{\unst}(x)\setminus \Hgood\right)
\leq e^{-\ref{a:Hgood} t}\mu_{x}^{\unst}(\mathsf B_r^{\unst}(x))
\ee
for some $\consta\label{a:Hgood}$ depending on $\bar K_\alpha$.
The estimate in~\eqref{eq:Hgood-large} implies the following:
\begin{multline}\label{eq:ncc-2}
\int_{\mlhoro(x)}\phi(a_{\mlsign t}y)\psi^{\unst}(y)\dif\!\mu_{x}^{\unst}(y)=O( e^{-\ref{a:Hgood} t})\mu_{x}^{\unst}(\mathsf B_r^{\unst}(x))\Sob(\psi^{\unst})\Sob(\phi)+\\\int_{\Hgood} \phi(a_{\mlsign t}y)\psi^{\unst}(y)\dif\!\mu_{x}^{\unst}(y).
\end{multline}

We now compute the term $\int_{\Hgood} \phi(a_{\mlsign t}y)\psi^{\unst}(y)\dif\!\mu_{x}^{\unst}(y)$ appearing in~\eqref{eq:ncc-2}.

For every $y\in\Hgood$ so that $a_{\mlsign t}y\in\mathsf B_r(z)$, 
there is an open neighborhood $\mathsf C_y$ of $y$ such that $a_{\mlsign t}\mathsf C_y$ 
is a connected component of $a_{\mlsign t}\mathsf B_r^{\unst}(x)\cap\mathsf B_{r'}(z)$ containing $a_{\mlsign t}y$.
We note that $\mathcal C=\{\mathsf C_y\}$ is a disjoint collection of open subsets in $\mathsf B_r^{\unst}(x)$. 
Further, in view of~\eqref{eq:Margulis-measure} we have
\be\label{eq:ncc-1}
\mu_{a_{\mlsign t}y}^{\unst}(\phi)=e^{ht}\mu_y^{\unst}\bigl(a_{- t}\phi\bigr)=e^{ht}\mu_{x}^{\unst}\bigl(a_{- t}\phi\bigr);
\ee
recall that $a_{-t}\phi(y')=\phi(a_ty')$.


\begin{claim}\label{claim:Cy-ball}
Let $y\in{\mathsf H}^{\unst}_t(x)$, then $\mathsf C_y\subset \mathsf B^{\unst}_{10e^{-\ref{a:half-unif-hyp}t}}(y)$. 
If we further assume that $y\in\Hgood$, then $\mathsf C_y\subset\mathsf B^{\unst}_{10e^{-\ref{a:half-unif-hyp}t}}(y)\subset \mathsf B_r^{\unst}(x)$.
\end{claim}

\begin{proof}[Proof of the claim]
Let $y'\in\mathsf C_y$.
It follows from the definition of $\mathsf C_y$ that $a_ty'\in \mlhoro(a_ty)\cap\mathsf B_{r'}(z)$.
Let us write $a_ty'=\Phi^{-1}(\Phi(a_ty)+w)$. 
Then, by~\eqref{eq:Jac-diam} we have 
\[
\|w\|_{\AGY, a_ty}\ll r'.
\] 
This, in view of Corollary~\ref{cor:unif-hyp'}, implies that
\[
\|w\|_{\AGY, y}\leq e^{-\ref{a:unif-hyp}t}\|w\|_{\AGY, a_ty}\ll e^{-\ref{a:unif-hyp}t}r'
\]
where the implied constant depends only on $\alpha$.
The claim follows from this estimate if we assume $t$ is large enough so that the above estimate implies  
\[
\|w\|_{\AGY, y}<e^{-\ref{a:half-unif-hyp}t};
\]
recall that $\ref{a:half-unif-hyp}=\ref{a:unif-hyp}/2$. The final claim follows from the definition of $\Hgood$.
\end{proof}

Claim~\ref{claim:Cy-ball} in particular implies that
\be\label{eq:MVT}
|\psi^{\unst}(y)-\psi^{\unst}(y')|\ll e^{-\ref{a:half-unif-hyp}t}\Sob(\psi^{\unst})\quad\text{ for all $y'\in \mathsf C_y$}
\ee
where the implied constant depends only on $\alpha$.

Returning to~\eqref{eq:ncc-2}, we get from~\eqref{eq:ncc-1} and~\eqref{eq:MVT} that
\be\label{eq:ncc-3}
\int_{\Hgood} \phi(a_{\mlsign t}y)\psi^{\unst}(y)\dif\!\mu_{x}^{\unst}(y)=O(e^{-\ref{a:half-unif-hyp}t})\Sob(\psi^{\unst})\Sob(\phi)+
e^{-ht}\sum_{\mathsf C_y\in\mathcal C} \psi^{\unst}(y)\mu_{a_{\mlsign t}y}^{\unst}(\phi).
\ee

Combining~\eqref{eq:ncc-2},~\eqref{eq:ncc-3}, and Proposition~\ref{prop:thickening-smooth}  we get the following
\be\label{eq:ncc-4}
\bigl|\sum_{\mathcal C} \psi^{\unst}(y)\mu_{a_{\mlsign t}y}^{\unst}(\phi)-e^{ht}\mu_{x}^{\unst}(\psi^{\unst})\mu(\phi)\bigr|\leq
\Sob(\psi^{\unst})\Sob(\phi)e^{(h-\ref{a:ncc-1}) t}
\ee
for some $\consta\label{a:ncc-1}$ depending on $\alpha$.
Thus, in order to get the conclusion, we need to control the difference between $\ncc'(t,\psi^{\unst},\phi)$ 
and the summation appearing on the left side of~\eqref{eq:ncc-4}. That is: the contribution of points $y\notin\Hgood$.

{\bf Contribution from points in ${\mathsf H}_t^{\unst}(x)$ which are not in $\Hgood$.}
Let $y\in{\mathsf H}_t^{\unst}(x)\setminus\Hgood$ be so that $a_{\mlsign t}y\in\mathsf B_r(z)$.
We note that
$\mathsf C_y$ is not necessarily contained in $\mathsf B^{\unst}_r(x)$;
however, in view Claim~\ref{claim:Cy-ball}, we have $\mathsf C_y$ is contained 
in $\mathsf B_{10e^{-\ref{a:half-unif-hyp}t}}(y)$.

The following is a consequence of the definition.
\[
\bigcup_{y\in{\mathsf H}_t^{\unst}(x)\setminus\Hgood}\mathsf B^{\unst}_{10e^{-\ref{a:half-unif-hyp}t}}(y)\subset \mathsf B_{r+O(e^{-\ref{a:half-unif-hyp}t})}^{\unst}(x)\setminus\mathsf B_{r-O(e^{-\ref{a:half-unif-hyp}t})}^{\unst}(x)=:\mathsf G(x)
\]
where the implicit multiplicative constant depends only on $\alpha$.

Let $0<\hat\kappa<\ref{a:half-unif-hyp}$ be a small constant which will be optimized later, and let 
$t\geq \frac{2\ref{A:AGY-Mod}\log u(x)}{\hat\kappa}$. 
We can cover $\mathsf G(x)$ with period balls $\{\mathsf B(y_i):1\leq i\leq I\}$
centered at $y_i$ and of radius $e^{-\hat\kappa t}$ with multiplicity bounded by $\ll e^{\ref{A:char-sob}\hat\kappa t}$, see~\cite[Lemma 1.4.9]{Hor} and also~\S\ref{sec:smooth-structure}.
We have
\be\label{eq:I-bound}
I\ll e^{N\hat\kappa t}
\ee
for some $N$ depending only on $\alpha$.  

For every $i$, let ${\hat{\mathsf B}}(y_i)$ denote the the period ball with the same center $y_i$
and with radius $0.04e^{-\hat\kappa t}$. Note that since $\hat\kappa<\ref{a:half-unif-hyp}=\ref{a:unif-hyp}/2$ we have
\[
2e^{-{\hat\kappa}t}> e^{-{\hat\kappa}t}+{10e^{-\ref{a:unif-hyp}t}}.
\] 
Therefore, $\cup_i {\hat{\mathsf B}}(y_i)$ covers a set $\mathsf G'(x)\supset\mathsf G(x)$ so that 
$\mu_x^{\unst}(\mathsf G'(x))\ll e^{-{\hat\kappa}t}$.

Let $0\leq \hat{\psi^{\unst}}_i\leq 1$ be a smooth function which is supported in ${\hat{\mathsf B}}^{\unst}(y_i)$
which equals $1$ on $\mathsf B_{2e^{-{\hat\kappa}t}}^{\unst}(y_i)$ so that
\be\label{eq:sob-tpsi}
\Sob(\hat\psi^{\unst}_i)\leq e^{\ref{A:char-sob}{\hat\kappa}t}\qquad\text{and}\qquad\sum \hat\psi^{\unst}_i\leq \mathbbm 1_{\mathsf G'(x)},
\ee
where $\ref{A:char-sob}\geq \ref{A:part-unity-mult}$ is chosen to account for the multiplicative constant in Lemma~\ref{lem:partition-unity-AG}.

Let $\mathcal I_i$ be the contribution coming from $\mathsf B(y_i)$ to 
$\ncc(t,\psi^{\unst},\phi)$. Then arguing as above and using Proposition~\ref{prop:thickening-smooth}, the choice of 
$\hat\psi^{\unst}$ implies that
\begin{equation}\label{eq:ncc-easy}
\mathcal I_i\leq e^{ht}\int_{\mlhoro(x)}\phi(a_{\mlsign t}y)\hat\psi^{\unst}_i(y)\dif\!\mu_{x}^{\unst}(y)\leq e^{ht}\mu(\phi)\int \hat\psi^{\unst}_i\dif\!\mu_{x}^{\unst}+\Sob(\hat\psi^{\unst}_i)\Sob(\phi)e^{(h-\ref{a:th-smooth}) t}
\end{equation}
Summing~\eqref{eq:ncc-easy} over all $1\leq i\leq I$ and using~\eqref{eq:sob-tpsi},~\eqref{eq:I-bound},
and $\int\hat\psi^{\unst}\dif\!\mu_{x}^{\unst}\ll e^{-h{\hat\kappa}t}$
we get
\begin{align*}
\sum_i\mathcal I_i&\ll e^{ht}\mu_x^{\unst}(\mathsf G'(x))+e^{N\hat\kappa t}\Sob(\phi)e^{(h-\ref{a:th-smooth}+\ref{A:char-sob} {\hat\kappa}) t}\\
&=e^{(h-{\hat\kappa})t}+\Sob(\phi)e^{(h-\ref{a:th-smooth}+(N+\ref{A:char-sob}) {\hat\kappa}) t}.
\end{align*}
Therefore, we can choose ${\hat\kappa}$ so that the above upper bound yields 
\be\label{eq:ncc-easy-1}
\sum_i\mathcal I_i\ll \Sob(\phi)\Sob(\psi^{\unst})e^{(h-\ref{a:ncc-easy})t}
\ee
for some $\consta\label{a:ncc-easy}$ depending only on $\alpha$.

{\bf Contribution from points in $\mathsf B_r^{\unst}(x)\setminus{\mathsf H}_t^{\unst}(x)$.}
Let $\mathcal J$ denote the contribution  to $\ncc'(t,\psi^{\unst},\phi)$ coming from points
$y\in \mathsf B_r^{\unst}(x)\setminus{\mathsf H}_t^{\unst}(x)$. Then there is a unique 
$z_y\in \mathsf B_{r+r'}^{\unst}(x)\setminus{\mathsf H}_t^{\unst}(x,\bar K_\alpha)$ such that $a_{\mlsign t}z_y\in \mathsf B_{r'}^{\rm cs}(z)$. In consequence, by Theorem~\ref{thm:excep-traj}, we have
\[
\mathcal J\ll u(x)^{\ref{N:EMR}}u(z)^{\ref{N:EMR}}\|\phi\|_\infty\|\psi^{\unst}\|_\infty e^{(h-0.5)t}\ll u(x)^{\ref{N:EMR}}u(z)^{\ref{N:EMR}}\Sob(\phi)\Sob(\psi^{\unst})e^{(h-0.5)t}.
\]
Assuming $t\gg \max\{\log u(x),\log u(z)\}$, the above implies
\be\label{eq:ncc-emr}
\mathcal J\leq \Sob(\phi)\Sob(\psi^{\unst})e^{(h-0.6)t}.
\ee   

The proposition now follows from~\eqref{eq:ncc-4} in view of~\eqref{eq:ncc-easy-1}
and~\eqref{eq:ncc-emr}.
\end{proof}


%

\begin{proof}[Proof of Proposition~\ref{prop:ncc}]
Let $\varrho=e^{-\kappa t}$ and let $\epsilon=\varrho^N$, for two constants $\kappa, N>0$ which will be optimized later. 
Put $\phi=1_{\mathsf B^{\unst}_\varrho(z)}\phi^{\rm cs}$. Then
\be\label{eq:mu-phi}
\mu(\phi)=\varrho^h\mu_z^{\rm cs}(\phi^{\rm cs})
\ee

In view of Lemma~\ref{lem:Margulis-prop}, properties ($\mathcal S$-1), ($\mathcal S$-2), and ($\mathcal S$-2) hold with 
$\epsilon$ and $f=1_{\mathsf B_{\varrho-2\epsilon}^{\unst}(z)}$. Let $\phi_1^{\unst}=\varphi_{+,\epsilon}$ for these choices.
Put $\phi_1=\phi_1^{\unst}\phi^{\rm cs}$; there exists some $\consta\label{a:vro-ep}$ so that
\be\label{eq:vro-ep}
\mu(\phi_1)-\mu(\phi)\leq \epsilon^{\ref{a:vro-ep}}\mu_z^{\rm cs}(\phi^{\rm cs}).
\ee 

By Lemma~\ref{lem:ncc}, we have 
\begin{align}
\notag\ncc'(t,\psi^{\unst},\phi_1)&=e^{ht}\mu_{x}^{\unst}(\psi^{\unst})\mu(\phi_1)+O(\Sob(\psi^{\unst})\Sob(\phi_1)e^{(h-\ref{a:lem-nnc}) t})\\
\notag{}^{\text{\eqref{eq:vro-ep}}\leadsto}&=e^{ht}\mu_{x}^{\unst}(\psi^{\unst})\mu(\phi)+O(\epsilon^{\ref{a:vro-ep}}e^{ht}\mu_{x}^{\unst}(\psi^{\unst})+\Sob(\psi^{\unst})\Sob(\phi^{\rm cs})\epsilon^{-\star}e^{(h-\ref{a:lem-nnc}) t})\\
\label{eq:Ball-phi-1}{}^{\text{\eqref{eq:mu-phi}}\leadsto}&=e^{ht}\varrho^h\mu_{x}^{\unst}(\psi^{\unst})\mu_z^{\rm cs}(\phi^{\rm cs})+
O(\epsilon^{\ref{a:vro-ep}}e^{ht}\mu_{x}^{\unst}(\psi^{\unst})+\Sob(\psi^{\unst})\Sob(\phi^{\rm cs})\epsilon^{-\star}e^{(h-\ref{a:lem-nnc}) t}).
\end{align}

Let now $\phi_2^{\unst}=\varphi_{+,\epsilon}$ for $\epsilon$ and $f=1_{\mathsf B_{\varrho}^{\unst}(z)}$.
Put $\phi_2=\phi_2^{\unst}\phi^{\rm cs}$. 
Then similar to the above estimate, using Lemma~\ref{lem:ncc}, we get that  
\be\label{eq:Ball-phi-2}
\ncc'(t,\psi^{\unst},\phi_2)=e^{ht}\varrho^h\mu_{x}^{\unst}(\psi^{\unst})\mu_z^{\rm cs}(\phi^{\rm cs})+
O(\epsilon^{\ref{a:vro-ep}}e^{ht}\mu_{x}^{\unst}(\psi^{\unst})+\Sob(\psi^{\unst})\Sob(\phi^{\rm cs})\epsilon^{-\star}e^{(h-\ref{a:lem-nnc}) t}).
\ee 
Since $\phi_1\leq \phi\leq\phi_2$, we have  
\be\label{eq:ncc-phi-phii}
\ncc(t,\psi^{\unst},\phi_1)\leq \ncc'(t,\psi^{\unst},\phi)\leq \ncc(t,\psi^{\unst},\phi_2).
\ee
Moreover, using the definitions of $\ncc$ and $\ncc'$ we have 
\begin{align*}
\ncc'(t,\psi^{\unst},\phi)&=\sum\psi^{\unst}(y)\mu_{a_ty}^{\unst}(\phi)\\
&=\sum\psi^{\unst}(y)\phi^{\rm cs}(a_ty)\mu_z^{\unst}(B_\varrho^{\unst}(z))=\varrho^{h}\sum\psi^{\unst}(y)\phi^{\rm cs}(a_ty)\\
&=\varrho^{h}\ncc(t,\psi^{\unst},\phi^{\rm cs}).
\end{align*}
This and~\eqref{eq:ncc-phi-phii} imply that 
\[
\varrho^{-h}\ncc(t,\psi^{\unst},\phi_1)\leq\ncc(t,\psi^{\unst},\phi^{\rm cs})\leq\varrho^{-h}\ncc(t,\psi^{\unst},\phi_1).
\]
Hence, using~\eqref{eq:Ball-phi-1} and~\eqref{eq:Ball-phi-2}, we get that
\[
\ncc(t,\psi^{\unst},\phi^{\rm cs})=
e^{ht}\mu_{x}^{\unst}(\psi^{\unst})\mu_z^{\rm cs}(\phi^{\rm cs})+
O(\varrho^{-h}\epsilon^{\ref{a:vro-ep}}e^{ht}\mu_{x}^{\unst}(\psi^{\unst})+\Sob(\psi^{\unst})\Sob(\phi^{\rm cs})\epsilon^{-\star}e^{(h-\ref{a:lem-nnc}) t}).
\]
We choose $N$ large enough so that $\ref{a:vro-ep}N-h>\ref{a:vro-ep}N/2$ then choose 
$\kappa$ small enough so that $\epsilon^{-\star}e^{(h-\ref{a:lem-nnc}) t}=e^{(h-\ref{a:lem-nnc}/2)t}$.
The proof is complete.
\end{proof}

We end this section with the following corollary.

\begin{cor}\label{cor:nnc-notsmooth}
There exist $\consta\label{a:ncc-ns-1}$, $\consta\label{a:ncc-ns-2}$, 
and $\constA\label{A:ncc-ns}$ with the following property. 
Let $x,z\in\Amnfld$. Let $\psi^{\unst}\in C_c^\infty(\mathsf B_{0.1r(x)}^{\unst}(x))$
with $0\leq \psi^{\unst}\leq 1$ and let $\phi^{\rm cs}\in\mathcal S_{W^{\rm cs}(z)}(z,0.1r(z))$.
Then for all $\delta<r(z)/10L$ and all $t\geq \ref{N:ncc-smooth}\max\{\log u(x),\log u(z)\}$ we have 
\[
|\ncc\bigl(t,\psi^{\unst},\phi^{\rm cs}\bigr)-e^{ht}\mu_{x}^{\unst}(\psi^{\unst})\mu_z^{\rm cs}(\phi^{\rm cs})|\ll
\Sob(\psi^{\unst})\delta^{-\ref{A:ncc-ns}}e^{(h-\ref{a:ncc-ns-1}) t}+\delta^{\ref{a:ncc-ns-2}}\Sob(\psi^{\unst})e^{ht}
\]
where $h=\tfrac{1}{2}(\dim_{\bbr}\mathcal Q(\alpha)-2)$.

In particular, there exists $\consta\label{a:ncc}$, depending only on $\alpha$, so that the following holds. 
Assume further that $t\geq 2|\log r(z)|=2\ref{A:AGY-Mod}\log u(z)$, see~\eqref{eq:def-r(x)}, then 
\be\label{eq:ncc-ns-final}
|\ncc\bigl(t,\psi^{\unst},\phi^{\rm cs}\bigr)-e^{ht}\mu_{x}^{\unst}(\psi^{\unst})\mu_z^{\rm cs}(\phi^{\rm cs})|\ll\Sob(\psi^{\unst})e^{(h-\ref{a:ncc})t}.
\ee
\end{cor}

\begin{proof}
The corollary follows from Proposition~\ref{prop:ncc} by approximating 
$\phi^{\rm cs}$ with smooth functions. Let $\delta<0.1r(z)/L$ and let $\phi^{\rm cs}_{\pm,\delta}$ be smooth functions 
satisfying ($\mathcal S$-1), ($\mathcal S$-2), and ($\mathcal S$-2) with 
$\delta$ and $\phi^{\rm cs}$. Hence, we have
\be\label{eq:S12-Sob-use}
\text{$\phi^{\rm cs}_{-,\delta}\leq \phi^{\rm cs}\leq \phi^{\rm cs}_{+,\delta}\quad$ and $\quad\Sob(\phi_{\pm,\delta})\ll\delta^{-\star}$;}
\ee
furthermore, property~($\mathcal S$-3) implies that
\be\label{eq:phi-pm-control}
|\mu_z^{\rm cs}(\phi_{+,\delta}^{\rm cs})-\mu_z^{\rm cs}(\phi_{-,\delta}^{\rm cs})|\ll \delta^\star.
\ee

With this notation and in  view of the first estimate in~\eqref{eq:S12-Sob-use}, we have 
\be\label{eq:phi-pm-again}
\ncc\bigl(t,\psi^{\unst},\phi^{\rm cs}_{-,\delta}\bigr)\leq \ncc\bigl(t,\psi^{\unst},\phi^{\rm cs}\bigr)\leq \ncc\bigl(t,\psi^{\unst},\phi^{\rm cs}_{+,\delta}\bigr).
\ee
In addition we may apply Proposition~\ref{prop:ncc} with $\psi^{\unst}$ and $\phi^+_{\pm\delta}$ and get that
\[
\ncc\bigl(t,\psi^{\unst},\phi^{\rm cs}_{\pm,\delta}\bigr)=e^{ht}\mu_{x}^{\unst}(\psi^{\unst})\mu_z^{\rm cs}(\phi^{\rm cs}_{\pm})+
O(\Sob(\psi^{\unst})\Sob(\phi_{\pm,\delta}^{\rm cs})e^{(h-\ref{a:ncc-smooth}) t}.
\]
This together with~\eqref{eq:phi-pm-again},~\eqref{eq:phi-pm-control}, and the second estimate in~\eqref{eq:S12-Sob-use} implies the first claim.

The second claim follows from the first claim by optimizing the choice $\delta=e^{-\star t}$.  
\end{proof}

\section{The space of measured laminations}\label{sec:tt}

In this section we recall some basic facts about the space of geodesic measured laminations and train track charts.  
The basic references for these results are \cite{Thurston:book:GTTM} and \cite{HP}. 

The space of geodesic measured laminations on $S$ is denoted by $\mathcal{ML}(S)$; 
it is a piecewise linear manifold homeomorphic to 
${\Bbb R}^{6g-6}$, but it does not have a natural differentiable structure \cite{Thurston:book:GTTM}. 
Train tracks were introduced by Thurston as a powerful technical device for understanding measured 
laminations. Roughly speaking train tracks are induced by squeezing almost parallel
strands of a very long simple closed geodesic to simple arcs on a hyperbolic surface.
A train track $\tau$\index{t@$\tau$ a train track} on a surface $S$ is a finite closed 1 complex $\tau \subset S$ with 
vertices (switches) which is 
\begin{itemize}
\item[-] embedded on $S$, 
\item[-] away from its switches, it is $C^1$, 
\item[-] it has tangent vectors at every point, and 
\item[-] for each component $R$ of $S-\tau$ , the double of $R$ along the interiors of the edges of $\partial(R)$ has negative Euler characteristic.
\end{itemize}
The vertices (or switches), $V$, of a train track are the points where $3$ 
or more smooth arcs come together. Each edge of $\tau$ is a smooth path with a well defined tangent vector. 
That is: all edges at a given vertex are tangent. 
The inward pointing tangent of an edge divides the branches that are incident 
to a vertex into incoming and outgoing branches. 

A train track $\tau$ is called maximal (or generic) if at each vertex there are two incoming edges and one outgoing 
edge.

 \subsection{Train track charts}\label{sec:ttc}
 A lamination $\lambda$\index{l@$\lambda$ a lamination} on $S$ is carried by a train track $\tau$ if there is a differentiable map 
 $f : S\rightarrow S$ so that
 \begin{itemize}
 \item[-] $f$ is homotopic to the identity,  
 \item[-] the restriction of $\dif\!f$ to a tangent line of $\lambda$ is nonsingular, and 
 \item[-] $f$ maps $\lambda$ onto $\tau$.  
 \end{itemize}
 Every geodesic lamination is carried by some train track.
Let $\lambda$ be a measured lamination with invariant measure $\mu$. 
If $\lambda$ is carried by the train track $\tau$, 
then the carrying map defines a counting measure $\mu(b)$ to each branch line $b$: 
$\mu(b)$ is just  the transverse measure of the leaves of $\lambda$ collapsed to a point on $b$. 
 At a switch, the sum of the entering numbers equals the sum of the exiting numbers. 
 
The piecewise linear integral structure on $\ML$ is induced by train tracks as follows.
Let $\cV(\tau)$ be the set of measures on a train track $\tau$; more precisely, 
$u \in \cV(\tau)$ is an assignment of positive real numbers to the edges 
of the train track satisfying the switch condition:
\[
\sum\limits_{\mbox{incoming}\; e_{i}}  u(e_{i})=\sum\limits_{\mbox{ outgoing}\;e_{j}} u(e_{j}).
\] 
Also, let $\cW(\tau)\index{w@ $\cW(\tau)$}$ be the vector space of all real weight systems on edges of $\tau$ satisfying the 
switch condition, i.e., $u(e_{i})$ need not be positive for $u\in \cW(\tau)$.
Then $\cV(\tau)\index{v@ $\cV(\tau)$}$ is a cone on a finite-sided polyhedron where the faces are of the form $\cV(\sigma) \subset \cV(\tau)$ 
where $\sigma$ is a sub train track of $\tau$.

If $\tau$ is {\em bi-recurrent}, then the natural map $\iota_{\tau}:\mathcal V(\tau) \rightarrow \ML$\index{i@$\iota_\tau$ natural embedding of $\mathcal V(\tau)$ in into $\ML$} is continuous and injective, see \cite[\S1.7]{HP}. Let 
\be\label{eq:def-ttc}
\ttc\index{u@$\ttc$ a train track chart in $\ML$}= \iota_{\tau}(\cV(\tau)) \subset \ML.
\ee
Moreover, we have the following.
\begin{lemma}\label{lem:p-linear}
Let $\mathcal U_{1} \subset \cV(\tau_1)$ and $\mathcal U_2 \subset \cV(\tau_2)$ be such that $\iota_{\tau_1}(\cU_1)= \iota_{\tau_2}(\cU_2)$. Then the map $\iota_{\tau_{2}}^{-1} \circ \iota_{\tau_1}: \cU_{1} \rightarrow \cU_2$ 
is a piecewise linear map and hence it is bilipschitz.
\end{lemma}

For the proof see \cite[\S2 and Thm.~3.1.4]{HP}.

\subsection{Thurston symplectic form on $\ML$}\label{sec:thu-symp}

We can identify $\cW(\tau)$ with the tangent space of $\ML$ at a
point $u \in \cV(\tau)$, see \cite{HP}.

For any train track $\tau$, the integral points in $\cV(\tau)$ are in one to one correspondence 
with the set of integral multicurves in $\ttc \subset \ML$.
The natural volume form on $\cV(\tau)$ defines a mapping class group invariant volume form
$\muth$ in the Lebesgue measure class on $\ML$.

In fact, the volume form on $\ML$ is induced by a mapping class group invariant $2$-form 
$\omega$ as follows. 
Suppose $\tau$ is maximal, for $u_{1}, u_{2} \in \cW(\tau)$ the symplectic pairing is defined  as follows.  
\begin{equation}\label{defs}
\omega(u_{1},u_{2})=\frac{1}{2}\Bigl(\sum u_{1}(e_{1}) \;u_{2}(e_{2})-\;u_{1}(e_{2})\; u_{2}(e_{1})\Bigr),
\end{equation}
the sum is over all vertices $v$ of the train track where $e_{1}$ and
$e_{2}$ are the two incoming branches at $v$ such that $e_{1}$ is on the right side of 
the common tangent vector.

This form defines an antisymmetric bilinear form on $\cW(\tau)$.

\begin{lemma}
Let $\tau$ be maximal. The Thurston form $\omega$, defined in~\eqref{defs}, is 
non-degenerate. Therefore it gives rise to a symplectic form on the piecewise linear
manifold $\ML$.
\end{lemma}

See \cite[\S 3]{HP} for a proof and also the relationship between the intersection pairing 
of $H^{1}(S, {\Bbb R})$ and Thurston intersection pairing.

\subsection{Combinatorial type of measured laminations and train tracks}\label{sec:comb-tt}
Each component of $S- \lambda$ is a region bounded by closed geodesics and infinite geodesics; further, any such region
can be doubled along its boundary to give a complete hyperbolic surface which has finite area.

We say a filling measured lamination $\lambda$ is of type ${\bf a}=(a_{1},....a_{\ipt})$ if and only if $S-\lambda$
 consists of ideal polygons with $a_{1},\ldots,a_{\ipt}$ sides. 
 By extending the measured lamination $\lambda$ to a foliation with isolated singularities on the complement, we see that
 $ \sum_{i=1}^{\ipt} a_{i}=4g-4+2\ipt$, see~\cite{Thurston:book:GTTM} and~\cite{Lev}.
 
Similarly, each component of the complement of a filling train track $\tau$  is a non-punctured or once-punctured cusped polygon of negative Euler index.
We say a train track $\tau$ is of type ${\bf a}=(a_{1},\ldots,a_{\ipt})$, if and only if 
$S-\tau$ consists of $k$ polygons with $a_{1},\ldots, a_{k}$ sides.
Every measured lamination of type ${\bf a}=(a_{1},\ldots,a_{\ipt})$ can be 
carried by a train track of type ${\bf a}$. 
 
  \begin{lemma}
 For any filling train track $\tau$ of type ${\bf a}=(a_1,\ldots, a_{\ipt})$ we have
 \[
 \operatorname{dim}(V(\tau))=2g+\ipt-1\quad  \text{ if $\tau$ is orientable;}
 \]
 \[
 \operatorname{dim}(V(\tau))=2g+\ipt-2\quad \text{ if $\tau$  is not orientable.}
 \]
 \end{lemma}

 More generally, a measured lamination $\lambda$ is said to be of type ${\bf a}$ 
 if there exists a quadratic differential $q\in {\mathcal{Q}(a_1-2,\ldots,a_\ipt-2)}$ such that 
 $\lambda= \Rfrak(q)$. It is easy to check that if $\lambda$ is filling, the above can happen only if $S-\lambda$
 consists of ideal polygons with $a_{1},\ldots,a_{\ipt}$ sides.

In general, see~\cite[\S 3]{HP}, we have: 

\begin{prop}\label{prop:qd-tt}
Given a measured lamination $\lambda$ of type ${\bf a}$, there exists a birecurrent train track of type ${\bf a}$ such that $\lambda$ is an interior point of $\ttc$.
\end{prop}

For every ${\bf a}=(a_1,\ldots,a_{\ipt})$ so that  
 $ \sum_{i=1}^{\ipt} a_{i}=4g-4+2\ipt$,
we can fix a collection $\tau_{{\bf a},1},\ldots,\tau_{{\bf a},\ttcn}$ of train tracks with the following property.  
Every $\lambda$ which can be carried by a train track of type ${\bf a}$ can be carried 
by at least one $\tau_{{\bf a},i}$ for some $i$.

\subsection{The Hubbard-Masur map}\label{sec:HubMas} 
Let $\mathcal{MF}(S)$ denote the space of measured foliations on $S$. Define
\[
\tilde{\mathcal P}\index{p@$\tilde{\mathcal P}$ the Hubbard-Masur map}:\mathcal Q\Teich(S)\to\mathcal {MF}(S)\times\mathcal {MF}(S)\setminus\Delta
\]
by $\tilde{\mathcal P}(q)=(\Rfrak(q^{1/2}),\Ifrak(q^{1/2}))$ where 
\[
\Delta=\{(\eta,\lambda):\text{ there exists $\sigma$ so that }\intsc(\sigma,\lambda)+\intsc(\sigma,\eta)=0\}.
\]


\begin{theorem}[Hubbard-Masur, Gardiner]\label{thm:HubMas}
The map $\tilde{\mathcal P}$ is a $\MC$ equivariant homeomorphism. 
\end{theorem}

This gives rise to an equivariant homeomorphism 
from $\mathcal Q\Teich(S)$ onto $\ML\times\ML\setminus\Delta$ which we continue to denote by $\tilde{\mathcal P}$, see~\cite{Thurston:book:GTTM} and~\cite{Lev}.

Recall that $\PML$ denotes the space of {\em projective} measured lamination. The map $\tilde{\mathcal P}$ 
also gives rise to an equivariant homeomorphism 
\[
\hubmas\index{p@$\hubmas$ the normalized Hubbard-Masur map}:\mathcal Q_1\Teich(S)\to\PML\times\ML\setminus\Delta
\] 
where $\hubmas(q)=([\Rfrak(q^{1/2})],\Ifrak(q^{1/2}))$
and $\Delta=\{([\eta],\lambda):\exists\;\sigma$ so that $\intsc(\sigma,\eta)+\intsc(\sigma,\lambda)=0\}$.

Recall that $\pi$ is the natural projection from $\mathcal Q_1\Teich(S)$ to $\mathcal Q_1(S)$, then we have the map
\be\label{eq:def-bhubmas}
\bhubmas:\PML\times\ML\setminus\Delta\to\mathcal Q_1(S).
\ee



\subsection{Convexity of the hyperbolic length function}\label{sec:add-ttc}
Let $\lambda_1,\lambda_2 \in \ttc=\iota_{\tau}(\cV(\tau))$, 
see \S\ref{sec:ttc} for the definition of $\iota_\tau$. 
The sum 
\[
\lambda_1 \oplus_{\tau} \lambda_2=\iota_{\tau} (\iota_{\tau}^{-1}(\lambda_1) {\bf +} \iota_{\tau}^{-1}(\lambda_2))
\]
could depend on $\tau$. 
However, it is proved in~\cite[App.~A]{M-Thesis} that given a closed curve $\gamma$, 
$i(\gamma,.): \ttc \rightarrow {\Bbb R}_{+}$ defines a convex function from which convexity of the hyperbolic length function is drawn in \cite[Thm.~A.1]{M-Thesis}. The following is an extension of \cite[Thm.~A.1]{M-Thesis} 
to the case of variable negative curvature. We are grateful to K.~Rafi for providing the proof of this theorem. 

\begin{theo}\label{thm:int-convex}
Let $\vM$ be a compact surface equipped with a Riemannian metric of negative curvature, and   
let $\tau$ be a train track. Let $\ell_\vM: \ttc \rightarrow {\Bbb R}^{+}$ denote the length function. 
For every pair of measured laminations 
$\lambda_1, \lambda_2 \in \ML$ carried by $\tau$ if $\mu = \lambda_1 \oplus_\tau \lambda_2$, then  
\[
  \ell_\vM(\mu) \leq \ell_\vM(\lambda_1) + \ell_\vM(\lambda_2).
\]
In particular, $\ell_\vM$ is convex.
\end{theo}

The following lemma is well known 

\begin{lemma}\label{lem: it is well known}
Let $\tau$ a train-track, and let $\lambda_1$ and $\lambda_2$ be multi-curves
carried by $\tau$. Then, there exists a multi-curve $\mu$ carried by 
$\tau$ such that $\mu=\lambda_1 + \lambda_2$ in coordinates given by $\tau$. 
Furthermore, $\mu$ can be obtained from $\lambda_1$ and $\lambda_2$ by a 
sequence of surgeries. 
\end{lemma}

We now turn to the proof of the Theorem~\ref{thm:int-convex}. 

\begin{proof}[Proof of Theorem~\ref{thm:int-convex}]
Let $\mathcal C$ be the space of geodesic currents on $\vM$, that is the space 
of $\pi_1(\vM)$-invariant Radon measures on the space of geodesics in $\vM$. 
Recall that the space of measured laminations can be topologically embedded into the 
space of geodesic currents, therefore, we can think of any $\lambda \in \ML$ as
a geodesic current, namely, an element of $\mathcal C$. 
Also recall from \cite{Bonahon:curr} that there is a continuous intersection pairing 
\[
i : \mathcal C \times \mathcal C \to \mathbb R. 
\]
Furthermore, there is a geodesic current $L_\vM \in \mathcal C$ such that for every
\[
i (L_\vM, \lambda) = \ell_\vM(\lambda), \quad \text{ for all } \lambda \in \ML,
\] 
see~\cite{Otal}.
The set of simple closed curves with rational weights are dense in $\ML$. 
Therefore, in view of the continuity of intersection pairing $i$, it is sufficient to check the statement 
of the theorem for rationally weighted simple closed curves only. Since, length
is homogeneous, we can in fact assume the weights are integers or 
$\lambda_1$, $\lambda_2$, and $\mu$ are multi-curves with the possibility 
of some curve appearing more than once. 

\begin{claim}
Assume $\lambda_1$ and $\lambda_2$ are two simple closed curves
with $i(\lambda_1, \lambda_2)>0$. Let $\beta$ be a curve obtained from 
$\lambda_1$ and $\lambda_2$ by a surgery at an intersection point. 
Then, $\ell_\vM(\beta) \leq \ell_\vM(\lambda_1) + \ell_\vM(\lambda_2)$. 
\end{claim}

\begin{proof}[Proof of the claim]
Note that $\lambda_1$ and $\lambda_2$ have unique geodesic representatives in 
$M$. Let $p$ be an intersection point of $\lambda_1$ and $\lambda_2$ where 
the surgery takes place. Then the free homotopy class of $\beta$ can be 
represented by a traversing $\lambda_1$ first (starting from $p$) then $\lambda_2$. 
Which means $\beta$ has a representative whose length is 
$\ell_\vM(\lambda_1) + \ell_\vM(\lambda_2)$. This proves the claim. 
\end{proof} 

Further, we note that, $\mu=\lambda_1 \oplus_\tau \lambda_2$ can be obtained
from $\lambda_1$ and $\lambda_2$ by a sequence of surgery maps, see Lemma~\ref{lem: it is well known}. 
This proves the theorem. 
\end{proof}

Let ${C} \subset {\Bbb R}^{n}$ be a cone and $f: {C} \rightarrow {\Bbb R}$ be a convex function. Let $K$ be a closed and bounded set contained in the relative interior of the domain of $f$. Then $f$ is Lipschitz continuous on $K$. That is: there exists
a constant $\lipc=\lipc(K)$ such that for all $x, y \in K$ we have
\[
|f(x)-f(y)| \leq \lipc |x-y|. 
\] 
Therefore, we have the following. 

\begin{coro}\label{cor:length-lip}
Let $\vM$ be a compact surface equipped with a Riemannian metric of negative curvature. Then
\[
\ell_{\vM}: \ML \rightarrow {\Bbb R}^+
\]
is locally Lipschitz. 
In other words, and in view of the fact that $\ell_\vM(t\,\cdot)=t\ell_\vM(\cdot)$ for all $t>0$, 
we can cover $\ML$ with finitely many cones such that $\ell_{\vM}$ is Lipschitz in each cone. 
\end{coro}

The Lipschitz constant {\em depends} on $\vM$.  See also \cite{LS}.

\section{Linear structure of $\ML$ and $\mathcal Q\mathcal T(S)$}\label{sec:linear-ml-qt}

Our arguments are based on relating the counting problems in $\ML$ to dynamical results in $\Qalpha$.
To that end, we need to compare the linear structure on $\Qalpha$, arising from period coordinates, with the piecewise linear structure on $\ML$, which arises from train track charts. This section establishes required results in this direction.  

From this point to the end of the paper, we will be concerned with the principal stratum, i.e., $\Qalpha$.
Also ${\bf a}=(3,\ldots,3)$ for the rest of the discussion. 

Fix once and for all a collection $\tau_1,\ldots,\tau_{\mathsf c}$ of train tracks 
so that every $\lambda$ can be carried by at least one $\tau_{i}$ for some $i$, see \S\ref{sec:comb-tt}.

Given a point $x=(M,q)\in\Qalpha$ we sometimes use $q$ to denote $x$. 
We fix a fundamental domain for $\Qalpha$,
and unless explicitly stated otherwise, by a lift $\tilde q$ of $q\in\Qalpha$ we mean a representative in this fundamental domain.

Let $x=(M,q)\in\Qalpha$. We denote by $\Rfrak(q^{1/2})$ (resp.\ $\Ifrak(q^{1/2})$) the real (resp.\ imaginary) 
foliation induced by $q$; abusing the notation we will often simply denote these foliations by $\Rfrak(q)$\index{r@$\Rfrak(q^{1/2})$ real foliation of $q$} and $\Ifrak(q)$\index{i@$\Ifrak(q^{1/2})$ imaginary foliation of $q$}. 
Note that $W^{\unst, \stbl}(x)$, which we sometimes also denote by $W^{\unst, \stbl}(q)$, may alternatively be defined as follows.
\[
\mlhoro(q):=\{q'\in\Qalpha: \Ifrak(q')=\Ifrak(q)\}, 
\]
and
$
W^{\stbl}(q):=\{q'\in\Qalpha: \Rfrak(q')=\Rfrak(q)\}.
$

Similarly, we will write $\mathsf B_r(q)$\index{b@$\mathsf B_r(q)$ and $\mathsf B_r^\bullet(q)$} and $\mathsf B_r^\bullet(q)$ for $\mathsf B_r(x)$ and $\mathsf B_r^{\bullet}(x)$, respectively.

Let $\tau$ be a maximal train track, i.e., a train track of type $(3,\ldots,3)$,
and let $\ttc$ be a train track chart, i.e., the set of weights on $\tau$ satisfying the switch conditions. 
Recall from~\S\ref{sec:ttc} that $\ttc$ 
has a linear structure, indeed $\ttc$ is a cone on a finite-sided polyhedron. 
We use the $L^1$-norm on $\cW(\tau)$ to define a norm on $\ttc$.
That is: for every measured lamination $\lambda\in\ttc$, we define $\|\lambda\|_\tau$ \index{n@$\norm{\lambda}_\tau$ and $\norm{\lambda}$ sum of the weights of $\lambda\in U(\tau)$} to be the sum of the weights of $\lambda$. 
Let us define
\be\label{eq:def-P-tau}
\ttco\index{p@$P(\tau)$ polyhedron of laminations whose weights add up to one} :=\{\lambda\in \ttc:\|\lambda\|_\tau=1\}.
\ee

For every $\lambda\in U(\tau)$, define
\[
\bar{\lambda}^{\!\tau}:=\tfrac{1}{\|\lambda\|_\tau}\lambda\in\ttco;
\]
if $\tau$ is fixed and clear from the context, 
we sometimes drop the subscript and the superscript $\tau$ and simply write $\|\lambda\|$ and 
$\bar\lambda$ for $\|\lambda\|_\tau$ and $\bar{\lambda}^{\!\tau},\index{l@ ${\bar{\lambda}}^{\tau}$ and $\bar{\lambda}$}$ respectively. 

By a {\em polyhedron} $\cube\index{u@$\cube$ a polyhedron in $U(\tau)$ of dimension $\dim \ttc-1$}\subset U(\tau)$, 
we mean a polyhedron of dimension $\dim \ttc-1$ 
where the angles are bounded below and the number of facets are bounded, both by constants depending only on the genus.
We will mainly be concerned with $\dim \ttc-1$ dimensional {\em cubes} in the sequel.

\begin{lemma}[Cf.~\cite{LM-Horospherical}, Thm.\ 6.4]\label{lem:ttco-pairing}
Let $\eta\in\ML$ be maximal. 
There is a compact subset $K\subset\Qalpha$, depending on $\tau$ and $\eta$, so that 
$\bhubmas([\eta],\ttco)\subset K$, see~\eqref{eq:def-bhubmas} for the definition of $\bhubmas$. 
\end{lemma}

\begin{proof}
Recall that we fixed a collection $\tau_1,\ldots,\tau_{\mathsf c}$ of train tracks so that 
every lamination $\lambda$ is carried by some $\tau_i$. In view of Lemma~\ref{lem:p-linear}, there exists 
some $L=L(\tau)$ so that
\[
\ttco\subset \bigcup_{i=1}^{\mathsf c}\{\lambda\in U(\tau_i):1/L\leq\|\lambda\|_i\leq L\};
\] 
where $\|\;\|_i=\|\;\|_{\tau_i}$.

For every $1\leq i\leq \mathsf c$, put 
$
U_i:=\{\lambda\in U(\tau_i):1/L\leq\|\lambda\|_i\leq L\}.
$ 
Since $\eta$ is a maximal measured lamination, for any 
$\lambda\in U(\tau_i)$ we have $\bhubmas([\eta],\lambda)\in\Qalpha$.
Define
\be\label{eq:def-K-tt}
K:=\bigcup_{i}\bhubmas(\{[\eta]\}\times U_i).
\ee 
Then $K\subset\Qalpha$ is a compact subset with the desired property.
\end{proof}

\begin{lemma}\label{lem:triangulation}
There is some $\constA\label{A:tgn}\label{A:1ttc}\geq \ref{A:AGY-Mod}$ so that the following holds, see~\eqref{eq:def-r(x)} 
for the definition of $\ref{A:AGY-Mod}$.
Let $q\in\Qalpha$. There exists a 1-complex $T\subset S$ with the following properties. 
\begin{enumerate}
\item Every edge of $T$ is a saddle connection of $q$. 
\item $|\Ifrak(e)|\geq 0.1 \ell_q(e)$ for any $e\in T$.
\item $S\setminus T$ is a union of triangles.
\item For every edge $e\in T$, we have $u(q)^{-\ref{A:tgn}}\leq \ell_{q}(e)\leq u(q)^{-\ref{A:tgn}}$. 
\item We have $u(q)^{-\ref{A:tgn}}\leq r(q)$, moreover, 
the parallel translate of $T$ to $q'\in\mathsf B_{u(q)^{-\ref{A:tgn}}}(q)$ satisfies {\em (1), (2), and (3)} above,
\end{enumerate}
Similar statement holds if we replace $\Ifrak(e)$ in (3) above by $\Rfrak(e)$.
\end{lemma}

\begin{proof}
We will find such a $T$ with $|\Ifrak(e)|\geq0.1 \ell_q(e)$, the proof of the fact that such a $T$ exists with $|\Rfrak(e)|>0.1 \ell_q(e)$
is similar, by replacing $a_tu_s$ with $a_{-t}\bar u_s$ in the following argument.

Let $K$ be the compact set given by Theorem~\ref{thm:fast-return}; 
let $r_0=\inf\{r(x): x\in K\}$, see~\eqref{eq:def-r(x)}. 
For every $q'\in K$, there is a graph $T'\subset S$ of saddle connections of $q'$ so that 
\begin{itemize}
\item the $q'$ length of each of these saddle connections is bounded by $L_0=L_0(K)$, and  
\item $S\setminus T'$ is a union of triangles.
\end{itemize}
We will always assume that $L_0>2$. 
Increasing $L_0$, if necessary, we will also assume that $L_0$ bounds the lengths 
of saddle connections obtained by parallel transporting $T'$ to $q''\in\mathsf B_{r_0}(q')$
for all $q'\in K$.
  
Set $R_q:=\{$saddle connections $\gamma$ of $q$ with $|\Ifrak(\gamma)|<0.1\ell_q(\gamma)\}$.
Note that for all $\gamma\in R_q$ and all $0\leq s\leq 1$, we have $|\Rfrak(u_s\gamma)|\geq \ell_q(\gamma)/2$.
Define the function 
\[
f(q):=\max\{1, \max \{1/\ell_q({\gamma}):\gamma\in R_q\}\}.
\] 

Apply Theorem~\ref{thm:fast-return} with 
$
t_0=L_0\log f(q).
$ 
There exists some
\be\label{eq:fast-return-tgn}
t_0<t\leq \max\{2t_0,\ref{A:fast-return}\log u(q)\}
\ee 
and some $s\in[0,1]$ so that $q'=a_t u_s q\in K$. 

Let now $T'$ be a graph of saddle connections for $q'$ defined as above. 
We claim that for every $e\in T'$, we have $e\not\in a_tu_s R_q$. 
To see the claim, first note that for every $\gamma\in R_q$ we have 
\begin{align*}
\ell_{q'}(a_tu_s\gamma)&\geq e^t\Rfrak(u_s\gamma)\\
&\geq e^{t}\ell_q(\gamma)/2&&\text{$|\Rfrak(u_s\gamma)|\geq \ell_q(\gamma)/2$}\\
&\geq e^{L_0}f(q)\ell_q(\gamma)/2> L_0&&\text{$t>L_0\log f(q)\;\;\&\;\; f(q)\ell_q(\gamma)\geq1$}.
\end{align*} 
Hence $a_tu_s\gamma$ is not contained in $T'$. In consequence, $T=u_{-s}a_{- t}T'$ satisfies (1), (2), (3), and (4). 
Note that for every $e\in T$, we have $u(q)^{-\star}\ll\ell_q(e)\ll u(q)^{\star}$ where the implied constants depend only on the genus.  

We now turn to the proof of part (5). 
First note that there is $N'$ so that $u(q)^{N'}\geq f(q)^{2L_0}$; put 
$N:=\max\{2N',2\ref{A:fast-return}, \ref{A:AGY-Mod}\}$.
Let $\ref{A:tgn}>N$ be so that 
\be\label{eq:choose-M1}
e^{2}\cdot 2^{N-\ref{A:tgn}}\leq r_0/2.
\ee

Let us write $0<r=u(q)^{-\ref{A:tgn}}$, then $0<r\leq r(q)$, recall that $\ref{A:tgn}\geq \ref{A:AGY-Mod}$. 
For every $z\in\mathsf B_r(q)$, we have $z=\Phi^{-1}(\Phi(q')+v)$ where $\|v\|_{\AGY,q}\leq u(q)^{-\ref{A:tgn}}$. 

Let $t\leq \max\{2L_0f(q),\ref{A:fast-return}\log u(q)\}$ and $0\leq s\leq 1$
be so that $q'=a_tu_sq\in K$; see the preceding discussion. 
Note that in view the choice of $t$ and $N$, we have
\be\label{eq:t-est-again}
e^{2t}\leq u(q)^{N}.
\ee
Now for all $v$ so that $\Phi^{-1}(\Phi(q')+v)\in\mathsf B_r(q)$ we have 
\begin{align*}
\|v\|_{\AGY, a_tu_sq}&\leq e^{2+2t}\|v\|_{\AGY, q}
&&\text{by~\eqref{eq:growth-AGY-norm}}\\
&\leq e^2\cdot u(q)^{N}\|v\|_{\AGY,q}
&&\text{by~\eqref{eq:t-est-again}}\\
&\leq e^{2}\cdot u(q)^{N-\ref{A:tgn}}
&&\text{$\|v\|_{\AGY,q}\leq u(q)^{-\ref{A:tgn}}$ by the choice of $r$}\\
&\leq e^{2}\cdot 2^{N-\ref{A:tgn}}\leq r_0/2
&&\text{since $u(q)\geq2$ and using~\eqref{eq:choose-M1}}.
\end{align*}
Hence $a_tu_s\mathsf B_r(q)\subset \mathsf B_{r_0}(q')$ which gives the claim in view of the definitions of $T$ and $T'$.  

Increasing  $\ref{A:tgn}$ if necessary part~(4) also holds for this exponent. 
\end{proof}

\begin{lemma}[Cf.~\cite{M-Earthquake}, Lemma 4.3]\label{lem:box-linear}
Let $q\in\Qalpha$, and let $\tilde q$ be a lift of $q$ in our fixed fundamental domain. Let
$r= 0.01u(q)^{-2\ref{A:tgn}}$, there is a maximal train track $\sigma$ 
with the following properties
\begin{enumerate}
\item $\mathsf B_r(\tilde q)$ projects homeomorphically onto $\mathsf B_r(q)\subset\Qalpha$.
\item The restriction of $\hubmas$ to $\mathsf B_r(\tilde q)$ is a homeomorphism.
\item $\{\Ifrak(\tilde p): \tilde p\in\mathsf B_r(\tilde q)\}$ is contained in one train track chart $U(\sigma)$. 
\item The linear structure on $U_{\Ifrak}(\tilde q):=\{\Ifrak(\tilde p): \tilde p\in\mathsf B_r(\tilde q), \Rfrak(\tilde p)=\Rfrak(\tilde q)\}$ 
as a subset of $U(\sigma)$ agrees with the linear structure on $U_{\Ifrak}(\tilde q)$ which is induced by the restriction of $\hubmas$ to 
$\{\tilde p\in\mathsf B_r(\tilde q): \Rfrak(\tilde p)=\Rfrak(\tilde q)\}\subset W^{\stbl}(\tilde q)$.
\end{enumerate} 
Moreover, the radius $r$ of $\mathsf B_r(\tilde q)$ can be taken to be uniform on compact subsets of $\Qalpha$.   
\end{lemma}

\begin{proof}
Let $T$ be a triangulation of $q$ given by Lemma~\ref{lem:triangulation}.
In particular, 
\begin{enumerate}
\item[(i)] every edge of $T$ is a saddle connection, 
\item[(ii)] $|\Ifrak(e)|\geq 0.1 \ell_q(e)$ for any $e\in T$,
\item[(iii)] $S\setminus T$ is a union of triangles, and
\item[(iv)] $A_q\leq \ell_{q}(e)\leq A_q^{-1}$ for every edge $e\in T$ where $A_q=u(q)^{-\ref{A:tgn}}$
\end{enumerate}
Our construction of the train track $\sigma$ will depend on $T$.

Recall that $r=0.01A_q^{2}$.
Then the balls $\mathsf B_r(\tilde q)$ and $\mathsf B_r(q)$ satisfy (1) and~(2) in the lemma by Lemma~\ref{lem:triangulation}(5).

Let $\sigma'$ be the null-gon dual graph to $T$, in particular, there is one triangle of $\sigma'$ 
in each component of $S\setminus T$.
Let $\sigma$ be the train track obtained from $\sigma'$ 
as follows. If $\Delta$ is a triangle in $T$ with edges $e_1^\Delta,e_2^\Delta, e_3^\Delta$, then there is a
permutation $\{i_1,i_2,i_3\}$ of $\{1,2,3\}$ so that
\be\label{eq:switch-cond}
|\Ifrak(e_{i_1}^\Delta)|=|\Ifrak(e_{i_2}^\Delta)|+|\Ifrak(e_{i_3}^\Delta)|;
\ee
put $\sigma:=\sigma'\setminus \bigcup \{$the edge corresponding to $e_{i_1}^\Delta$ in $\sigma'\}$. 

We claim the lemma holds with $\sigma$. 
To see the claim, first note that $\sigma$ is a maximal train track. 
Assign the weight $|\Ifrak(e_b)|$ to each branch $b\in\sigma$ where $e_b\in T$ 
is the edge which intersects $b$. In view of~\eqref{eq:switch-cond} and the fact that $|\Ifrak({\gamma})|=\intsc({\gamma},\Rfrak(\tilde q))$ for every saddle connection ${\gamma}$, we get that $\lambda=\Ifrak(\tilde q)$ is carried by $\sigma$.
  
By Lemma~\ref{lem:triangulation}, for any $\tilde p\in\mathsf B_r(\tilde q)$ we identify $T$
with its image (under parallel transport) on $\tilde p$. 
Let $\tilde p\in\mathsf B(\tilde q)$ and write $\tilde p=\tilde q+ w$ for some $w$ with $\|w\|_{\AGY,q}\leq 0.01A_q^{2}$. Then 
\[
|\Ifrak({\rm hol}_{\tilde p}(e_b))|=|\Ifrak({\rm hol}_{\tilde q}(e_b))+\Ifrak(w(e_b))|.
\]
Further, we have 
$|w(e_b)|\leq 0.01A_q^{2}\ell_q(e_b)<0.01A_q\leq 0.1|\Ifrak({\rm hol}_{\tilde q}(e_b))|$;
we used (ii) and (iv) in the last inequality. 
Hence, $|\Ifrak({\rm hol}_{\tilde p}(e_b))|>0$ and $\Ifrak(\tilde p)$ is carried by the train track $\sigma$. 

Taking $w\in {\bf i}H_1(M,\Sigma,\bbr)$, the above discussion also implies that $\sigma$ satisfies~(3) and~(4).
\end{proof}

\section{Counting integral points in $\ML$}\label{sec:int-pts}
Let the notation be as in \S\ref{sec:linear-ml-qt}.
In particular, $\tau$ is a maximal train track.
Also recall that $\ttco$ denotes the finite-sided polyhedron in $\ttc$
corresponding to laminations with $\|\lambda\|_\tau=1$.

The smallest $t$ so that a lamination $\lambda\in U(\tau)$ lies in 
\[
[0,e^{t}]\ttco=\{\lambda'\in U(\tau):\|\lambda'\|_\tau\leq e^{t}\}
\] 
can be thought of as a measure of complexity (or length) for the lamination $\lambda$.
In this section we obtain an effective counting result with respect to this complexity. 
In \S\ref{sec:proof-thm} we will use the convexity of the hyperbolic length function in $U(\tau)$ 
to relate the counting problem in Theorem~\ref{thm:main} to this counting problem.

Let $\cube\subset\ttco$ be a cube. For every $t\geq0$, define
\be\label{eq:def-Ip-U}
\Ip(\gamma_0,e^t,\cube)\index{o@$\Ip(\gamma_0,e^t,\cube)$}:=\bigl\{\gamma\in \MC.\gamma_0\cap [0,e^t]\cube\bigr\}.
\ee 

The following strengthening of Theorem~\ref{thm:main-MLS} is the main result of this section.


\begin{theorem}\label{prop:counting}
There exist $\consta\label{a:size-cube}$ and $\consta\label{a:final-U-exp-prop}$ so that the following holds.
Let $t\geq 1$, and let $\cube\subset\ttco$ be a cube of size $\geq e^{-\ref{a:size-cube}t}$. Then  
\[
\#\Ip(\gamma_0,e^t,\cube)={\rm v}(\gamma_0)\mu_{\rm Th}([0,1]\cube){e^{ht}}+ O_{\tau,\gamma_0}(e^{(h-\ref{a:final-U-exp-prop})t})
\]
where ${\rm v}(\gamma_0)$ is defined as in~\eqref{eq:volume-W+} and $h=6g-6$. 
\end{theorem}

The basic tool in the proof of Theorem~\ref{prop:counting} is Proposition~\ref{prop:ncc}. 
We relate the counting problem in Theorem~\ref{prop:counting} 
to a counting problem for translations of $W^{\unst}(q_0)$ in Lemma~\ref{lem:int-horo}.  
Proposition~\ref{prop:ncc} studies a more local version of this latter counting problem. 
That is: one works with translations of a small region in $W^{\unst}(q_0)$. 
Using Corollary~\ref{cor:small-chart-trans}, we will reduce to this local analysis. 
The main step in the proof of Theorem~\ref{prop:counting} is Lemma~\ref{lem:number-of-lifts} below.

Let us begin with some preparation.
Recall that $\ML$ does not have a natural differentiable structure, in particular, $\hubmas$ is only a homeomorphism. 
The situation however drastically improves so long as we restrict to one train track chart and 
fix a transversal lamination. Therefore, we fix a maximal lamination $\eta$ which is carried by $\tau$ for the rest of the discussion.

Let $\delta>0$, and let $\cube\subset\ttco$ be a cube of size $\geq \delta$ centered at $\lambda$. Let $\epsilon\leq \delta$. 
We always assume 
$\hubmas^{-1}$ is a homeomorphism on $\{[\eta]\}\times\{e^r\cube: |r|\leq \delta\}$.
Put $\PU_\cube\index{w@$\PU_\cube$}=\hubmas^{-1}(\{[\eta]\}\times \cube)$ and 
\be\label{eq:def-U-W}
\PU_{\cube,\epsilon}\index{w@$\PU_{\cube,\epsilon}$}=\hubmas^{-1}\bigl(\{[\eta]\}\times\{e^r\cube: -\epsilon\leq r\leq 0\}\bigr).
\ee


Let $\gamma_0\in\ttc$ be a rational multicurve.
For all $t\geq0$ and $0<\epsilon<1$, define 
\be\label{eq:def-Ip}
\Ip(\gamma_0,t,\cube,\epsilon)\index{o@$\Ip(\gamma_0,T,\cube,\epsilon)$}:=\left\{\gamma\in U(\tau)\cap\MC.\gamma_0: e^{t-\epsilon}\leq\|\gamma\|_\tau\leq e^t\text{ and } {\gammabar}^{\tau}\in \cube\right\}.
\ee 

Put $\tilde q_0:=\hubmas^{-1}([\eta],\gammabar_0^\tau)$.  
Without loss of generality we assume $\gamma_0$ and $\eta$ are so that $\tilde q_0$ belongs  to our fixed fundamental domain.



\begin{lemma}\label{lem:int-horo}
Let $\delta>0$, and let $\cube\subset\ttco$ be a cube of size $\geq \delta$. 
Let $\lambda$ denote the center of $\cube$. For all $\epsilon\leq \delta$ and all large enough $t\geq0$ 
we have: 
\[
\text{$\mce\gamma_0\in\Ip(\gamma_0,t,\cube,\epsilon)\;$ if and only if $\;\PU_{\cube,\epsilon} \cap \mce\cdot a_{\mlsign t} \tmlhoro(\tilde q_0)\neq\emptyset$.}
\] 
\end{lemma}

\begin{proof}
Since $\tau$ is fixed throughout, we drop it from the subscript and superscript for the norm and the normalization. 

Suppose  $\gamma=\mce \gamma_0\in\Ip(\gamma_0,t,\cube,\epsilon)$ for some $\mce\in\MC$;
such $\mce$ is not unique, however, for any other $\mce'\in\MC$ with
$\mce \gamma_0=\mce'\gamma_0$ we have $\mce\cdot \tmlhoro(\tilde q_0)=\mce'\cdot\tmlhoro(\tilde q_0)$.
Put $\tilde q=\mce\cdot \tilde q_0$. 
Then $\mce\gamma=\Ifrak(\tilde q)$, moreover, 
\[
\mce\cdot a_{\mlsign t} \tmlhoro(\tilde q_0)=a_{\mlsign t}\tmlhoro(\tilde q).
\]
Recall that $\gammabar\in \cube$ and put $\tilde p':=\hubmas^{-1}([\eta], \gammabar)$.
Then, $\tilde p'\in \PU_\cube$;
moreover, it follows from the definition that $\Ifrak(\tilde p')=\gammabar$. Hence, $\tilde p'\in a_{t_1} \tmlhoro(\tilde q)$ where 
$t_1=\log\|\gamma\|$.

Put $s=t_1-t;$ since $\gamma\in\Ip(\gamma_0,t,\cube,\epsilon)$ we have $-\epsilon\leq s\leq 0$.
We get from the above and the definition of $\PU_{\cube,\epsilon}$ that  
$
a_{s}\tilde p'\in a_{\mlsign t} \tmlhoro(\tilde q)\cap\PU_{\cube,\epsilon}.
$ 
In particular, 
\[
\PU_{\cube,\epsilon}\cap a_{\mlsign t} \tmlhoro(\tilde q)=\PU_{\cube,\epsilon}\cap\mce\cdot a_{\mlsign t} \tmlhoro(\tilde q_0)\neq\emptyset.
\]

Conversely, suppose that for some $\mce\in\MC$ we have 
$\PU_{\cube,\epsilon}\cap\mce\cdot a_{\mlsign t} \tmlhoro(\tilde q_0)\neq\emptyset$.
Put $\gamma=\mce\gamma_0$;
we claim that $\gamma\in\Ip(\gamma_0,t,\cube,\epsilon)$.

Set $\tilde q=\mce\cdot\tilde q_0$. Then $\Ifrak(\tilde q)=\gamma$, and as above we have 
$\mce\cdot a_{\mlsign t} \tmlhoro(\tilde q_0)=a_{\mlsign t}\tmlhoro(\tilde q)$.
Let now $\lambda\in \cube$ and $-\epsilon\leq s\leq0$ be so that
\[
\hubmas^{-1}([\eta],e^{s}\lambda)\in \PU_{\cube,\epsilon}\cap a_{\mlsign t}\tmlhoro(\tilde q).
\]
Let us write $\hubmas^{-1}([\eta],e^s\lambda)=a_t\tilde q'$ where ${\tilde q}'\in \tmlhoro(\tilde q)$.
Then, we have 
\[
e^{-t}\gamma=\Ifrak(a_{t}{\tilde q}')=e^s\lambda\in e^s\cube.
\] 
This gives $\bar\gamma=\lambda$, hence, $\bar\gamma\in \cube$ and $\|\gamma\|=e^{t+s}$; 
we get $\gamma\in\Ip(\gamma_0,t,\cube,\epsilon)$ as we claimed.
\end{proof}

\subsection{Strebel differentials}\label{sec:closed-W-}

Problems related to the existence and uniqueness of Jenkins-Strebel differentials have been extensively studied.

\begin{theorem}[Cf.~\cite{Str}, Thm.~20.3]\label{thm:strabel-thm}
Let $\gamma=\sqcup_{i=1}^d\gamma_i$ be a rational multi-geodesic on $M$, 
and let $r_1,\ldots,r_d$ be positive real numbers. Then there exists a unique holomorphic quadratic differential $q$ on $M$ 
(Jenkins-Strebel differential) with the following properties.
\begin{enumerate}
\item If $\Gamma$ is the critical graph\footnote{Recall that the critical graph of a quadratic differential is the union of the compact leaves of the measured foliation induced by $q$ which contain a singularity of $q$.} of $q$, then $M\setminus\Gamma =\cup_{i=1}^d\Omega_i$, where $\Omega_i$ 
is either empty or a cylinder whose core curve is $\gamma_i$. 
\item If $\Omega_i$ is not empty, it is swept out by trajectories whose $q$ length is $r_i$.
\end{enumerate}
\end{theorem}

The following lemma will be used in the sequel.

\begin{lemma}\label{lem:horo}
Let $\gamma\in\ttc$ be rational, and let $\tilde q=\hubmas^{-1}([\eta],\gamma)\in\mathcal Q^1\Teich(\alpha)$
be a quadratic differential so that $\Ifrak(\tilde q)=\gamma;$ 
put $q:=\pi(\tilde q)$. Then 
\begin{enumerate}
\item $\mlhoro(q)\subset\Qalpha$ is a properly immersed, affine submanifold 
which carries a natural finite Borel measure $\nu$. 

\item There exists some $\epsilon_0=\epsilon_0(\tau,\eta,\|\gamma\|_\tau)>0$ so that the following holds. Let $0<\hat\epsilon<\epsilon_0$ and let
\[
\text{$K(\hat\epsilon)\index{k@$K(\epsilon)$ the $\epsilon$-thick part}=\{q:$ all saddle connections on $q$ are $\geq\hat\epsilon\}$.}
\]
Put $\mathsf D(\hat\epsilon)=\mathsf D_{\rm cusp}({\hat\epsilon})\index{D@$\mathsf D_{\rm cusp}({\epsilon})$ the $\epsilon$-thin part of $W^{\unst}(q)$}:=\mlhoro(q)\cap K({\hat\epsilon})^{\complement}$. 
There are constants $\consta\label{a:W-cusp}$ and $\constA\label{A:W-cusp}$, 
and a smooth function $0\leq\psi^{\unst}_{\hat\epsilon}\leq 1$ supported on $\mlhoro(q)$ so that 
\begin{enumerate} 
\item $\Sob(\psi^{\unst}_{\hat\epsilon})\ll {\hat\epsilon}^{-\ref{A:W-cusp}}$,
\item $\|\psi^{\unst}_{\hat\epsilon}\|_{2,\nu}\ll {\hat\epsilon}^{\ref{a:W-cusp}}$,
\item $\psi^{\unst}_{\hat\epsilon}|_{\mathsf D({\hat\epsilon})}=1$, and $\|1_{\mathsf D({\hat\epsilon})}-\psi^{\unst}_{\hat\epsilon}\|_{2,\nu}\ll {\hat\epsilon}^{\ref{a:W-cusp}}$.
\end{enumerate}
In particular, we have $\nu(\mathsf D({\hat\epsilon}))\leq {\hat\epsilon}^{\ref{a:W-cusp}}$ for all small enough ${\hat\epsilon}$.
\end{enumerate}
\end{lemma}

\begin{proof}
We first show that $\mlhoro(q)$ is a properly immersed submanifold of $\Qalpha$. 
This is equivalent to showing the following two statements.
\begin{enumerate}
\item[(i)] $\mce_1\cdot\tmlhoro(\tilde q)\cap\mce_2\cdot\tmlhoro(\tilde q)\neq\emptyset$ if and only if $\mce_1\cdot\tmlhoro(\tilde q)=\mce_2\cdot\tmlhoro(\tilde q)$.
\item[(ii)] $\bigcup_{\mce\in\MC} \mce\cdot\tmlhoro(\tilde q)\subset\mathcal Q^1\mathcal T(\alpha)$ is closed.
\end{enumerate}
Recall that $\tmlhoro(\tilde p)=\{{\tilde p}':\Ifrak({\tilde p}')=\Ifrak(\tilde p)\}$
and that $\mce\cdot \tmlhoro(\tilde p)=\tmlhoro(\mce\cdot\tilde p)$ for all $\tilde p\in\mathcal Q^1\mathcal T(\alpha)$.
These imply (i).
To see (ii), note further that the set 
\[
\bigcup_{\mce\in\MC}\mce\cdot\tmlhoro(\tilde q) 
\] 
is the set of quadratic differentials $\tilde p\in\mathcal Q^1\Teich(1,\ldots,1)$ so that $\Ifrak(\tilde p)\in\MC.\gamma$. 
Since $\gamma$ is rational, $\MC.\gamma$ is a discrete $\MC$-invariant set; (ii) follows.

Let $\gamma$ be as in the statement. Write $\gamma=\sum_ia_i\gamma_i$ where each $\gamma_i$ 
is a simple closed curve and $a_i\in\bbq$. 
By Theorem~\ref{thm:strabel-thm} we have: the locus $W^{\unst}(q)\cap\Qalpha$
is identified with a linear subspace $\mathcal W=\{(x_{ij}):\sum_j x_{ij}=r_i, x_{ij}>0\}$ in the period coordinates, where $r_1,\ldots,r_d$
are positive real numbers.
Moreover, the measure $\nu$ is the pull back of the Lebesgue measure from $\mathcal W$ to $W^{\unst}(q)$. This finishes the proof of~(1).

To see part~(2), let $\epsilon_0$ be so that $\bhubmas([\eta],\gamma)\in K(\epsilon_0)$, recall from Lemma~\ref{lem:ttco-pairing}
that $\epsilon_0$ depends only on $\tau$, $\eta$, and $\|\gamma\|_\tau$.
For any $0<\hat\epsilon<\epsilon_0$ put 
\[
\mathcal W(\hat\epsilon)=\{(x_{ij})\in\mathcal W: \text{$0< x_{ij}< \hat\epsilon$ for some $(i,j)$}\}.
\]
Using Theorem~\ref{thm:strabel-thm}, we have $W^{\unst}(q)\cap K(\hat\epsilon)^{\complement}\subset \Phi^{-1}(\mathcal W(\hat\epsilon))$.
The claims in part (2) now follow from Lemma~\ref{lem:partition-unity-AG}. Indeed apply Lemma~\ref{lem:partition-unity-AG}
with $D=\mathsf D(2\hat\epsilon)\setminus\mathsf D(\hat\epsilon/2)$, and let $\{\varphi_i\}$ be the collection of functions
obtained by that lemma. Define 
\[
\psi^{\unst}_{\hat\epsilon}(p)=\begin{cases}\sum\varphi_i(p)&\text{if $p\in W^{\unst}(q)\setminus\mathsf D(\hat\epsilon/2)$}\\ 1&\text{if $p\in \mathsf D(\hat\epsilon/2)$}\end{cases}.
\]
This function satisfies the claims.
\end{proof}



Let $\gamma_0$ and $\tilde q_0\in\mathcal Q^1\mathcal T(1,\ldots,1)$ be as in Lemma~\ref{lem:int-horo}
and put $q_0:=\pi(\tilde q_0)$.
Then by Lemma~\ref{lem:horo} we have $\mlhoro(q_0)$ is an affine submanifold of $\Qalpha$. 
We will put
\be\label{eq:volume-W+}
{\rm v}(\gamma_0)\index{v@${\rm v}(\gamma_0)$ volume of $W^{\unst}(\pi\circ\hubmas^{-1}([\eta],\gamma_0))$ for a rational multigeodesic $\gamma_0$}=\nu(W^{\unst}(q_0))
\ee
where $\nu$ is the finite measure in Lemma~\ref{lem:horo}.

Let $b>0$; this choice will be optimized later. 
Apply Lemma~\ref{lem:horo}(2) with $\hat\epsilon=10b$ and let $\mathsf D_{\rm cusp}(10b)$
be as in that lemma. Put $\mathsf D_b\index{d@ $\mathsf D_b$}:=W^{\unst}(q)\setminus\mathsf D_{\rm cusp}(10b)$.

\begin{lemma}\label{lem:part-unity-W+}
For every $b$ there exists some $N(b)\ll b^{-\constA\label{A:th-W-c}}$ so that the following holds.
There exists a collection of functions $\{\psi^{\unst}_i:0\leq i\leq N(b)\}$ with the following properties:
\begin{enumerate}
\item $\psi^{\unst}_0=\psi^{\unst}_{10b}$ where $\psi^{\unst}_{10b}$ is given by Lemma~\ref{lem:horo}(2).
\item $0\leq \psi^{\unst}_i\leq 1$ for all $i\geq0$.
\item For all $i\geq 1$, $\psi^{\unst}_i$ is supported in $\mathsf B_{b}^{\unst}(y_i)$ where $y_i\in \mathsf D_b$; 
furthermore, the multiplicity of $\{\mathsf B_{b}^{\unst}(y_i)\}$ is at most $\ref{A:part-unity-mult}$.
\item $\sum_{i=1}^{N(b)} \psi^{\unst}_i\leq 1$ and $\sum_{i=1}^{N(b)}\psi^{\unst}_i=1$ on $\cup_{i=1}^{N(b)}\mathsf B_b(y_i)$.
\end{enumerate}
Moreover, we have 
\be\label{eq:u-sob}
\Sob(\psi^{\unst}_i)\leq\ref{A:mul-u-sob} b^{-\ref{A:power-u-sob}}\text{ for all $0\leq i\leq N(b)$}
\ee
where $\constA\label{A:power-u-sob}$ is an absolute constant  
and $\constA\label{A:mul-u-sob}$ is allowed to depend on $q_0$. 
\end{lemma}

\begin{proof}
This follows from Lemma~\ref{lem:partition-unity-AG} applied with $D=\mathsf D_b$ and Lemma~\ref{lem:horo}. 
\end{proof}

Let us also fix a fundamental domain $\tilde{\mathsf D}\subset \tW^{\unst}(\tilde q_0)$\index{d@$\tilde{\mathsf D}$} which projects to $W^{\unst}(q_0)$.
For each $i\geq 1$, we let $\tilde y_i\in\tilde{\mathsf D}$ be a lift of $y_i$, see Lemma~\ref{lem:part-unity-W+}. 
Let $N'(b)$ be so that 
\be\label{eq:def-Nb'}
\text{$\mathsf B^{\unst}_b(\tilde y_i)\subset\tilde{\mathsf D}$ for all $N'(b)<i\leq N(b)$. }
\ee
For simplicity in notation, let $\mathsf B_b^{\unst}(\tilde y_0)\subset\tilde{\mathsf D}$ denote the lift of $\mathsf D_{\rm cusp}(10b)$. 
Increasing $N'(b)$, if necessary, we assume that $\mathsf B^{\unst}_b(\tilde y_i)\cap \mathsf B_b^{\unst}(\tilde y_0)=\emptyset$ for all 
$i\geq N'(b)$.

\subsection{Counting in linear sectors in $\ML$}\label{sec:linear-count-ML}
Recall from the beginning of this section that $\cube\subset\ttco$ is a box of size $\geq \delta$. 
Let $\lambda$ be the center of $\cube$, and let $\epsilon\leq \delta$. 
Let $\eta\in\ML$ be fixed as in the beginning of this section. 
We always assume $0<\delta<1/2$ and $\eta$ are so that $\hubmas^{-1}$ 
is a homeomorphism on $\{[\eta]\}\times\{e^r\cube: |r|<\delta\}$.
Recall also our notation $\PU_\cube=\hubmas^{-1}(\{[\eta]\}\times \cube)$ and 
\[
\PU_{\cube,\epsilon}=\hubmas^{-1}\bigl(\{[\eta]\}\times\{e^r\cube: -\epsilon<r\leq 0\}\bigr).
\]


Abusing the notation, we denote by $\mu_{\rm Th}(\cube)$ the measure induced
from $\mu_{\rm Th}$ on $\ttco$. 
The following lemma is a crucial step in the proof of Theorem~\ref{prop:counting}.

\begin{lemma}\label{lem:number-of-lifts}
There exist $\consta\label{a:epsilon-exp-final}$ and $\consta\label{a:b-exp-final}$ so that the following holds.
Let $t\geq 0$ and in the above notation, define
\[
\nlifts(\tilde q_0,t,\cube,\epsilon)\index{n@$\nlifts(\tilde q_0,t,\cube,\epsilon)$}:=\bigl\{\mce\cdot\tmlhoro(\tilde q_0): {\mce}\in\MC\text{ and } \PU_{\cube,\epsilon}\cap\mce\cdot a_{\mlsign t}\mlhoro(\tilde q_0)\neq\emptyset\bigr\}.
\]
Suppose $\epsilon\geq e^{-\ref{a:epsilon-exp-final}t}$, then 
\[
\#\nlifts(\tilde q_0,t,\cube,\epsilon)={\rm v}(\gamma_0)\mu_{\rm Th}(\cube)(\tfrac{1-e^{-h\epsilon}}{h})e^{ht}+O_{\tau,\gamma_0}((1-e^{-h\epsilon})e^{(h-\ref{a:b-exp-final})t}).
\]
\end{lemma}

We will prove Lemma~\ref{lem:number-of-lifts} using Proposition~\ref{prop:ncc}, more precisely Corollary~\ref{cor:nnc-notsmooth}.
In order to use those results we need to control the {\em geometry} of $\PU_{\cube,\epsilon}$.

%
%
%
%

\begin{lemma}\label{lem:linear-ml-qt-used}
The characteristic function of     
\[
\PU_{\cube,\epsilon}=\hubmas^{-1}(\{[\eta]\}\times \{e^s\cube:|s|\leq\epsilon\})
\]
belongs to $\mathcal S_{\tilde W^{\rm cs}(\tilde q_j)}(\tilde p,\epsilon)$ where $\tilde p=\hubmas^{-1}([\eta],\lambda)$.
\end{lemma}

\begin{proof}
Apply Lemma~\ref{lem:ttco-pairing} with $\tau$
and let $K=K(\tau)$ be defined as in~\eqref{eq:def-K-tt}. Then
\[
\bhubmas([\eta],\ttco)\subset K.
\]

Let $\{\mathsf B_{r_p}(p):p\in K\}$ be the covering of $K$ by period boxes given by Lemma~\ref{lem:box-linear}.
Let $\mathsf B_\cdot(q_1),\ldots,\mathsf B_{\cdot}(q_{\mathsf b'})$ be a finite subcover 
of this covering.
Consider all lifts of $\mathsf B(q_j)$ to period boxes based at lifts $\tilde q_j$ of $q_j$ in our fixed (weak) fundamental domain. 
Denote these lifts by $\mathsf B_{r_1}(\tilde q_1),\ldots,\mathsf B_{r_{\mathsf b}}(\tilde q_{\mathsf b})$ --- note that 
we only fixed a weak fundamental domain, hence there might be more than one lift, however, there is a universal bound on the number of lifts. 

For every $1\leq j\leq \mathsf b$, let $\sigma_j$ be a train track obtained 
by applying Lemma~\ref{lem:box-linear} to $\mathsf B_{r_j}(\tilde q_j)$.
Assume $\epsilon$ is smaller than the radius of $\mathsf B_{r_j}(\tilde q_j)$ for all $j$.
Write $\cube=\cup \hat \cube_{i}$ where 
\[
\hat \cube_{i}=\cube\cap U(\sigma_j).
\] 
By Lemma~\ref{lem:p-linear} each $\hat \cube_{i}$ is a piecewise linear subset of $\cube_i$. 
The claim now follows from Lemma~\ref{lem:box-linear}(4)  
if we ignore those $\hat \cube_{i}$'s which have size less than $\epsilon^N$ for some $N>1$ depending only on the dimension. 
\end{proof}



\begin{proof}[Proof of Lemma~\ref{lem:number-of-lifts}]
Recall that $\lambda$ is the center of $\cube$;
put $\tilde p=\hubmas^{-1}([\eta],\lambda)$ and $p=\pi(\tilde p)$.
Let $\tilde\phi^{\rm cs}$ be the characteristic function of $\PU_{\cube,\epsilon}\subset \tW^{\rm cs}(\tilde p)$.
Define 
\[
\phi^{\rm cs}:=\tilde\phi^{\rm cs}\circ\bigl(\pi^{-1}|_{\pi(\supp (\tilde\phi^{\rm cs}))}\bigr)
\] 
--- the push-forward of $\tilde\phi^{\rm cs}$ to $W^{\rm cs}(p)$.
Recall from Lemma~\ref{lem:linear-ml-qt-used} that $\phi^{\rm cs}\in\mathcal S_{W^{\rm cs}(p)}(p,\epsilon)$.

Recall from \S\ref{sec:notation} that $\mu$ denotes the 
$\SL(2,\reals)$-invariant probability measure on $\mathcal Q_1(1,\ldots,1)$ which is in the Lebesgue measure class. 
The measures $\mu^{\unst}_x$ and $\mu^{\stbl}_x$ are the conditional measures of $\mu$ along $W^{\unst}(x)$ and $W^{\stbl}(x)$; 
$\mu_x^{\rm cs}$ and $\mu^{\rm cu}_x$ are defined accordingly.

%

Recall also that 
$
\mu_{\rm Th}(\{e^s\cube: -\epsilon<s\leq 0\})=\tfrac{1-e^{-h\epsilon}}{h}\mu_{\rm Th}(\cube).
$ 
Therefore, we have 
\be\label{eq:nlifts-Jac-+}
\mu_{p}^{\rm cs}(\phi^{\rm cs})=\tfrac{1-e^{-h\epsilon}}{h}\mu_{{\rm Th}}(\cube).
\ee



For simplicity in notation, let us write $\PU=\PU_{\cube,\epsilon}$ and put 
\[
\nlifts=\nlifts(\tilde q_0,t,\cube,\epsilon).
\]
Let $\mce\in\MC$ be so that  
$
\PU\cap\mce\cdot a_{\mlsign t}\tmlhoro(\tilde q_0)\neq\emptyset.
$
Recall that $\{\mathsf B_i^{\unst}(\tilde y_i):0\leq i\leq N(b)\}$ cover $\tilde{\mathsf D}\subset\tW^{\unst}(\tilde q_0)$, 
see Lemma~\ref{lem:part-unity-W+} and the paragraph following that lemma;
there exists some $\mce'\in\MC$ so that $\mce'\cdot\tmlhoro(\tilde q_0)=\tmlhoro(\tilde q_0)$
and some $0\leq i\leq N(b)$ so that
\be\label{eq:reurn-local}
\PU\cap{\mce}{\mce}'\cdot a_{\mlsign t}\tpbox_{b}^{\unst}(\tilde y_i)\neq\emptyset.
\ee
Let $N'(b)$ be defined in~\eqref{eq:def-Nb'}. We claim that the following holds:
\begin{multline}
\label{eq:Oy0}\#\{\mce\cdot \tmlhoro(\tilde q_0):\text{~\eqref{eq:reurn-local} holds for some $0\leq i\leq N'(b)$ }\}\ll 
\epsilon^{-\star}b^{-\star}{\rm v}(\gamma_0)e^{(h-\ref{a:th-smooth}) t} +\\ b^\star {\rm v}(\gamma_0)e^{ht}
\end{multline}
where the implied constants depend on the genus. 

Let us assume~\eqref{eq:Oy0} and finish the proof.
Let 
\[
\nlifts'\index{n@$\nlifts'$}:=\{\mce\cdot \tW^{\unst}(\tilde q_0)\in\nlifts: \text{~\eqref{eq:reurn-local} does {\em not} hold for any $0\leq i\leq N'(b)$}\} 
\]
i.e, the contribution to $\nlifts$ coming from $N'(b)< i\leq N(b)$. We claim that
\be\label{eq:nlifts-0}
\big|\#\nlifts'-\sum_{i}\sum_y\psi^{\unst}_i(y)\big|\ll \epsilon^{-\star}b^{-\star}{\rm v}(\gamma_0)e^{(h-\ref{a:th-smooth}) t}
+b^\star {\rm v}(\gamma_0)e^{ht}
\ee
where the outer summation is over all $N'(b)<i\leq N(b)$ and the 
inner summation is over all $y\in\tpbox_{b}^{\unst}(y_i)$ so that $a_{\mlsign t}y\in\pi(\PU)$.

To see the claim,
first note that by the definition of $\nlifts'$, if $\mce\cdot \tW^{\unst}(\tilde q_0)\in\nlifts'$,
then~\eqref{eq:reurn-local} holds with some $N'(b)< i\leq N(b)$. 
Let now $\mce_1,\mce_2\in\MC$ and $N'(b)\leq i_1,i_2\leq N(b)$ be so that 
\[
\PU\cap{\mce}{\mce}_j\cdot a_{\mlsign t}\tpbox_{b}^{\unst}(\tilde y_{i_j})\neq\emptyset.
\]
Then $\mce_j\tW^{\unst}(q_0)=\tW^{\unst}(q_0)$ for $j=1,2$, see the discussion preceding~\eqref{eq:reurn-local}; hence by Corollary~\ref{cor:small-chart-trans} we have
\[
\PU\cap{\mce}{\mce}_1\cdot a_{\mlsign t}\tpbox_{b}^{\unst}(\tilde y_{i_1})=\PU\cap{\mce}{\mce}_2\cdot a_{\mlsign t}\tpbox_{b}^{\unst}(\tilde y_{i_2}).
\]
In particular, $\mce_1\tpbox_{b}^{\unst}(\tilde y_{i_1})\cap\mce_2\tpbox_{b}^{\unst}(\tilde y_{i_2})\neq\emptyset$. 
Since $\tpbox_{b}^{\unst}(\tilde y_{i_j})\subset\tilde{\mathsf D}$ for $j=1,2$ --- recall that $N'(b)\leq i_1,i_2\leq N(b)$ --- 
we get that $\mce_1=\mce_2$.
Therefore, 
\[
\PU\cap{\mce}{\mce}_1\cdot a_{\mlsign t}\tpbox_{b}^{\unst}(\tilde y_{i_1})
\] 
corresponds to points lying in the intersection 
$\tpbox_{b}^{\unst}(\tilde y_{i_1})\cap\tpbox_{b}^{\unst}(\tilde y_{i_2})$ but not in $\cup_{i=0}^{N'(b)}\mathsf B_{b}^{\unst}(\tilde y_i)$. 
Recall from Lemma~\ref{lem:part-unity-W+} that $\sum_i \psi^{\unst}=1$ on $\cup_{i=1}^{N(b)}\mathsf B_b(y_i)$, hence 
$\sum_{N'(b)<i\leq N(b)}\psi^{\unst}_i=1$ on 
$\mathsf D_b\setminus \cup_{i=1}^{N'(b)}\mathsf B_b^{\unst}(y_i)$. In particular, since $\psi_i^{\unst}\geq0$, we get that
\[
\#\nlifts'\leq \sum_{i}\sum_y\psi^{\unst}_i(y)
\] 
where the outer summation is over all $N'(b)<i\leq N(b)$ and the 
inner summation is over all $y\in\tpbox_{b}^{\unst}(y_i)$ so that $a_{\mlsign t}y\in\pi(\PU)$.
Moreover, in view of the fact that $\mathsf B^{\unst}_b(\tilde y_i)\cap \mathsf B_b^{\unst}(\tilde y_0)=\emptyset$ for all 
$i\geq N'(b)$ and using Lemma~\ref{lem:part-unity-W+}(2) and~(4), we have 
\[
\sum_{i}\sum_y\psi^{\unst}_i(y)-\#\nlifts'\ll\#\{\mce\cdot \tmlhoro(\tilde q_0):\text{~\eqref{eq:reurn-local} holds for some $1\leq i\leq N'(b)$ }\}
\]
where the implied constant depends on $\alpha$. 
The claim in~\eqref{eq:nlifts-0} thus follows in view of the estimate in~\eqref{eq:Oy0}.

Let us now investigate $\sum_{i}\sum_y\psi^{\unst}_i(y)$. 
Using the definition of $\ncc$ in~\eqref{eq:ncc-r'}, we have 
\begin{align*}
\ncc\bigl(t,\psi^{\unst}_i,\phi^{\rm cs}\bigr)&=\sum \psi^{\unst}_i(y)\phi^{\rm cs}(a_ty)\\
&=\sum\psi_i^{\unst}(y)&&\phi^{\rm cs}(a_ty)=0,1
\end{align*}
where the summations are over all $y\in\tpbox_{b}^{\unst}(y_i)$ so that $a_{\mlsign t}y\in\pi(\PU)=\supp(\phi^{\rm cs})$.
Now apply Corollary~\ref{cor:nnc-notsmooth}, see in particular~\eqref{eq:ncc-ns-final}, with $\psi^{\unst}_i$ and $\phi^{\rm cs}$, 
and get that
\be\label{eq:Nnc-applied}
\Bigl|\sum\psi_i^{\unst}(y)-\mu_{q_0}^{\unst}(\psi_i^{\unst})\mu_p^{\rm cs}(\phi^{\rm cs})e^{ht}\Bigr|\leq
\Sob(\psi_i^{\unst})e^{(h-\ref{a:ncc}) t}.
\ee
In view of~\eqref{eq:nlifts-Jac-+} and the estimate $\Sob(\psi^{\unst}_i)\leq\ref{A:mul-u-sob} b^{-\ref{A:power-u-sob}}$, see~\eqref{eq:u-sob}, we get the following from~\eqref{eq:Nnc-applied}.
\be\label{eq:nlifts-2}
\Bigl|\sum_{y}\psi^{\unst}_i(y)-\mu_{q_0}^{\unst}(\psi^{\unst}_i)\mu_{{\rm Th}}(\cube)(\tfrac{1-e^{-h\epsilon}}{h})e^{ht}\Bigr|\ll\ref{A:mul-u-sob} \epsilon^{-\star} b^{-\star}e^{(h-\ref{a:ncc}) t}.
\ee

Summing up~\eqref{eq:nlifts-2} over all $N'(b)\leq i\leq N(b)$ and using the fact that $N(b)\ll b^{-\star}$, we get that
\be\label{eq:sum-up-0}
|\textstyle\sum_i\textstyle\sum_{y}\psi^{\unst}_i(y)-\sum_{i}\mu_{q_0}^{\unst}(\psi^{\unst}_i)\mu_{{\rm Th}}(\cube)(\tfrac{1-e^{-h\epsilon}}{h})e^{ht}|\ll \ref{A:mul-u-sob} \epsilon^{-\star} b^{-\star}e^{(h-\ref{a:ncc}) t}.
\ee

We now compare $\sum_{i}\mu_{q_0}^{\unst}(\psi^{\unst}_i)$ and ${\rm v}(\gamma_0)$.
Indeed, using Lemma~\ref{lem:horo}, see also~\eqref{eq:volume-W+}, and the relationship between $\nu$ and
$\mu_{q_0}^{\unst}$ we get the following:
\be\label{eq:nlifts-thick}
(1-b^{\ref{a:W-cusp}}){\rm v}(\gamma_0)\leq {\rm v}(\gamma_0)-\nu(\mathsf D_b')\leq \sum_{i=N'(b)}^{N(b)}\mu_{q_0}^{\unst}(\psi^{\unst}_i)\leq {\rm v}(\gamma_0)
\ee
where $\mathsf D_b'=\mathsf D_{\rm cusp}(10b)\cup(\cup_{i=1}^{N'(b)}\mathsf B_b(y_i))$.
The estimate in~\eqref{eq:nlifts-thick} implies that
\begin{multline}\label{eq:sum-up-1}
\bigl|\textstyle\sum_i\textstyle\sum_{y}\psi^{\unst}_i(y)-{\rm v}(\gamma_0)\mu_{{\rm Th}}(\cube)(\tfrac{1-e^{-h\epsilon}}{h})e^{ht}\bigr|\leq b^{\ref{a:W-cusp}}{\rm v}(\gamma_0)\mu_{{\rm Th}}(\cube)(\tfrac{1-e^{-h\epsilon}}{h})e^{ht}\;\;+\\\bigl|\textstyle\sum_i\textstyle\sum_{y}\psi^{\unst}_i(y)-\sum_{i}\mu_{q_0}^{\unst}(\psi^{\unst}_i)\mu_{{\rm Th}}(\cube)(\tfrac{1-e^{-h\epsilon}}{h})e^{ht}\bigr|.
\end{multline}

We now use these estimates to get an estimate for $\#\nlifts'$. First note that
\begin{align}
\notag\bigl|\#\nlifts'-{\rm v}(\gamma_0)\mu_{{\rm Th}}(\cube)(\tfrac{1-e^{-h\epsilon}}{h})e^{ht}\bigr|&\leq\bigl|\#\nlifts'-\textstyle\sum_i\textstyle\sum_{y}\psi^{\unst}_i(y)\bigr|\quad+\\
\notag&\big|\textstyle\sum_i\textstyle\sum_{y}\psi^{\unst}_i(y)-{\rm v}(\gamma_0)\mu_{{\rm Th}}(\cube)(\tfrac{1-e^{-h\epsilon}}{h})e^{ht}\big|\\
\notag{}^{\text{\eqref{eq:nlifts-0}}\leadsto}&\ll\epsilon^{-\star}b^{-\star}{\rm v}(\gamma_0)e^{(h-\ref{a:th-smooth}) t}+b^\star {\rm v}(\gamma_0)e^{ht}\;+\\
\notag&\big|\textstyle\sum_i\textstyle\sum_{y}\psi^{\unst}_i(y)-{\rm v}(\gamma_0)\mu_{{\rm Th}}(\cube)(\tfrac{1-e^{-h\epsilon}}{h})e^{ht}\big|
\end{align}
where the implied constant depends only on the genus. 
This estimate and~\eqref{eq:sum-up-1} imply that
\begin{multline*}
\big|\#\nlifts'-{\rm v}(\gamma_0)\mu_{{\rm Th}}(\cube)(\tfrac{1-e^{-h\epsilon}}{h})e^{ht}\big|\ll\epsilon^{-\star}b^{-\star}{\rm v}(\gamma_0)e^{(h-\ref{a:th-smooth}) t}\;\;+\;\;b^\star {\rm v}(\gamma_0)e^{ht}\qquad+\\
b^{\ref{a:W-cusp}}{\rm v}(\gamma_0)\mu_{{\rm Th}}(\cube)(\tfrac{1-e^{-h\epsilon}}{h})e^{ht}+\big|\textstyle\sum_i\textstyle\sum_{y}\psi^{\unst}_i(y)-\sum_{i}\mu_{q_0}^{\unst}(\psi^{\unst}_i)\mu_{{\rm Th}}(\cube)(\tfrac{1-e^{-h\epsilon}}{h})e^{ht}\big|.
\end{multline*}

Putting this estimate and~\eqref{eq:sum-up-0} together we get that
\begin{multline}\label{eq:nlifts'-estimate}
\big|\#\nlifts'-{\rm v}(\gamma_0)\mu_{{\rm Th}}(\cube)(\tfrac{1-e^{-h\epsilon}}{h})e^{ht}\big|\ll\epsilon^{-\star}b^{-\star}{\rm v}(\gamma_0)e^{(h-\ref{a:th-smooth}) t}+b^\star {\rm v}(\gamma_0)e^{ht} \\ b^{\star}{\rm v}(\gamma_0)\mu_{{\rm Th}}(\cube)(\tfrac{1-e^{-h\epsilon}}{h})e^{ht}+ \ref{A:mul-u-sob} \epsilon^{-\star} b^{-\star}e^{(h-\ref{a:ncc}) t}.
\end{multline}


We now choose 
$\epsilon$ and $b$ of size $e^{-\star t}$ so that $\epsilon^{-\star}b^{-\star}e^{(h-\ref{a:th-smooth}) t}$ in~\eqref{eq:Oy0}
is $<e^{(h-\star)t}$ and so that $\ref{A:mul-u-sob} \epsilon^{-\star} b^{-\star}e^{-\ref{a:ncc} t}$ 
on the right side of~\eqref{eq:nlifts'-estimate} is $<(1-e^{-h\epsilon})e^{-\star t}$. 
The lemma follows from this in view of~\eqref{eq:Oy0}. 

Let us now turn to the proof of~\eqref{eq:Oy0}.
The argument is similar to the one that was used in the proof of~\eqref{eq:ncc-easy}.
For $1\leq i\leq N'(b)$, let $\hat\psi^{\unst}_i$ be so that 
$\supp(\hat\psi^{\unst}_i)\subset\mathsf B_{2b}(y_i)$, $\hat\psi^{\unst}_i|_{\mathsf B_{b}(y_i)}=1$, and $\Sob(\hat\psi^{\unst}_i)\ll b^{-\star}$, see Lemma~\ref{lem:partition-unity-AG}. Let $\hat\psi^{\unst}_0=\psi^{\unst}_0$.

Let $\varrho>0$ be small enough so that $10\varrho$-neighborhood of $\supp(\phi^{\rm cs})$ embeds in $\mathcal Q(1,\ldots,1)$, 
and let $\kappa>0$ be a constant which will be chosen later. 
In view of Lemma~\ref{lem:Margulis-prop}, we have 
\[
1_{\mathsf B_{\varrho}^{\unst}(p)}\in\mathcal S(\mathsf B_{\varrho}^{\unst}(p),\varrho /10).
\] 
Therefore, properties ($\mathcal S$-1), ($\mathcal S$-2), and ($\mathcal S$-2) hold with 
$\epsilon=0.1\varrho e^{-\kappa t}$ and $f=1_{\mathsf B_{\varrho}^{\unst}(p)}$. 
Let $\phi_1^{\unst}=\varphi_{+,0.1\varrho e^{-\kappa t}}$ for these choices. 

Similarly, using Lemma~\ref{lem:Margulis-prop}
(this time, it is applied to the function $\phi^{\rm cs}$ with $\epsilon=0.1\varrho e^{-\kappa t}$) 
we let $\phi_1^{\rm cs}=\varphi_{+,0.1\varrho e^{-\kappa t}}$.

Put $\phi_1:=\phi_1^{\unst}\phi_1^{\rm cs}$. Note that $1_{\mathsf B_{\varrho}^{\unst}(p)}\phi^{\rm cs}\leq \phi_1\leq 1_{\mathsf B_{2\varrho}^{\unst}(p)}\phi^{\rm cs}$. 
Therefore, 
\be\label{eq:mu-phi1-mup-u}
\mu_p^{\unst}\bigl(\mathsf B_{\varrho}^{\unst}(p)\bigr)\mu_p^{\rm cs}(\phi^{\rm cs})\leq \mu(\phi_1)\leq \mu_p^{\unst}\bigl(\mathsf B_{2\varrho}^{\unst}(p)\bigr)\mu_p^{\rm cs}(\phi^{\rm cs}).
\ee
Moreover, $\mu_p^{\unst}(\phi_1)\geq \mu_p^{\unst}(\mathsf B_{\varrho}^{\unst}(p))$.

Since $\hat\psi^{\unst}_i|_{\mathsf B_b(y_i)}=1$ and 
$\mu_p^{\unst}(\phi_1)\geq \mu_p^{\unst}(\mathsf B_{\varrho}^{\unst}(p))$, we have
\begin{multline*}
\#\{ \mce\cdot\mlhoro(\tilde q_0):\text{~\eqref{eq:reurn-local} holds with $0\leq i\leq N'(b)$}\}\ll \\\tfrac{e^{ht}}{\mu_p^{\unst}(\mathsf B_{\varrho}^{\unst}(p))}\sum_i\int_{\mlhoro(q_0)}\phi_1(a_{\mlsign t}y)\hat\psi^{\unst}_i(y)\dif\!\mu_{q_0}^{\unst}(y).
\end{multline*}
Moreover, by Proposition~\ref{prop:thickening-smooth} we have 
\[
\int_{\mlhoro(q_0)}\phi_1(a_{\mlsign t}y)\hat\psi^{\unst}_i(y)\dif\!\mu_{q_0}^{\unst}(y)=
\mu(\phi_1)\mu_{q_0}^{\unst}(\psi^{\unst}_i)+O(\Sob(\psi^{\unst}_i)\Sob(\phi_1)e^{(h-\ref{a:th-smooth}) t})
\]
for all $0\leq i\leq N'(b)$.

Combining these two estimates and using the fact that in view of the estimates in~\eqref{eq:mu-phi1-mup-u} we have 
$\mu(\phi_1)/\mu_p^{\unst}(\mathsf B_\varrho(p))\ll1$ we conclude that
\begin{multline}\label{eq:proof-of-eq:Oy0}
\#\{ \mce\cdot\mlhoro(\tilde q_0):\text{~\eqref{eq:reurn-local} holds with $0\leq i\leq N'(b)$}\}\ll e^{ht}\sum\mu_{q_0}^{\unst}(\psi^{\unst}_i)+\\O(\Sob(\psi^{\unst}_i)\Sob(\phi_1)e^{(h-\ref{a:th-smooth}) t})N'(b).
\end{multline}

In view of \eqref{eq:u-sob} we have $\Sob(\psi^{\unst}_i)\ll b^{-\star}{\rm v}(\gamma_0)$; moreover, $\Sob(\phi)\ll\epsilon^{-\star}$ and $N'(b)\ll N(b)\ll b^{-\star}$. Recall also from ~\eqref{eq:nlifts-thick} that $\sum_{i=0}^{N'(b)}\mu_{q_0}^{\unst}(\psi^{\unst}_i)\ll b^{\ref{a:W-cusp}}{\rm v}(\gamma_0)$.

If we now choose $\kappa$ small enough,~\eqref{eq:Oy0} follows from~\eqref{eq:proof-of-eq:Oy0} and the proof of complete. 
\end{proof}

\begin{corollary}\label{cor:exp-sum-counting}
There exist some $\consta\label{a:epsilon-exp-final-2}$ and $\consta\label{a:epsilon-exp-final-1}$ so that the following holds.
Let $t\geq0$ and let $\epsilon\geq e^{-\ref{a:epsilon-exp-final-2}t}$. Then 
\be\label{eq:exp-sum-counting}
\#\Ip(\gamma_0,t,\cube,\epsilon)={\rm v}(\gamma_0)\mu_{\rm Th}(\cube)(\tfrac{1-e^{-h\epsilon}}{h})e^{ht}+ O_{\gamma_0}((1-e^{-h\epsilon})e^{(h-\ref{a:epsilon-exp-final-1})t})
\ee
where as in~\eqref{eq:def-Ip} we have 
\[
\Ip(\gamma_0,t,\cube,\epsilon)=\{\gamma\in \MC.\gamma_0\cap \bigl([0,e^t]\cube\setminus [0,e^{t-\epsilon} ]\cube\bigr)\}.
\]
\end{corollary}

\begin{proof}
We will show this holds with $\ref{a:epsilon-exp-final-2}=\ref{a:epsilon-exp-final}/2$.
By Lemma~\ref{lem:int-horo} we have $\gamma\in\Ip(\gamma_0,e^t,\cube,\epsilon)$ if and only if
\[
\mce\cdot a_{\mlsign t} \tmlhoro(\tilde q_0)\cap \PU_{\cube,\epsilon}\neq\emptyset.
\] 
Therefore, it suffices to show that
\[
\#\nlifts(\tilde q_0,t,\cube,\epsilon)={\rm v}(\gamma_0)\mu_{\rm Th}(\cube)(\tfrac{1-e^{-h\epsilon}}{h})e^{ht}+O_{\gamma_0}((1-e^{-h\epsilon})e^{(h-\star)t}).
\]
This last statement is proved in Lemma~\ref{lem:number-of-lifts}.
\end{proof}

\begin{proof}[Proof of Theorem~\ref{prop:counting}]
Let $\epsilon\geq e^{-\ref{a:epsilon-exp-final-2}t}$, and
for every $n\geq 0$ define $t_n:=t-n\epsilon $.
Then~\eqref{eq:exp-sum-counting} applied with $t=t_n$ implies that
\begin{align*}
\#\Ip(\gamma_0,t_n,\cube,\epsilon)&={\rm v}(\gamma_0)\mu_{\rm Th}(\cube)(\tfrac{1-e^{-h\epsilon}}{h})e^{ht_n}+ O_{\gamma_0}((1-e^{-h\epsilon})e^{(h-\ref{a:epsilon-exp-final-1})t_n})\\
&={\rm v}(\gamma_0)\mu_{\rm Th}(\cube)(\tfrac{e^{-nh\epsilon}-e^{(-n-1)h\epsilon}}{h})e^{ht}+O_{\gamma_0}((1-e^{-h\epsilon})(e^{(h-\ref{a:epsilon-exp-final-1})t-(h-\ref{a:epsilon-exp-final-1}n\epsilon}).
\end{align*}
Summing these up over all $n\geq0$ so that ${t_n}\geq {\frac{h-1}{h}t}$ we get that
\[
\#\{\gamma\in \MC.\gamma_0\cap ([0,e^t]\cube\setminus [0,e^{\frac{h-1}{h}t}]\cube)\}=
{\rm v}(\gamma_0)\mu_{\rm Th}(\cube)(\tfrac{1-e^{\frac{h-1}{h}t}}{h})e^{ht}+O_{\gamma_0}(e^{(h-\star)t}).
\]
This implies the proposition --- note that by basic lattice point count in Euclidean spaces\footnote{As we remarked in the introduction, the point here is that we are counting the number of point in one $\MC$-orbit.}, 
we have the number of integral points $\gamma\in \ttc$ so that $\|\gamma\|\leq e^{\frac{h-1}{h}t}$ is $\ll e^{(h-1)t}$.
\end{proof}


\section{Proof of Theorem~\ref{thm:main}}\label{sec:proof-thm}

We are now in the position to prove Theorem~\ref{thm:main}. The proof relies on Theorem~\ref{prop:counting}.
We cover $\ML$ with finitely many train track charts $U(\tau_1),\ldots, U(\tau_{\mathsf c})$.
Using the convexity of the hyperbolic length function, we can reduce the counting problem in Theorem~\ref{thm:main}
to an orbital counting in sectors on $U(\tau_i)$, with respect to linear structure, where the length function $\ell_\vM$ 
is well approximated by the $\|\;\|_{\tau_i}$. 
Theorem~\ref{prop:counting} is then brought to bear in the study of the latter counting problem.    

Let $\vM$ be a compact surface equipped with a Riemannian metric of negative curvature.
Recall that $\ell_\vM:\ML\to\ML$ denotes the length function. 
It satisfies $\ell_\vM(t\lambda)=t\ell_\vM(\lambda)$ for any $t>0$.

Let $\tau$ be a maximal train track.
By Corollary~\ref{cor:length-lip}, $\ell_\vM$ is Lipschitz in $U(\tau)$.
Let $\lipc_\tau$\index{l@ $\lipc_\tau$ the Lipschitz constant of $\ell_\vM$ in $U(\tau)$} be the Lipschitz constant, hence 
\be\label{eq:hyp-length-lip}
|\ell_\vM(\lambda)-\ell_\vM(\lambda')|\leq \lipc_\tau\|\lambda-\lambda'\|_\tau.
\ee 
Recall that $U(\tau)$ is a cone on the polyhedron $\ttco$.

\begin{lemma}\label{lem:support-plane}
There exists a constant $\hat\lipc_\tau$, depending on $\lipc_\tau$, with the following property.
For every $\lambda,\lambda'\in\ttco$ we have 
$
|\tfrac{1}{\ell_\vM(\lambda)}-\tfrac{1}{\ell_\vM(\lambda')}|\leq \hat\lipc_\tau\delta.
$
\end{lemma}

\begin{proof}
First note that there exists some $\ell_{\vM,\tau}>1$\index{l@ $\ell_{\vM,\tau}$} so that $1/\ell_{\vM,\tau}\leq \ell_\vM(\lambda)\leq \ell_{\vM,\tau}$ for all $\lambda\in\ttco$.
The claim thus follows from~\eqref{eq:hyp-length-lip}.
\end{proof}

For any $T>0$, let $C_\vM(\tau,T)=\{\lambda\in U(\tau): \ell_\vM(\lambda)\leq T\}$\index{c@$C_\vM(\tau,T)$ the set of $\lambda\in\ttc$ with $\ell_\vM(\lambda)\leq T$}. 
To simplify the notation, we will write $C_\vM(\tau)$\index{c@$C_\vM(\tau)$ the set of $\lambda\in\ttc$ with $\ell_\vM(\lambda)\leq 1$} for $C_\vM(\tau,1)$.
Let $S_\vM(\tau)\index{s@$S_\vM(\tau$ the set of $\lambda\in\ttc$ with $\ell_\vM(\lambda)=1$}=\{\lambda\in U(\tau):\ell_\vM(\lambda)=1\}$. Then 
\[
C_\vM(\tau,T)=TC_\vM(\tau)=[0,T]S_\vM(\tau).
\]

\begin{proof}[Proof of Theorem~\ref{thm:main}]
Let $\vM$ be as above.
Let $\tau_1,\ldots,\tau_{\mathsf c}$ be finitely many maximal train tracks with the following properties. 
\begin{itemize}
\item $\ML=\cup_{i=1}^{\mathsf c} U(\tau_i)$, and 
\item $\ell_\vM: U(\tau_i)\to\bbr$ is $\lipc_i$-Lipschitz for all $1\leq i\leq \mathsf c$.
\end{itemize}  
Let $\lipc=\max\lipc_i$; increasing $\lipc$ if necessary 
we will also assume that the conclusion of Lemma~\ref{lem:support-plane} holds with $\lipc$.

Let us fix some $1\leq i\leq \mathsf c$ and write $\tau=\tau_i$; when there is no confusion we drop $\tau$
from the notation for the norm and normalization in $U(\tau)$. We will first consider the contribution coming from 
$U(\tau)$ and then will combine contributions of different $\tau_i$ for $1\leq i\leq \mathsf c$. 

In the following we will use the following upper bound estimate for the number of integral point in a Euclidean region: 
the number of lattice points in a Euclidean region is $\ll$ the volume of the 1-neighborhood of the region. 
 
Let $\gamma_0$ be a rational (multi) geodesic. 
For every $T>0$ define 
\be\label{eq:hyp-count-tau-i}
\mathcal N_\tau(\gamma_0,T)=\#\{\mce\gamma_0\in U(\tau): \ell_\vM(\mce\gamma_0)\leq T\}.
\ee

Fix some $\delta>0$; this will be optimized later and will be chosen to be of size $T^{-\star}$. 
Define 
\be\label{eq:ttco}
\ttcod\index{p@$\ttcod$ points in $\ttco$ where each coordinates is at least $2\delta$}:=\{(b_i)\in\ttco: b_i\geq 2\delta\text{ for all $i$}\}.
\ee

Cover $\ttco$ with cubes of size $\delta$ with disjoint interior.
Let $\{U_j: j\in J_\delta\}$ be the subcollection of these cubes so that $U_j\cap \ttcod\neq\emptyset$

For every $j$, let $\lambda_j\in U_j$ be the center of $U_j$. 
The number of $U_j$'s required to cover $\ttco$ is $\ll \delta^{-\constA\label{a:E-Latt-count-2}}$ for some 
$\ref{a:E-Latt-count-2}$ depending on $\tau$.


There is some $\consta\label{k:E-Latt-count-2}$, depending only on the dimension, with the following property. 
If $\delta\geq T^{-\ref{k:E-Latt-count-2}}$, then the number of integral points $\gamma\in\ttc$ with $\|\gamma\|\leq \ell_{\vM,\tau} T$ and 
\be\label{eq:ttco-ttcod}
\bar\gamma={\gamma}/{\|\gamma\|}\in\ttco\setminus\ttcod
\ee 
is $\ll \delta T^h$.

For each $j$, let $U_{j,-}$ denote the cube which has the same center $\lambda_j$ as $U_j$, but has size 
$\delta-\delta^{\constA\label{a:E-Latt-count-1}}$ where $\ref{a:E-Latt-count-1}=\ref{a:E-Latt-count-2}+1$.

Then, if $\delta^{\ref{a:E-Latt-count-1}}\geq T^{-\ref{k:E-Latt-count-2}}$, the number of integral points 
$\gamma\in\ttc$ with $\|\gamma\|\leq \ell_{\vM,\tau} T$ and 
\be\label{eq:Uj-Uj-}
\bar\gamma\in\bigcup_j U_j\setminus U_{j,-}
\ee
is $\ll \delta^{-\ref{a:E-Latt-count-2}}\delta^{\ref{a:E-Latt-count-1}} T^h\ll \delta T^h$.

Altogether, we have: if $\delta^{\ref{a:E-Latt-count-1}}\geq T^{-\ref{k:E-Latt-count-2}}$ , then   
\be\label{eq:number-exceptional}
\#\{\gamma\in \MC.\gamma_0\cap U(\tau): \ell_\vM(\gamma)\leq T, \bar\gamma\text{ satisfies~\eqref{eq:ttco-ttcod} or~\eqref{eq:Uj-Uj-}}\}\ll \delta T^h.
\ee

We now find an estimate for 
\[
\#\{\gamma\in \MC.\gamma_0\cap C_\vM(\tau,T):\bar\gamma\in \cup U_{j,-}\}.
\]
Put $U_{j,-,+}=\{\tfrac{\lambda}{\ell_\vM(\lambda_j)-\lipc\delta}:\lambda\in U_{j,-}\}$ and 
$U_{j,-,-}=\{\tfrac{\lambda}{\ell_\vM(\lambda_j)+\lipc\delta}:\lambda\in U_{j,-}\}$. 
Then it follows from~\eqref{eq:hyp-length-lip} that
\[
[0,1]U_{j,-,-}\subset \{\lambda\in C_\vM(1,\tau):\bar\lambda\in U_{j,-}\}\subset [0,1]U_{j,-,+}
\]
Therefore, applying Theorem~\ref{prop:counting}, with $U=U_{j,-,\pm}$, we get that
\begin{multline*}
\tfrac{{\rm v}(\gamma_0)\mu_{\rm Th}(U_{j,-})}{h(\ell_\vM(\lambda_j)+\lipc\delta)^h}T^h+ O_{\tau,\gamma_0}(T^{h-\ref{a:final-U-exp-prop}})\leq\\ 
\#\{\gamma\in\MC.\gamma_0: \gamma\in C_\vM(\tau,T),\bar\gamma\in U_{j,-}\}\leq 
\tfrac{{\rm v}(\gamma_0)\mu_{\rm Th}(U_{j,-})}{h(\ell_\vM(\lambda_j)-\lipc\delta)^h}T^h+ O_{\tau,\gamma_0}(T^{h-\ref{a:final-U-exp-prop}});
\end{multline*}
this estimate implies that 
\begin{multline}\label{eq:hyp-length-j-2}
\#\{\gamma\in\MC.\gamma_0: \gamma\in C_\vM(\tau,T),\bar\gamma\in U_{j,-}\}=\\
\tfrac{{\rm v}(\gamma_0)\mu_{\rm Th}(U_{j,-})}{h(\ell_\vM(\lambda_j))^h}T^h+ O_{\tau,\gamma_0}(\delta \mu_{\rm Th}(U_{j,-}) T^h+T^{h-\ref{a:final-U-exp-prop}}).
\end{multline}

Let us put $S_{\vM}(\tau,j)=\{\lambda\in S_\vM(\tau): \bar\lambda\in U_{j,-}\}$. Then by Lemma~\ref{lem:support-plane} we have 
\[
\mu_{\rm Th}([0,1]S_{\vM}(\tau,j))=\int_{U_{j,-}}\tfrac{1}{h\ell_\vM(\lambda)^h}\operatorname{d}\!\mu_{\rm Th}=\frac{\mu_{\rm Th}(U_{j,-})}{h(\ell_\vM(\lambda_j))^h}+O(\delta)\mu_{\rm Th}(U_{j,-}).
\]
This observation together with~\eqref{eq:hyp-length-j-2} gives that 
\begin{multline}\label{eq:hyp-length-j}
\#\{\gamma\in\MC.\gamma_0: \gamma\in C_\vM(\tau,T),\bar\gamma\in U_{j,-}\}=\\
{\rm v}(\gamma_0)\mu_{\rm Th}([0,1]S_{\vM}(\tau,j))T^h+ O_{\tau,\gamma_0}(\delta \mu_{\rm Th}(U_{j,-})T^h+T^{h-\ref{a:final-U-exp-prop}}).
\end{multline}
Recall also that $\ell_\vM^{\pm1}$ is bounded on $\ttco$; we have 
$\sum\mu_{\rm Th}([0,1]S_{\vM}(\tau,j))=\mu_{\rm Th}([0,1]S_\vM(\tau))+O(\delta^\star)$.
Hence, summing~\eqref{eq:hyp-length-j} over all $j$'s we get
\begin{multline}\label{eq:hyp-length-tau-delta}
\#\{\gamma\in\MC.\gamma_0: \gamma\in C_\vM(\tau,T),\bar\gamma\in \cup U_{j,-}\}=\\
{\rm v}(\gamma_0)\mu_{\rm Th}([0,1]S_{\vM}(\tau))T^h+ O_{\tau,\gamma_0}(\delta^\star T^h+\delta^{-\ref{a:E-Latt-count-2}} T^{h-\ref{a:final-U-exp-prop}}).
\end{multline}
Now choose $\delta=T^\star$ so that 
$\delta^\star T^h+\delta^{-\ref{a:E-Latt-count-2}} T^{h-\ref{a:final-U-exp-prop}}=T^{h-\consta\label{a:final-exp-hyp-tau}}$. 
Then we get from~\eqref{eq:hyp-length-tau-delta} and~\eqref{eq:number-exceptional} that
\be\label{eq:hyp-length-tau}
\#\{\gamma\in\MC.\gamma_0: \gamma\in C_\vM(\tau,T)\}={\rm v}(\gamma_0)\mu_{\rm Th}([0,1]S_{\vM}(\tau))T^h+O(T^{h-\ref{a:final-exp-hyp-tau}}).
\ee
This conclude the contribution arising from a single train track chart $U(\tau)$. 

Recall now that the regions in $U(\tau_i)$ which are carried by other $U(\tau_{i'})$
are finite sided polyhedra, see Lemma~\ref{lem:p-linear}. We may thus find disjoint finite sided polyhedra $\cube_i\subset P(\tau_i)$
to the $\cup\bbr^+.\,\cube_i=\ML$.    
Repeating the above argument for each $\cube_i$, the theorem follows from the estimate in~\eqref{eq:hyp-length-tau}.  
\end{proof}

%


We conclude with the following which are of independent interest. 
Let $\Gamma\subset\MC$ be a finite index subgroup and let $\tau$ be a maximal train track.
Define 
\[
\Nip_{\Gamma,\tau}(\gamma_0,T)\index{n@ $\Nip_{\Gamma,\tau}(\gamma_0,T)$}:=\left\{\gamma\in\Gamma.\gamma_0\cap U(\tau): \|\gamma\|_\tau\leq T\right\}.
\]

\begin{theorem}\label{thm:ttc-coutning}
There exists some $\consta\label{a:tc-final-Gamma}=\ref{a:tc-final-Gamma}(\Gamma)$ so that the following holds.
For every rational multi curve $\gamma_0\in\ttc$, there exists some constant $c_{\Gamma,\tau}(\gamma_0)$ so that
\[
\#\Nip_{\Gamma,\tau}(\gamma_0,T)=c_{\Gamma,\tau}(\gamma_0)T^{6g-6}+O_{\gamma_0,\tau,\Gamma}(T^{6g-6-\ref{a:tc-final-Gamma}})
\]
\end{theorem}

\begin{proof}
The argument is similar to our argument in the proof of Theorem~\ref{thm:main-MLS}.
Recall that we normalized the Masur-Veech measure to be a probability measure on $\mathcal Q_1(1,\ldots,1)$.
Let $\mu_\Gamma$ denote the lift of the Masur-Veech measure to $\mathcal Q^1\mathcal T(1,\ldots,1)/\Gamma$, then 
$\mu_\Gamma(\mathcal Q^1\mathcal T(1,\ldots,1)/\Gamma)=[\MC:\Gamma]$.

Similar to~\eqref{eq:volume-W+}, define ${\rm v}_\Gamma(\gamma_0)$ to be the measure of the lift of $W^{\unst}(q_0)$  
to $\mathcal Q^1\mathcal T(1,\ldots,1)/\Gamma$ where $\mathfrak{I}(q_0)=\gamma_0$.

Now, by virtue of Theorem~\ref{prop:counting}, we have
\[
\#\{\gamma\in\Gamma.\gamma_0\cap\ttc: \|\gamma\|_\tau\leq T\}={\rm v}'_\Gamma(\gamma_0)\mu_{\rm Th}([0,1]U(\tau))T^h+O_{\gamma_0,\tau,\Gamma}(T^{h-\ref{a:tc-final-Gamma}})
\] 
where ${\rm v}'_\Gamma(\gamma_0)={\rm v}_\Gamma(\gamma_0)/[\MC:\Gamma]$ and ${\rm v}_\Gamma(\gamma_0)$ is as above.

The exponent $\ref{a:tc-final-Gamma}$ depends on the exponential mixing rate for the Teichm\"{u}ller geodesic flow on $(\mathcal Q^1\mathcal T(1,\ldots,1)/\Gamma,\mu_\Gamma)$.  
\end{proof}

Let $\Gamma\subset\MC$ be a finite index subgroup.
Given a rational multi-geodesics $\gamma_0$ on $\vM$ define 
\[
s_{\vM,\Gamma}(\gamma_0, T)\index{s@$s_{\vM,\Gamma}(\gamma_0, L)$}:=\#\{\gamma\in\Gamma.\gamma_0:\ell_\vM(\gamma)\leq T\}
\]
We also have the following generalization of Theorem~\ref{thm:main}.

\begin{theorem}\label{thm:main-Gamma}
There exists some $\consta\label{k:final-Gamma}=\ref{k:final-Gamma}(\Gamma)>0$, dependence on $\Gamma$ is related to the exponential mixing rate for the Teichm\"{u}ller geodesic flow on $\mathcal Q^1\mathcal T(1,\ldots,1)/\Gamma$, and some $c=c(\gamma_0, \vM, \Gamma)$ so that the following holds.  
\[
s_{\vM,\Gamma}(\gamma_0,T)=c\,T^{6g-6}+O_{\gamma_0, \vM,\Gamma}(T^{6g-6-\ref{k:final-Gamma}})
\]
\end{theorem}

\begin{proof}
Similar to the discussion in the proof of Theorem~\ref{thm:ttc-coutning}, the proof of Theorem~\ref{thm:main} applies mutatis mutandis to $s_{\vM,\Gamma}(\gamma_0, T)$. 
\end{proof}

\printindex

%

\end{document}